\documentclass[graybox]{svmult}

\usepackage{type1cm}        
\usepackage{makeidx}         
\usepackage{graphicx}        
\usepackage{multicol}        
\usepackage[bottom]{footmisc}

\usepackage{newtxtext}       %
\usepackage[varvw]{newtxmath}       

\usepackage{hyperref}
\usepackage{cprotect}

\renewcommand{\le}{\leqslant}
\renewcommand{\ge}{\geqslant}

\newcommand{\R}{\mathbb{R}}

\newcommand{\ff}{\mathcal{F}}
\newcommand{\s}{\mathcal{S}}

\newcommand{\aaa}{\mathcal{A}}
\newcommand{\bb}{\mathcal{B}}

\renewcommand{\E}{\mathsf{E}}

\newcommand{\m}{\mathcal}

\newcommand{\T}{\mathcal{T}}
\newcommand{\W}{\mathcal{W}}

\newtheorem{thm}{Theorem}

\newtheorem{lem}[thm]{Lemma}
\newtheorem{cor}[thm]{Corollary}
\newtheorem{prop}[thm]{Proposition}
\newtheorem{conj}{Conjecture}
\newtheorem{obs}[thm]{Observation}
\newtheorem{cla}[thm]{Claim}
\newtheorem{defn}[thm]{Definition}

\title{Delta-system method: a survey}
\author{Andrey Kupavskii}
\institute{Andrey Kupavskii\at Moscow Institute of Physics and Technology, St. Petersburg State University, Innopolis University, Russia \email{kupavskii@ya.ru}}

\date{}

\begin{document}
\maketitle

\section{Introduction}
In 1960 Erd\H os and Rado \cite{ER} published a paper that, in retrospect, became one of the most influential papers in extremal set theory. (It was also probably the second paper on the subject, after the famous paper of Sperner.) They proved a result of Ramsey theoretic flavour, stating that in any sufficiently large family of sets of bounded size there is a homogeneous substructure, called a {\it $\Delta$-system} (also known under the name of a {\it sunflower}). For many qualitative results in Discrete Mathematics and Theoretical Computer Science, this has become a very powerful tool to analyze complex set families. Extremal set theory flourished in the 1970's--80's, and many exciting developments happened then. One of them was the development of the  {\em $\Delta$-system method} in the works of Frankl and F\"uredi. In this survey, we try to give a concise picture of this method starting from its early stages and to the modern day. We also tried to present the proofs of most of the key results. On top of this, we survey the literature on the problems that the Delta-systems was applied to. 

We should also note that two other excellent sources on the Delta-system method are the survey of F\"uredi \cite{Fur91} and the book of Frankl and Tokushige \cite{FT}. The content of this survey turned out to have quite significant intersections with these two sources. We decided to leave it as it is, partly because we aimed for completeness, and partly because presenting proofs allowed us to elaborate on some of the aspects of the method.
\tableofcontents

\subsection{Notation} We use notation, standard for extremal set theory: for non-negative integers $n,a,b$, put  $[n]=\{1,\ldots, n\}$ and, more generally, $[a,b]=\{a,a+1,\ldots, b\}$; for a set $X$, let $2^{X}$ (${X\choose k}$, ${X\choose \le k}$) stand for  the collection of all subsets of $X$ (all $k$-element subsets of $X$, all $\le k$-element subsets of $X$); a {\it family} is simply a collection of subsets of the (not necessarily specified) {\it ground set}; a {\it matching} is a family of sets that are pairwise disjoint. We typically denote elements of the ground set by lower case letters ($x,y,v$ etc.), sets (subsets of the ground set) by capital letters ($A,B,F$ etc.), and families by calligraphic letters ($\mathcal A, \m B, \ff$ etc.). For a family $\ff$, we use $\cup \ff$ as a shorthand for $\cup_{F\in\ff}F$, and similarly for $\cap \ff$. Given a family $\ff$ and an element $x$, the {\it degree of $x$} is the number of sets from $\ff$ containing $x$.

We use the following standard notation for families $\ff, \m S$ and a set $X$.
\begin{align*}
  \ff(X) &= \big\{A\setminus X: A\in \ff, X\subset A\big\}, \\
  \ff[X] &= \big\{A: A\in \ff, X\subset A\big\}, \\
  \ff[\m S] &= \bigcup_{X\in \m S}\ff[X].
\end{align*}
We use $\ff(x)$ as a shorthand for $\ff(\{x\})$. We use $O_{k,s}(n)$ to denote a function of the form $g(k,s)\cdot O(n)$ for some $g(k,s)$, and similarly for $\Omega_k(n)$ etc. Notation related to F\"uredi's structural theorem is given in Section~\ref{sec31}. For a family $\m G$ and integer $m$, the {\it shadow on level $m$} is defined as follows: $\partial^{(m)}\m G:= \cup_{G\in \m G}{G\choose m}$.

The key extremal functions that we study are $m(n,k,L)$, the largest size of family in  ${[n]\choose k}$ with pairwise intersections in $L$ (see the beginning of Section~\ref{sec2}); $f(n,k,\ell,s)$, the largest family in ${[n]\choose k}$ with no sunflower with $s$ petals and core of size $\ell$ (see Section~\ref{sec41}); $sim(n,k,d)$, $ssim(n,k,d)$, $sim_\ell(n,k,d)$, $ssim_\ell(n,k,d)$, the largest sizes of families without different simplices and special simplices (see Section~\ref{sec43});  more generally, $ex_k([n],\m H)$, the largest family in ${[n]\choose k}$ that does not contain a fixed family $\m H$. Sometimes we write {\it avoids $\m H$} or {\it is $\m H$-free} to denote the same thing.

\subsection{$\Delta$-systems and the result of Erd\H os and Rado} A family $\aaa = \{A_1,\ldots, A_s\}$ of sets is a {\em $\Delta(s)$-system}, or a {\it sunflower} if $A_i\cap A_j = \bigcap_{i\in [s]}A_i$ for all $i\ne j\in [s]$. If we want to specify the size, the we say that it is a $\Delta(s)$-system or an $s$-sunflower. The set $\bigcap_{i\in [s]}A_i$ is called the {\em core} (or {\em kernel}) of the $\Delta$-system. Erd\H os and Rado introduced $\Delta$-systems in 1960 in the paper \cite{ER} and proved the following result.

\begin{thm}[Erd\H os and Rado \cite{ER}]\label{thmer}
If $\ff$ is a family of sets of size $\le k$ and $|\ff|> k!(s-1)^k$, then $\ff$ contains a $\Delta(s)$-system.
\end{thm}
\begin{proof}
  The proof is by induction on $k$. The result is evident for $k=1$. Assuming that the result holds for $k-1$, let us prove it for $k$. Take the largest matching $M_1,\ldots, M_{\ell}\in \ff$. If $\ell\ge s$, then $\{M_1,\ldots, M_s\}$ form a $\Delta(s)$-system, and thus the theorem is proved. If $\ell\le s-1$, then $M:=\bigcup_{i\in[\ell]} M_i$ has size at most $(s-1)k$. By maximality of the chosen matching, any set from $\ff$ intersects $M$. Thus, there is an element $x\in M$  such that
  $$\ff(x):=\{F\setminus \{x\}: x\in F, F\in \ff\}$$
  has size at least $\frac{|\ff|}{(s-1)k}>(k-1)!(s-1)^{k-1}$. The sets in $\ff(x)$ have size at most $(k-1)$, and $|\ff(x)|>(k-1)!(s-1)^{k-1}$. Thus we may apply the inductive assumption to $\ff(x)$ and find a $\Delta(s)$-system in $\ff(x)$. If we append $x$ to every set in that $\Delta(s)$-system, we get a $\Delta(s)$-system in $\ff$, as claimed.
\end{proof}

Let us denote
\begin{multline}\label{eqphi}\phi(k,s):=\max\big\{|\ff|: \ff \text { consists of sets of size }\le k\\ \text{ and }\ff \text{ contains no }\Delta(s+1)-\text{system}\big\}.\end{multline}
In these terms, Theorem~\ref{thmer} states that $\phi(k,s)\le k!s^k$. Also note that the maximum is attained on a family of sets of size $k$: we can always complement sets of size $<k$ by new elements of the ground set that are all pairwise distinct. Clearly, this will not create new $\Delta$-systems. Abbot, Hanson and Sauer \cite{AHS} in 1972, and then Spencer \cite{S} in 1977 improved upper bounds on $\phi(k,s)$. The result of Spencer states that for any fixed $s$ and $\epsilon>0$ there exists $C$ such that $\phi(k,s)\le Ck!(1+\epsilon)^k$. Kostochka \cite{Kost} improved this bound for $\Delta(s)$-free systems to $\phi(k,s)\le C(s,\alpha)k! \Big(\frac{(\log\log\log k)^2}{\alpha \log\log k}\Big)^k$ for any $\alpha>1$ and a certain function $C(s,\alpha)$. For fixed $k$ and large $s$ Kostochka, R\"odl and Talysheva \cite{KRT} showed that $\phi(k,s) = (1+o(1))k^{s}$.

One natural example for a $\Delta(s+1)$-free $k$-uniform family is a $k$-partite hypergraph with parts of size $s$. It has $s^k$ sets. There are, however, examples that are exponentially bigger. The authors of \cite{AHS} showed that $\phi(2,s)=s(s+1)$ for even $s$ and $s^2+\frac{s-1}2$ for odd $s$.  The following observation allows us to extend the graph bound to higher uniformities. \begin{obs}We have $\phi(a+b,s)\ge \phi(a,s)\phi(b,s)$.\end{obs}
\begin{proof} Let $\aaa$ be a family of $a$-element sets with no $\Delta(s)$-system and of size $\phi(a,s)$. Define $\bb$, $|\bb| = \phi(b,s)$ similarly. Assume that $\aaa$ and $\bb$ live on disjoint ground sets $X$ and $Y$, respectively. Consider a family $\m C:=\{A\sqcup B: A\in \m A, B\in \m B\}$. It is easy to see that $|\m C| = \phi(a,s)+\phi(b,s)$ and that it is a $\Delta(s+1)$-system free.
\end{proof}
Abbott, Hanson and Sauer also showed a lower bound $\phi(k,3)\ge 10^{(k/2)-c\log k}$ with some positive constant $c$, which is exponentially better than $3^k$ which follows from the construction above. They used a construction of a $3$-uniform family of size $10$ and with no $\Delta(3)$-system, and then leveraged it to any uniformity using an iterated product construction, which gives a recursion $\psi(ab)\ge \psi(a)\psi(b)^a$, where $\psi(a)$ is the size of their iterated construction of uniformity $a$.
The celebrated conjecture of Erd\H os and Rado, mentioned in their paper \cite{ER}, suggests that we cannot improve the lower bound given by a complete $k$-partite hypergraph too much further. It states that the dependence on $k$ is exponential: $\phi(k,s)\le (Cs)^k$ for some absolute constant $C$.

For a more detailed survey on the Erd\H os--Rado question and related questions (weak $\Delta$-systems, as well as the Erd\H os--Sz\'emeredi question on $\Delta$-systems for a restricted ground set), we refer to a survey of Kostochka \cite{Kost00}.

\subsection{How $\Delta$-systems changed to sunflowers?}
The original name of the object from the paper of Erd\H os and Rado is a {\it $\Delta$-system}, and it was called this way for the next 20 years.

The first occurrence of the name {\it sunflower} I could trace is in the name of the paper of Deza and Frankl \cite{DF} from 1981 ``Every large set of equidistant $(0,+1,-1)$-vectors form a sunflower,'' which can be seen as a follow-up to the paper of Deza \cite{D0} that we discuss below. Funnily, the word ``sunflower'' does not appear in the body of the paper, only in the title. Moreover, the conclusion of the main theorem of that paper guarantees a structure which is not exactly a sunflower, but a (one possible) $(-1,0,1)$-generalization of it. Peter Frankl (private communication) confirmed that he came up with the term while working with Michel Deza on the paper \cite{DF}.

The second occurrence of the term I could trace (and the first one where it is actually used in the paper and gets its present meaning) is the influential paper of Alon, Frankl and Lov\'asz \cite{AFL} from 1985 on the chromatic number of Kneser hypergraphs. Two years later, in 1987, Alon and Boppana \cite{AB} write another influential paper ``Monotone circuit complexity of Boolean functions'', where they use structures akin to sunflowers for proving certain complexity lower bounds. Their result extended a breakthrough result of Razborov \cite{Raz}, who also used similar structures in the argument. My guess is that it is at this point, when the term entered TCS community, that it solidified. Later, it appeared in different other complexity lower bounds papers, in Jukna's book \cite{Juk}, and ultimately took over the $\Delta$-system terminology.

In combinatorics community, $\Delta$-system terminology prevailed for a while. But the object went under different names, notably, it was called an {\it $s$-star} in some of the most important papers on the subject of the $\Delta$-system method. This term was introduced by Chung and, in particular, used by F\"uredi in his paper \cite{Fu1} and by Frankl and F\"uredi in \cite{FF1}. Peter Frankl (private communication) pointed out that the $s$-star term came as a generalization of graph stars, which are $2$-uniform $\Delta$-systems with kernel of size $1$. In another influential paper \cite{FF2}, however, Frankl and F\"uredi used the sunflower terminology.

\subsection{Recent progress on the Erd\H os--Rado conjecture}

A breakthrough in the study of the Erd\H os--Rado Delta-system question came several years ago in a work of Alweiss, Lovett, Wu and Zhang \cite{Alw}. They managed to reduce the dependence on $k$ in  Theorem~\ref{thmer} from $k^k$ to essentially $(\log k)^k$. It spurred a lot of research. Several follow-up papers slightly improved upon the bound and presented the proof quite differently. Rao \cite{Rao} gave a code-theoretic proof with a better bound; Bell, Chueluecha and Warnke \cite{BCW} improved upon some of the arguments of the original proof, which lead to the  bound $\phi(s,k)\le (Cs\log k)^k$;  Tao~\cite{Tao} gave a proof based on entropy, which, however, contained a mistake; Lu~\cite{Hu} and then Stoeckl \cite{Sto} gave another entropy proof, which gave the bound $\phi(s,k)\le (64s\log k)^k$; Mossel, Niles-Weed, Sun, and Zadik \cite{MNSZ} gave a second moment proof; finally, Rao \cite{Rao2} gave the same bound with a very short combinatorial proof in the spirit of the original paper \cite{Alw}.

The proof is based on a result, which is called the spread lemma. Given a real number $r>1$, we say that a family $\ff$ of sets is {\it $r$-spread} if for each set $X$ we have $|\ff(X)|< r^{-|X|}|\ff|$. We say that $W$ is a {\it $p$-random subset} of $[n]$ if each element of $[n]$ is included in $W$ with probability $p$ and independently of others.

\begin{thm}[The spread lemma, \cite{Alw}, a sharpening due to \cite{Sto}]\label{thmtao}
  If for some $n,k,r\ge 1$ a family $\ff\subset {[n]\choose \le k}$ is $r$-spread and $W$ is a $(\beta\delta)$-random subset of $[n]$, then $$\Pr[\exists F\in \ff\ :\ F\subset W]> 1-\Big(\frac 2{\log_2(r\delta)} \Big)^\beta k.$$
\end{thm}
In order to prove the aforementioned bound on $\phi(k,s)$, one needs, first, to find a spread subfamily $\ff(X)$ in a large family $\ff$. (Choose $X$ to be a maximal set that violates the $r$-spreadness condition.) Then, using the spread lemma, one sees that, if the parameters are chosen right (say, $r=32s \log_2(2k), \delta = 1/(2s \log_2(2k)) ,\beta = \log_2 (2k)$), a random $\frac 1{2s}$-subset contains a set from $\ff(X)$ with probability at least $1/2$. Then, in a random $2s$-coloring there are in expectation at least $s$ colors that contain a set from the family. These sets form a matching in $\ff(X)$, which gives a $\Delta(s)$-system with kernel $X$ in $\ff$.

The paper of Rao \cite{Rao2} gives a broad view on different applications of sunflowers.

\subsection{Structure of the survey}
Roughly speaking, the survey consists of three parts. In the first part (Sections~\ref{sec2}-\ref{sec4}), we discuss the development of the method in the 1970's and 1980's. In the second part (Sections~\ref{sec6}-\ref{sec7}), we discuss aspects of the method and some relevant recent developments. In the third part (Section~\ref{sec8}), I tried to make a comprehensive survey of the literature that deals with the problems from the earlier parts.

Section~\ref{sec2} contains the early manifestations of the method. It starts with the Erd\H os--Ko--Rado theorem, then it gives the result of Deza on $\Delta$-systems and families with one intersection size, and then presents the theorem of Deza, Erd\H os and Frankl on $(n,k,L)$-systems. The latter is considered to be the starting point of the method. Then we discuss the development of the method  by Peter Frankl up until 1983. One thing that I realized while going through the literature is that Peter Frankl developed a (lesser known) approach that can be considered as an `intermediate' approach between the earlier `base' variant of the method and the `homogeneous families' variant.

Section~\ref{sec3} is devoted to F\"uredi's structural theorem on homogeneous families, which underlies further development of the method, and its early applications. In particular, we present the results of F\"uredi on families with three intersection sizes.

In Section~\ref{sec33}, we discuss the `Forbidding one intersection' paper of Frankl and F\"uredi, which to a large extent shaped the $\Delta$-system method. The result of F\"uredi meets the result of Frankl and Katona on families with many intersection conditions, and then they develop the shadow/stability machinery to get exact results.

Section~\ref{sec4} is devoted to the paper `Exact solutions of Tur\'an-type problems', in which Frankl and F\"uredi applied the method to a variety of Tur\'an-type questions. Although most of the key $\Delta$-system-related ideas were already present in the previous paper, here they harvest the fruits of their previous findings, but also delineate the scope and the limitations of the method in terms of classes of hypergraphs that it works for.

In Section~\ref{sec6}, we discuss stability and supersaturation that could be achieved using Delta-systems. Stability via Delta-systems is an aspect that is often overlooked in the literature, and we go through the proof of Frankl and F\"uredi \cite{FF1} and derive an efficient stability result. We present some further ideas on stability via Delta-systems for harder problems. We also discuss a recent result of Jiang and Longbrake that gives a supersaturation result for $2$-reducible trees via $\Delta$-systems.

In Section~\ref{secbase}, we discuss different approaches to base constructions. This is a key part of the earlier variant of the method, and can be used in conjunction with other approaches. We go over different approaches that appeared in the literature and  present the peeling/simplification procedure, which is an important part of the spread approximation method.

In Section~\ref{sec7}, we briefly discuss other general methods that were developed recently to tackle T\'uran-type problems: junta approximations, spread approximations and hypercontractivity. We do not go in depth on the methods, but rather underline the different regimes, in which these methods work. In particular, we show that, as it stands now, each of the $\Delta$-system method, hypercontractivity method and spread approximation method allows to solve Tur\'an-type problems in incompatible regimes of parameters.

In Section~\ref{sec8}, we survey literature concerning the topics of the previous sections, and especially further studies on the topics of the papers \cite{DEF}, \cite{FF1}, \cite{FF2}. We cover forbidden one intersection problem; $(n,k,L)$-systems; forbidden sunflowers and unavoidable hypergraphs; forbidden simplices and clusters; forbidden hypertrees and expanded hypergraphs.

\section{Early days of the Delta-system method}\label{sec2}
\subsection{$(n,k,L)$-systems of Deza, Erd\H os and Frankl, the precursor of the Delta-system method and the proof of the Erd\H os--Ko--Rado theorem for large $n$}\label{sec21} $\Delta$-system method was introduced in the paper of Deza, Erd\H os and Frankl \cite{DEF}. They introduced and studied the following question, along with its several generalizations. Given positive integers $n,k$ and a set $L\subset [0,k-1]$, we say that $\ff\subset {[n]\choose k}$ is an {\it $(n,k,L)$-system} if for any $A\ne B\in \ff$ we have $|A\cap B|\in L$. The question is as follows: \begin{quote} what is the largest size $m(n,k,L)$ of an $(n,k,L)$-system, for different values of $n,k, L$?\footnote{In the paper \cite{FF1}, Frankl and F\"uredi attribute  the idea of using Delta-systems in this context to Deza.}\end{quote}
Before discussing the result of Deza, Erd\H os and Frankl and the $\Delta$-system method, we want to put the result and the method into context. Another cornerstone result (and probably the third paper on the topic of extremal set theory) in extremal set theory is the Erd\H os--Ko--Rado theorem \cite{EKR}. It determines for all $n,k$ the size of the largest $(n,k,\{1,2,\ldots, k-1\})$-system, and moreover, the size of the largest  $(n, k, \{t,t+1,\ldots, k-1\})$-system for any $t\in [1,k-1]$, provided $n\ge n_0(k)$. (Such families are called {\it intersecting} and {\it $t$-intersecting}, respectively.) It is instructive to discuss the proof of the EKR theorem.

\begin{thm}[Erd\H os, Ko and Rado \cite{EKR}] Let $n\ge k\ge t$ be positive integers and suppose that $\ff\subset {[n]\choose k}$ is $t$-intersecting, that is, $|A\cap B|\ge t$ for any $A,B\in \ff$. Then $|\ff|\le {n-t\choose k-t}$ for $n>n_0(k)$.
\end{thm}
Here, we sketch a variant of their proof. In the proof of Erd\H os, Ko and Rado, $n_0(k)$ is roughly ${k\choose t}^3$.
\begin{proof}
Take a maximal $t$-intersecting family $\ff$. Maximality implies that there are $A,B\in \ff$ with $|A\cap B|=t$. There are two cases to consider. If $|\bigcap_{A\in \ff} A|\ge t$, then $|\ff|\le {n-t\choose k-t}$ and we are done. If $|\bigcap_{A\in \ff} A|< t$ then take two sets $A,B\in\ff$ such that $|A\cap B|=t$ and add a third set $C$ such that $C\not\supset A\cap B$ (we leave the verification of the existence of such $A,B,C$ as an exercise). Then $|A\cap B\cap C|<t$. Any set $F\in \ff$ intersects each of $A,B,C$ in $\ge t$ elements, and thus it must intersect $A\cup B\cup C$ in at least $t+1$ elements. Therefore, $|\ff|\le {|A\cup B\cup C|\choose t+1}{n-t-1\choose k-t-1}\le {3k\choose t+1}{n-t-1\choose k-t-1}$. For $n>n_0(k)$ this is smaller than ${n-t\choose k-t}$.
\end{proof}

Summarizing, the `trivial $t$-intersecting' families are covered by one set of size $t$ (that is, all sets from the family contain a given set of size $t$), while the non-trivial families are covered by $const(k,t)$ sets of size $t+1$, and for large $n$ the latter gives smaller sizes since the number of ways to extend sets of size $t+1$ to sets of size $k$ is roughly $n$ times smaller than the number of ways to extend sets of size $t$. The important common point between the $\Delta$-system method and the proof above is that in both we compare contributions of `cover sets' of different sizes. The $\Delta$-system method, however, finds and deals with these covers more efficiently.

\subsection{Families of sets with only one intersection size}\label{sec22}
One result that precedes the general question of  Deza, Erd\H os,  and Frankl is due to Deza and deals with the case when $L = \{\ell\}$, that is, all pairwise intersections in the family have size $\ell$. According to the paper of Deza \cite{D1}, the following conjecture was posed by Erd\H os on a seminar in Paris in 1973.
\begin{quote}
  Let $k> \ell$ be positive integers and consider a family $\ff$ of at least $k^2-k+2$ sets, each of size $k$, such that for any $A,B\in \ff$ we have $|A\cap B|=\ell$. Then $\ff$ is a $\Delta$-system.
\end{quote}
The bound $k^2-k+2$ is best possible: for $k-1$ being a prime power, one can consider a projective plane of order $k-1$. It contains exactly $k^2-k+1$ lines, each two intersecting in exactly $1$ point and is obviously not a sunflower.

Deza \cite{D1} noticed that this conjecture is true because it immediately reduces to his result from the recently published paper \cite{D0}, in which he studied one-distance codes. The main result of \cite{D0} has an elegant proof with geometric flavour, and we present it here. (For the rest of this subsection, $u,v,w$ will exceptionally stand for vectors.) Recall that, given two vectors $v,w$ in $\{0,1\}^n$, the {\it Hamming distance} between $v$ and $w$ is the number of coordinates $i$ such that $v_i\ne w_i$. Also recall that there is a natural correspondence between vectors in $\{0,1\}^n$ and subsets of $[n]$: set $\leftrightarrow$ characteristic vector.

\begin{thm}[Deza \cite{D0}]\label{thmde} Let $m,q$ be integers and  $\mathcal F$ be a collection of vectors $v^0,v^1,\ldots v^m\in \m \{0,1\}^n$ such that the Hamming distance between any two vectors is $2q$. If $m>q^2+q+1$ then the code is {\em trivial}: for each coordinate $i\in[n]$, either $\sum_{j\in[m]} v_i^j\le 1$ or $\sum_{j\in[m]} v_i^j\ge m-1$.
\end{thm}

This theorem allows us to deduce the conjecture of Erd\H os and Lov\'asz.

\begin{cor}[Deza \cite{D1}]\label{corde} Let $k> \ell$ be positive integers and consider a family $\ff$ of at least $k^2-k+2$ sets, each of size $k$, such that for any $A,B\in \ff$ we have $|A\cap B|=\ell$. Then $\ff$ is a $\Delta$-system.
\end{cor}
\begin{proof} Consider a collection of $m$ sets, each of size $k$ and with pairwise intersections of size $\ell\ge 1$. Note that, seen as vectors $v^1,\ldots, v^{m}$, they have pairwise Hamming distance $2(k-\ell)$. 
 We want to transform them into vectors with pairwise Hamming distance equal to their Hamming weight. For this, add to each vector $\ell-1$ new coordinates with ones, such that these $\ell-1$ coordinates for different vectors are pairwise distinct, as well as $(k-\ell-1)$ new coordinates where all vectors will have ones. It is easy to check that now each vector has Hamming weight $2q := 2(k-1)$, and pairwise Hamming distances are also $2q$. Now we add the all-zero vector $v^0$ to the collection and get a family of $m+1$ vectors $v^0,\ldots, v^m$ forming a code as in Theorem~\ref{thmde}. Thus, if $m>(k-1)^2+(k-1)+1$, then the code is trivial. This condition is equivalent to $m>k^2-k+1$.

In set terms, what does it mean for the corresponding code to be trivial? If two vectors have a common $1$, then all vectors have $1$ there, which means that any element common for two sets is common for all sets. This is exactly the definition of a $\Delta$-system.
\end{proof}

Let us prove Theorem~\ref{thmde}. The first step can be summarized in the following elegant lemma, which is an interpretation and generalization of \cite[Theorem 2.1]{D0}.

\begin{lem}[Lemma on cross-variance]\label{lemcrossvar} Consider two independent random vectors $\mathbf X, \mathbf Y\in \mathbb R^n$. Let $(\mathbf{X_1},\mathbf{Y_1}),$ $(\mathbf{X_2},\mathbf{Y_2})$ be independent copies of $(\mathbf X,\mathbf Y)$.  Then
\begin{equation}\label{eqcrossvar} 2\E[(\mathbf X-\mathbf Y)^2]\ge \E[(\mathbf X_1-\mathbf X_2)^2]+\E[(\mathbf Y_1-\mathbf Y_2)^2],\end{equation}
where the norm is the Euclidean norm.
\end{lem}
We will need the case when the random vectors are uniformly distributed on sets of vectors in $\R^m$, and the vectors are actually $\{0,1\}$-vectors.
\begin{proof}
  We have \begin{align*}
            \E[(\mathbf X_1-\mathbf X_2)^2]&=2\E[\mathbf X^2]-2\E[\mathbf X]^2 \\
            \E[(\mathbf Y_1-\mathbf Y_2)^2]&=2\E[\mathbf Y^2]-2\E[\mathbf Y]^2 \\
            \E[(\mathbf X-\mathbf Y)^2]&= \E[\mathbf X^2]+\E[\mathbf Y^2]-2\E[\mathbf X]\E[\mathbf Y].
              \end{align*}
  The statement now follows from $\E[\mathbf X]^2+\E[\mathbf Y]^2-2\E[\mathbf X]\E[\mathbf Y]\ge 0$.
\end{proof}
The motivation behind the name {\em cross-variance of $\mathbf X$ and $\mathbf Y$} for $\E[\|\mathbf X-\mathbf Y\|^2]$ is because $\E[\|\mathbf X_1-\mathbf X_2\|^2]$ is twice the variance of $\mathbf X$.

\begin{proof}[Proof of Theorem~\ref{thmde}] We may w.l.o.g. assume that $v^0$ is the all-zero vector. For any coordinate $j$, we show that  \begin{equation}\label{eqvalq} s_j:=\sum_{i\in [m]}v^i_j \text{ \ \ belongs to }[0,q+1]\cup [m-q,m].\end{equation}
 Split the set of indices $[0,m]$ into $A\sqcup B$, where $A:=\{i\in[0,m]: v_j^i=1\}$ and $B:=[0,m]\setminus A$. For a vector $v^i\in \mathcal F$, define the vector $w^i\in \mathbb R^{n-1}$ obtained from $v^i$ by deleting the $j$-th coordinate. Let us define a random vector $\mathbf X$ ($\mathbf Y$) by selecting a uniformly random $i\in A$ ($i\in B$) and taking the  vector $w^i$. We apply Lemma~\ref{lemcrossvar} to $\mathbf X$ and $\mathbf Y$ and get
  $$2\E[\|\mathbf X-\mathbf Y\|^2]\ge \E[\|\mathbf X_1-\mathbf X_2\|^2]+\E[\|\mathbf Y_1-\mathbf Y_2\|^2]$$
  At the same time, $\|w^i-w^\ell\|^2$ is exactly the Hamming distance between $w^i$ and $w^\ell$. It is equal to $2q$ for any two distinct vectors with indices in $A$, or any two distinct vectors with indices in $B$. If the vectors correspond to different sets of indices, then their Hamming distance is $2q-1$. Also note that $|A| = s_j$ and $|B| = m-s_j+1$, and thus $\E[\|\mathbf X_1-\mathbf X_2\|^2] = \frac{s_j-1}{s_j} \cdot 2q$, and $\E[\|\mathbf Y_1-\mathbf Y_2\|^2] = \frac{m-s_j}{m-s_j+1} \cdot 2q$. Combining with the displayed inequality, we get
  \begin{align*} 2(2q-1)&\ge 2q\left(\frac{s_j-1}{s_j}+\frac{m-s_j}{m-s_j+1}\right)\\
  \frac 1{s_j}+\frac1{m-s_j+1} &\ge \frac 1q\\
  (m+1)q&\ge s_j(m-s_j+1)\end{align*}
The last inequality and the condition on $m$ from the statement imply that $s_j\le q+1$ or $s_j\ge m-q$. Consider the set of coordinates $U$ such that $U:=\{j: s_j\ge m-q\}$.
\begin{cla}We have $|U|\le q$.\end{cla}
\begin{proof}   For any $j\in U$, there are at most $q$ vectors  $v\in \{v^1,\ldots, v^m\}$ such that $v_j=0$. If $|U|\ge q+1$, then, taking  $U'\subset U$ with $|U'|=q+1$, and by pigeon-hole principle, there are $m-(q+1)q\ge 2$ vectors from $\mathcal F$ that have $1$'s on each coordinate from $U'$. Since $v^0$ is the all-zero vector, every other vector has exactly $2q$ coordinates equal to $1$. Thus, these two vectors have $q-1$ coordinates equal to $1$ outside $U'$, and thus their Hamming distance is at most $2q-2$, a contradiction.\end{proof}

\begin{cla}For any vector $v^i\in \mathcal F$, $i\in[m]$, the set $U^i:=U\cap\{j: v_j^i=1\}$ has size at least $q$. \end{cla}
\begin{proof}
Put ${\rm supp}(v^{i'}):=\{j: v_j^{i'}=1\}$.  On the one hand, $$\sum_{i'=1}^m|{\rm supp}(v^{i})\cap {\rm supp}(v^{i'})|=2q+(m-1)q=(m+1)q.$$  The first term is the sum of coordinates of $v^i$, and the second term accounts for the fact that each of the $m-1$ non-zero vectors has $q$ common ones with $v^i$. On the other hand, if $|U^i|\le q-1$ then \begin{align*}\sum_{i'=1}^m\sum_{j\in {\rm supp}(v^i)} v^{i'}_j&\le m|U^i|+(q+1)|{\rm supp}(v^i)\setminus U| \le m(q-1)+(q+1)^2\\ &= (m+1)q-m +q^2+q+1<(m+1)q,\end{align*}
a contradiction. 
\end{proof}
Combining the two claims, we see that $U^i = U$ for all $i\in[m]$, and thus $s_j=m$ for all $j\in U$. We have $|U|=q$, and so in order to assure Hamming distance $2q$, the vectors $v_i$, $i\in[m]$ must have pairwise disjoint sets of $1$'s outside $U$. This completes the proof of the theorem.
\end{proof}
Let us summarize that, as long as $\ff$ is an $(n,k,\{\ell\})$-system and $|\ff|>k^2-k+1$, then $\ff$ is a $\Delta$-system with kernel of size $\ell$. Since outside the kernel the sets from the family are disjoint, we must have
$$|\ff|\le \frac{n-\ell}{k-\ell}.$$
\subsection{The result of Deza, Erd\H os and Frankl}\label{sec23}
This result came as a natural generalization of the result of Deza \cite{D1} and the result of Deza, Erd\H os and Singhi \cite{DES}, who studied the $|L|=2$ case of the problem. In \cite{DES} the authors also cite and use a result of a preprint of Deza and Erd\H os ``to appear in Aequationes Math.'', which has never actually appeared. (My guess is that it was absorbed by the Deza--Erd\H os--Frankl paper.)
\begin{thm}[Deza, Erd\H os, Frankl \cite{DEF}]\label{thmdef}
Fix positive integers $n,k>0$ and $r\ge 2$, and a set $L=\{\ell_1<\ldots<\ell_r\}\subset [0,k-1]$. Assume that $\epsilon>0$, $n\ge n_0(k,\epsilon)$, and that $\ff$ is an $(n,k,L)$-system.
\begin{itemize}
  \item[(i)] If $|\ff|>\big(\epsilon+\max\{k-\ell_1+1,\ell_2^2-\ell_2+1\}\big) \prod_{i=2}^r\frac{n-\ell_i}{k-\ell_i}$ then there exists a set $D$ of cardinality $\ell_1$ such that $D\subset A$ for every $A\in \ff$.
  \item[(ii)] If $|\ff|\ge k^22^{r-1}n^{r-1}$ then
  $$(\ell_2-\ell_1)|(\ell_3-\ell_2)|\ldots |(\ell_r-\ell_{r-1})|(k-\ell_r).$$
  \item[(iii)] We have $$|\ff|\le \prod_{i=1}^r\frac{n-\ell_i}{k-\ell_i}.$$
\end{itemize}
\end{thm}
First, remark that we restrict to $r\ge 2$ since the case $r=1$ is dealt with in Corollary~\ref{corde}.
Below we sketch the original proof of Deza, Erd\H os and Frankl and discuss the $\Delta$-system part in detail. The proof of the theorem is by induction on $k$, the case $k=1$ being trivial.

We first  treat the case $\ell_1=0$. For each $x$, the family $\ff(x)$ is an $(n-1,k-1,\{\ell_2-1,\ldots, \ell_r-1\})$-system. Apply inductive assumption (iii)  to each $\ff(x)$ and then do double-counting $k|\ff| = \sum_{x\in[n]}|\ff(x)|$. It is not hard to check that it gives the bound on $|\ff|$ as is claimed in (iii). The divisibility part $(\ell_3-\ell_2)|\ldots|(k-\ell_r)$ is by inductive assumption (ii), applied to any large $\ff(x)$. We must have at least one large $\ff(x)$ by double counting, provided that $\ff$ itself is large.

It remains to show that $\ell_2|(\ell_3-\ell_2)$ in order to complete the proof of (ii) and the theorem in this case. First, we only leave vertices of high degree. Specifically, we remove elements $x$ with $|\ff(x)|\le k^22^{r-2}n^{r-2}$  one by one, until we reach the situation with all $x$ having  degree more than that expression. (Such deletion is possible thanks to the generous lower bound on $|\ff|$ in (ii).) Then, applying inductive statement (i) for $k-1$ to $\ff(x)$ with any of the remaining $x$, we see that there must be a set $D_x$ of size $\ell_2-1$ that all sets from $\ff(x)$ must contain. Using the lower bound on $|\ff(x)|$, it is not difficult to check that sets $D_x\cup \{x\}$ for different $x$ either coincide or are disjoint and that, consequently, each $A\in \ff$, if intersecting $D_x\cup \{x\}$, must contain it entirely.\footnote{If $y\in D_x$ and $D_y\cup \{y\}\ne D_x\cup\{x\}$, then all sets containing $x$ must contain $D_x$, and thus $y$. But then they also contain $D_y$, and have a fixed set of size bigger than $\ell_2$, which reduces the number of available distances to $r-2$. This, in turn, implies by induction that the degree of $x$ is too small.} Thus all intersections of the sets in $\ff$ must be divisible by the size of $D_x\cup \{x\}$, which is $\ell_2$. At the same time, intersection $\ell_3$ must appear, since otherwise $|\ff(x)|$ is too small.

The main challenge is in the case $\ell_1\ge 1$ and, more specifically, in the part (i) of the theorem. Once we managed to show that all sets should contain a  set $X$ of size $\ell_1$, we can finish the proof by applying inductive statement for the $(n-\ell_1,k-\ell_1,\{0,\ell_2-\ell_1,\ldots, \ell_r-\ell_1\})$ system $\ff(X)$. (That is, reduce it to the $\ell_1=0$ case with a smaller $k$.)

Proving (i) for $\ell_1\ge 1$ is where the $\Delta$-system method pops up. The goal is to find an economical family $\mathcal B$ (the {\it base}) that {\em covers} $\mathcal F$ in the sense that, for each $A\in \ff$ there is $B\in \m B$ such that $B\subset A$, and then to bound the size of $\m F$ using the structure of $\m B$. The first observation is as follows.
\begin{obs}\label{obsdelta1} Let $k\ge \ell>0$ be integers. Let $A$ be set of size $k$ and $\m F$ is a $\Delta(k-\ell+2)$-system with kernel $D$ of size $\ell$. Then $|A\cap F|\ge \ell$ for every $F\in \m F$ implies $|A\cap D|\ge \ell$.
\end{obs}
\begin{proof}
  Put $\ff = \{F_1,\ldots, F_{k-\ell+2}\}$ and assume that $|A\cap D| = \ell-x$ for some $x\ge 1$. Then $|A\cap (F_i\setminus D)|\ge x$ for each $i$ and, since $F_i\setminus D$ are pairwise disjoint, we get that $|A|\ge (\ell-x)+(k-\ell+2)x =k+1+(k-\ell+1)(x-1)\ge k+1$, a contradiction.
\end{proof}
A very similar observation, that was not explicitly stated in \cite{DEF}, but was de-facto used there in the proof is as follows. We omit the proof, which is akin to the previous proof.
\begin{obs}\label{obsdelta2} Let $k$ be an integer. Let $A$ be set of size $k$ and $\m F$ is a $\Delta(k+1)$-system with kernel $D$. Then $|A\cap D| = |A\cap F|$ for some $F\in \m F$.
\end{obs}

The following lemma is an important ingredient for the construction of the base of the family, which comes in the lemma after.
\begin{lem}\label{lemdeltadef}
  Let $s,k,t$ be integers such that $t>(s-1)k$. For each $i\in [s]$ let $F_i^1,\ldots F_i^t$ form a $\Delta(t)$-system with kernel $E_i$, where $|F_i^j| = k$ for $i\in[s]$, $j\in[t]$.

  Suppose that $E_1,\ldots, E_s$ form a $\Delta(s)$-system with kernel $D$. Then there are distinct indices $j_1,\ldots, j_s\in[t]$, such that the sets $F_1^{j_1},\ldots, F_s^{j_s}$ form a $\Delta(s)$-system with kernel $D$.
\end{lem}
\begin{proof}
  For each $i=1,\ldots, s$ we shall find a set $F^{j_i}_i$ that is disjoint from \begin{equation}\label{eqsetdelta}\Big(\bigcup_{\ell=1}^{i-1}F^{j_\ell}_\ell \cup\bigcup_{\ell=i+1}^s E_\ell\Big)\setminus D.\end{equation}
 Then it should be clear that the sets  $F_1^{j_1},\ldots, F_s^{j_s}$ form a $\Delta(s)$-system with kernel $D$. Why is it possible to find such sets? First, by the choice of $F_{\ell}^{j_\ell}$, $\ell<i$, we have $$F_{\ell}^{j_\ell}\cap F_i^j\setminus D\subset F_i^j\setminus E_i$$   for any $\ell<i$ and any $j\in [t]$.  Second, the sets  $F_i^j\setminus E_i$, $j\in[t]$ are pairwise disjoint. Third, the set \eqref{eqsetdelta} has size at most $k(s-1)$. 
 Concluding, we have $t>(k-1)(s-1)$, and so by the pigeon-hole principle we can find $j_i\in[t]$ such that $F_i^{j_i}\setminus E_i$ is disjoint from the displayed set.
\end{proof}
Using this lemma, we may construct the {\em base} $\m B$ of the family $\ff$. For a family $\m X\subset 2^{[n]}$, denote by $\m X^{(\ell)}:=\m X\cap {[n]\choose \ell}$ the $\ell$-th layer in $\m X$.
\begin{lem}\label{lemdeltabase}
  For any $(n,k,L)$-family $\ff\subset {[n]\choose k}$ there exists a family $\m B\subset {[n]\choose \le k}$ such that for any $F\in \m F$ there exists $B\in \m B$ with $B\subset F$, and that satisfies the following:
  \begin{itemize}
  \item[(i)] for any $B\in \m B$ we have $|B|\in L\cup \{k\}$;
  \item[(ii)] for any $B\in \m B^{(\ell_1)}$, the family $\ff$ contains a $\Delta(k-\ell_1+2)$-system with center in $B$. For any $i\in [2,r]$ and $B\in \m B^{(\ell_i)}$, the family $\ff$ contains a $\Delta(k^i+1)$-system with kernel $B$;
  \item[(iii)] the family $\m B^{(\ell_2)}$ does not contain a $\Delta(k-\ell_1+2)$-system. For each $i\in [3,r]$, the family  $\m B^{(\ell_i)}$ does not contain a $\Delta(k^{i-1}+1)$-system. The family $\m B^{(k)}$ does not contain a $\Delta(k^r+1)$-system.
  \end{itemize}
\end{lem}
Note that we treat the first layers of $\m B$ slightly differently.
\begin{proof}
  We construct the base $\m B$ in a greedy fashion: first, we search for sets of size $\ell_1$ that are a kernel of a $\Delta(k-\ell_1+2)$-system. If we find one, we remove the sets forming this $\Delta$-system from the family $\ff$, add the kernel to $\m B$ and continue our search in the family. Once there are no more kernels of size $\ell_1$ of $\Delta(k-\ell_1+2)$-systems, we start searching for sets of size $\ell_2$ that are kernels of $\Delta(k^2+1)$-system in the remaining family. Once we ran out of such sets of size $\ell_2$, we start searching for sets of size $\ell_3$ that are kernels of a $\Delta(k^3+1)$-system in the remaining family etc.

  At the last step, we include in $\m B$ the remaining sets of size $k$ and note that there is no $\Delta(k^{r}+1)$-system formed by these sets. Indeed, otherwise its core would have been included into $\m B^{(\ell_i)}$ for some $i\in[r]$ and the sets would have been removed.

  This verifies the first two parts of the lemma. 
  To verify (iii), we argue indirectly. Assume that for some $i\ge 3$ the sets $B_1,\ldots, B_{k^{i-1}+1}\in \m B^{(\ell_i)}$ form a $\Delta$-system (with core of size in $L$). Recall that each of them is a core of a $\Delta(k^i+1)$-system of sets from $\m F$. We may apply  Lemma~\ref{lemdeltadef} to these $\Delta$-systems and conclude that the union of these $\Delta$-systems contains a $\Delta(k^{(i-1)}+1)$-system with core of size at most $\ell_{i-1}$. However, this contradicts the greedy construction of $\m B$: we should have included the kernel of this sunflower in $\m B$ before including $B_1,\ldots, B_{k^{i-1}+1}$. The argument for $i=2$ is virtually identical. This contradiction proves the lemma.
\end{proof}

In retrospect, this lemma is far from being efficient if one tries to get best bounds on $n$ when applying the $\Delta$-system method. We shall discuss and compare different possible approaches for constructing bases in Section~\ref{secbase}.

Let us return to the proof of the theorem. Lemma~\ref{lemdeltabase} allows us to relate the structure of the base and the size of $\ff$. If $\ff$ contains a $\Delta(k-\ell_1+2)$-system with core $X$ of size $\ell_1$, then by Observation~\ref{obsdelta1}, all sets in $\m F$ must contain $X$, and we are done. Thus, we may assume that this is not the case and, in particular, $\m B^{(\ell_1)}$ is empty. Next, remark that by Lemma~\ref{lemdeltadef} with $s=2$, for any two sets $B_1,B_2$ in $\m B^{(\ell_2)}$ we can find two sets $F_1,F_2$ from the $\Delta$-systems with cores $B_1,B_2$, respectively, such that $F_1\cap F_2 = B_1\cap B_2$.  Since $|B_1\cap B_2|< |B_1|=\ell_2$ and $|F_1\cap F_2|\in L$, we get $|B_1\cap B_2|=\ell_1$. Thus, if $\m B^{(\ell_2)}$ contains at least $\max\{k-\ell_1+2, \ell_2^2-\ell_2+2\}$ sets, then $\m B^{(\ell_2)}$ must be a $\Delta(\ge k-\ell_1+2)$-system with kernel $K$ of size $\ell_1$ by Corollary~\ref{corde}. But this contradicts Lemma~\ref{lemdeltabase} (iii). Thus, $|\m B^{(\ell_2)}|\le \max\{k-\ell_1+1, \ell_2^2-\ell_2+1\}$.

What is left is to count the contribution of each $\m B^{(\ell)}$ to the cardinality of $\m \ff$. For each $B\in \m B^{(\ell_2)},$
the family $\ff(B)$ is an $(n-\ell_2,k-\ell_2,\{0,\ell_3-\ell_2,\ell_4-\ell_2,\ldots, \ell_r-\ell_2\})$-system, and so by induction $|\ff(B)|\le \prod_{i=2}^r\frac{n-\ell_i}{k-\ell_i}$. We have
$$\sum_{B\in \m B^{(\ell_2)}}|\ff(B)|\le \max\{k-\ell_1+1, \ell_2^2-\ell_2+1\}\prod_{i=2}^r\frac{n-\ell_i}{k-\ell_i}.$$
Next, for each $B\in \m B\setminus \m B^{(\ell_2)}$, the family $\ff(B)$ has only $\le r-2$ distances and thus $|\ff(B)|\le n^{r-2}$. At the same time, using Theorem~\ref{thmer}, there are at most $\sum_{i=2}^r\varphi(k,k^{i}+1)$ sets in $\m B$, and thus
$$\sum_{B\in \m B\setminus\m B^{(\ell_2)}}|\ff(B)|\le \sum_{i=1}^r\varphi(k,k^{i}+1)n^{r-2}<\epsilon \prod_{i=2}^r\frac{n-\ell_i}{k-\ell_i},$$
provided $n$ is large enough. Thus, if there is no set $D$ of size $\ell_1$ contained in all sets of $\ff$, then $|\ff|<(\max\{k-\ell_1+1, \ell_2^2-\ell_2+1\}+\epsilon)\prod_{i=2}^r\frac{n-\ell_i}{k-\ell_i}$, contradicting the assumption. This concludes the proof.

{\bf Remark. }  For a different proof of Theorem~\ref{thmdef}  we refer to the book \cite{FT}. That proof uses Observation~\ref{obsdelta1}, but avoids the  construction of the base for the family.

\subsection{Further developments of the early $\Delta$-system method}\label{sec2.4} The development of the method for the next several years came mostly thanks to the work of Peter Frankl. He pushed forward the original approach and proposed another point of view on how to decompose a family using $\Delta$-systems. This work mostly revolved around families with prescribed intersections.

In Section~\ref{sec22}, we discussed families with only one intersection size. Erd\H os and S\'os (see the collection of problems by Erd\H os \cite[page 18]{Erd75}) suggested the {\it opposite} question on families $\ff\subset {[n]\choose k}$ that are allowed to have all intersection but one. Specifically, the conjecture of Erd\H os and S\'os was that if $\ff\subset {[n]\choose k}$ satisfies $|A\cap B|\ne 1$ for each $A,B\in \ff$, then $|\ff|\le {n-2\choose k-2}$ for $n>n_0(k)$. That is, forbidding intersection $1$ gives the same bound as forbidding intersections $0$ and $1$, by the Erd\H os--Ko--Rado theorem. In terms of Section~\ref{sec21}, it asks for the value $m(n,k,\{0,2,3,\ldots, k-1\})$ of the size of the largest $(n,k,\{0,2,3,\ldots, k-1\})$-system.

In the paper \cite{Fra1}, Frankl proved this conjecture using the $\Delta$-system method.\footnote{This paper was published in 1977, which is earlier than the Deza--Erd\H os--Frankl paper \cite{DEF}, and does not refer to \cite{DEF}, but the latter was received in the journal earlier.} He employs a variant of the $\Delta$-system construction described in the previous section, applies it to families $\ff(x)$ and, importantly, notices the following:
\begin{obs}
  If $\ff$ is an $(n, \{0,2,3,\ldots, k-1\},k)$ system and  $\m B_x$, $\m B_y$ are the bases of $\ff(x),\ff(y)$, respectively, constructed as in the previous section, then for any $C_x\in \m B_x$ and $C_y\in \m B_y$, we have $|(\{C_x\cup \{x\})\cap (C_y\cup \{y\})|\ne 1$.
\end{obs}
That is, he explicitly states the fact that the property of the family is transferred to the base. (It is somewhat implicit in the statement of Lemma~\ref{lemdeltabase}, although it is being used in the proof.) He then again analyzes the structure of sets of small uniformity from the base and their contributions to the size.

\subsubsection{The second approach to base construction}\label{sec241} In 1980, Peter published two papers on restricted intersections, in which he developed another way to decompose the family using sunflowers. Both use a certain proposition (Proposition~\ref{prop77a}) on $\Delta$-systems that was proved in another of his papers \cite{Fra77b} from 1977, which suggests that he employed and developed aspects of the $\Delta$-system method in that paper. Unfortunately, I could not find an electronic copy of that paper.

In the paper \cite{Fra80b}, Frankl managed to determine the order of magnitude of $m(n,k,L)$ for all $k\le 7$ and all possible $L$, with the exceptions of $k=7$, $L=\{0,2,3,5\}$ and $L=\{0,2,3,5,6\}$. That is, in this short paper, he dealt with an impressive list of more than $100$ possible cases. The methods combined the $\Delta$-system method, linear-algebraic method and algebraic constructions. Below, we present the $\Delta$-system part of his arguments. In what follows, we assume that $L = \{\ell_1\ldots, \ell_r\}$.

Peter first uses Lemma~\ref{lemdeltabase} to conclude that the Deza--Erd\H os--Frankl base has size at most $c(k)$,\footnote{Here and below, we use $c(k)$ some large constant, depending on $k$ only.} and thus one of the bases $B$ is contained in $|\ff|/c(k)$ sets. Looking only at $\ff(B)$ allows to reduce the problem to the case when the smallest element in $L$ is $0$. We can formalize this as follows.
\begin{prop}\label{propfra}
  We have $m(n,k,\{\ell_1,\ldots, \ell_r\}) = \Theta_k(m(n-\ell_1,k-\ell_1,\{0,\ell_2-\ell_1,\ldots, \ell_r-\ell_1\}))$.
\end{prop}

 Then he proposes the following greedy approach to construct another base.  Take the {\it largest} set $K$ such that $K$ is a kernel of a $\Delta(k+1)$-system in $\ff$, add $\{(K,F):K\subset F\in \ff\}$ to a separate `base' family $\m S$, and update $\ff:=\ff\setminus \ff[K]$. Repeat until no $\Delta(k+1)$-systems are left in $\ff$. At the end of the procedure, we get the family $\m S$, which contains both the base and the decomposition of $\ff$ by the base sets, and the remainder family of size at most $c(k)$, which we may ignore in this case. This decomposition has several properties, which we list as three propositions. We then illustrate their use.

The first proposition is again about the transference of the property to the base. Moreover, the property is  enhanced in some cases.

\begin{prop}\label{prop77b}
  Let $(K_1,F_1),(K_2,F_2)\in \m S$. Then $|K_1\cap K_2|\in L=\{\ell_1,\ldots, \ell_r\}$. Moreover, if $F_1\setminus K_1 = F_2\setminus K_2$, then $|K_1\cap K_2|\in \{\ell_1-|F_1\setminus K_1|,\ldots, \ell_r-|F_1\setminus K_1|\}$.
\end{prop}

The second proposition is as follows.
\begin{prop}\label{prop77c} Suppose that $\ff$ is an $(n,k,\{\ell_1,\ldots, \ell_r\})$-system and it does not contain a $\Delta(k+1)$-system with kernel of size at least $\ell_{s+1}$, $s\in [r-1]$. Then there is a family $\ff'\subset \ff$ such that $|\ff'|\ge |\ff|/c(k)$ and such that $\ff'$ is an $(n,k,\{\ell_1,\ldots, \ell_s\})$-system.
\end{prop}
That is, if a certain size does not appear as a kernel of a $\Delta(k+1)$-system, then we can remove it from the allowed intersection list at a cost of reducing the size of the family by a constant.

The proof is akin to a greedy construction of a large independent set in a graph of small maximum degree: include a new set $A$ into the family $\ff'$ and then remove all sets that intersect it in at least $\ell_{s+1}$ elements. The number of these sets is at most $c(k)$ since, for any $\ell_{s+1}$-element subset $S$ of $A$, there is no $\Delta(k+1)$-system with kernel containing $S$.

Finally, he refers to \cite{Fra77b} for the following statement. (It concerns the greedy procedure that we described before  Proposition~\ref{prop77b}.)

\begin{prop}\label{prop77a} Each of the families $\ff[U]$ removed at each step of the procedure satisfies $|\ff[U]|\le c(k)n$. Consequently, the family of base sets $\m D:=\{K: (K,F)\in \m S\text{ for some }F\in \ff\}$ satisfies  $|\m D|\ge |\ff|/c(k)n$.
\end{prop}
This is easy to verify since each $\ff(U\cup\{x\})$, $x\not\in U$, does not contain a $\Delta(k+1)$-system by the choice of $U$.

Let us show, how the first two propositions are being applied in order to prove that if $\ff$ is an $(n,5,\{0,1,3,4\})$-system, then $|\ff|\le c(5)n^2$. In the base construction process for $\ff$, we first select kernels $U$ of size $4$. Let $s$ be the last step when we selected a kernel $U$ of size $4$. Let $\m S_s$ be the pairs added to the base up to the step $s$, and let $\ff_s$ be the remainder family. By the choice of $s$, the family $\ff_s$ does not have $\Delta(k+1)$-systems with kernel of size $4$, and thus by Proposition~\ref{prop77c} it contains a subfamily $\ff'$ of size $|\ff'|\ge |\ff_s|/c(5)$ that is an $(n,5,\{0,1,3\})$-system.

Peter then shows that an $(n,5,\{0,1,3\})$-system must have size at most $c(5)n^2$ with the help of the following proposition (presented in a slightly modified form).

\begin{prop}\label{prop77d}
  Fix integer $q$ and a constant $m$. Any family $\ff\subset 2^{[n]}$ has a subfamily $\m G\subset \ff$ of size at least $|\ff|- m{n\choose q}$ such that for any set $Q$ of size $q$, $Q\subset F\in\m G$, it is contained in at least $m$ sets from $\m G$.
\end{prop}
The statement and the proof of the proposition are akin to the statement and the proof of the fact that any graph with many edges contains a subgraph with high minimum degree.

Using it, he chooses a subfamily $\m G\subset \ff'$ as guaranteed by the proposition, and then shows that $\m G$ should be empty using a simple, but somewhat tedious, `covering number' argument which we avoid.

So far, we showed that $|\ff_s|\le c(5)n^2$ with the help of Propositions~\ref{prop77c} and~\ref{prop77d}. Next, we bound the size of $\m S_s$. Take $y$ such that $\{y\} = F\setminus K$ for at least $|\m S_s|/n$ pairs $(K,F)\in \m S_s$ (recall here that by construction all $K$'s in $\m S_s$ have size $4$). Consider the family $\ff_y:=\{K: (K,F)\in \m S_s, F\setminus K = \{y\}\}$. Then, in view of Proposition~\ref{prop77b}, for $K\ne K'\in \ff_y$ we have $|K\cap K'|\in \{0,3\}$. That is, $\ff_y$ is an $(n,\{0,3\}, 5)$-system. We can apply Theorem~\ref{thmdef} (ii) and note that $\ell_2-\ell_1\nmid k-\ell_2$, and thus $|\ff_y|\le c(5)n$. Therefore, $|\m S_s|\le c(5)n^2$. We conclude that $|\ff|\le |\m S_s|+|\ff_s|\le 2c(5)n^2$.

Finally, let us illustrate, how Proposition~\ref{prop77a} is being used. Let $\ff$ be an $(n,6,\{0,1,3,4,5\})$-system. Consider the family $\m D$ and take the value of $\ell$, for which the family $\m D^{(\ell)}$ of $\ell$-element sets in $\m D$  is the largest. then $\m D^{(\ell)}$ is an $(n,\ell,\{0,1,3,4\})$-system (some of the intersection sizes may be redundant if $\ell<5$). Above, we showed that an $(n,5,\{0,1,3,4\})$-system should have size at most $c(5)n^2$ (in the cases $\ell<5$ the same bound holds as well), which implies that an $(n,6,\{0,1,3,4,5\})$-system should have size at most $c(6)n^{3}$.\\

In the second paper \cite{Fra80a}, Frankl proved a general result concerning $\ell$-avoiding set systems: he showed that if $\ff\subset {[n]\choose k}$ satisfies $|A\cap B|\ne \ell$ for each $A,B\in \ff$, then $|\ff|\le c(k){n-\ell-1\choose k-\ell-1}$ for $n>n_0(k)$ and $n\ge 3\ell$, and sketched a proof for replacing $c(k)$ by $(1+o(1))$ for $k>3\ell+2$. In order to do that, he used the same machinery as above, combined with a clever induction on $k$: instead of proving a result for one forbidden intersection, he proves stronger result with a forbidden interval of intersections. The reason behind it is Proposition~\ref{prop77b}, which may cause the change in the number of forbidden intersections.

\subsubsection{Results for $t$-intersecting families} In the paper \cite{Fra2}, Frankl obtains various stability results for the intersecting and $t$-intersecting cases of the Erd\H os--Ko--Rado theorem. Specifically, he studies the families in which no element is contained in more than $c$-fraction of sets. Here, he proposes yet another base, which is a variant of the Deza--Erd\H os--Frankl construction. For a family $\ff$, he considers the family $\m B$ of all sets that intersect each set in $\ff$ in at least $t$ elements. He then restricts his attention to inclusion-minimal sets and, importantly, only to the sets which have many extensions: he fixes the smallest size $\ell$ of a set in the base and essentially restricts to studying $X$ of size $\ell$ such that $|\ff(X)|\ge k{n-\ell-1\choose k-\ell-1}$.

In a follow-up paper \cite{Fur78} on the same question, F\"uredi used a similar notion of a {\it nucleus}. There, he suggested the following construction: replace sets in a family with their subsets while preserving the property of the family (there, it was the value of the matching number, which obviously included the property of being intersecting). He refers to the work of Erd\H os and Lov\'asz \cite{EL} and Lov\'asz' dissertation for some of the properties of such nucleus. In particular, he gives an efficient bound for the number of sets in such a nucleus. This notion of a base is close to the notion of a {\em generating set} from the celebrated paper of Ahlswede and Khachatrian \cite{AK} (which they claimed to be new).  We discuss the construction of Erd\H os and Lov\'asz and Ahlswede and Khachatrian in Section~\ref{secbase}.

\subsection{A precursor to F\"uredi's structural result}\label{sec25}
In the years that followed, Peter Frankl continued to actively work on $(n,k,L)$-systems. One of these papers \cite{Fra83} developed an approach that, in retrospect, can be seen as a weak version of F\"uredi's structural result that we discuss in the next section (see Theorem~\ref{thmfur} from Section~\ref{sec31}). Frankl studies $(n,k,L)$-families for $|L|=3$ and, in particular, gives conditions when $|\ff| = O(n)$ and $|\ff| = \Omega(n^2)$. We sketch some of the ideas from that paper below.

Using Proposition~\ref{propfra}, we assume that $L=\{0,a,b\}$, $b>a$. Using the $\Delta$-system method, he selects a subfamily $\ff'\subset \ff$ that is much more structured. He defines $\m A, \m B$ to be the kernels of large $\Delta$-systems in $\ff$ of size $a$ and $b$, respectively. First, he shows that all but $O_k(n)$ sets $F\in \ff$ must satisfy the following property: $F$ is equal to the union of kernels from $\aaa\cup \bb$ contained in $F$. Second, he shows that all but $O_k(n)$ pairs $B_1,B_2\in \bb$ that intersect must intersect in a set $A\in \m A$.\footnote{The proofs of both claims are similar to, say, the proof of Proposition~\ref{prop77a}, but slightly more involved.} Every `bad' pair is contained in at most $1$ set from $\ff$, and thus we can remove $O_k(n)$ sets to get rid of both situations and get the family $\ff'$. Now every $F\in \ff'$ is a union of kernels contained in it and, moreover, every two $b$-kernels in $F$ are either disjoint or intersect in an $a$-kernel. This property ($F$ equal to the union of kernels) actually gives necessary and sufficient condition for $|\ff|\gg n$. He also showed that, if the condition holds, then $m(n,k,L) = \Omega(n^{k/(k-1)})$ and, moreover, showed cases when $\Omega(n^{k/(k-1)})=m(n,k,L)=O(n^{3/2})$, which was the first manifestation of the fact that determining $m(n,k,L)$ is a hard problem in general.

\section{Delta-system method after F{\"{u}}redi's structural theorem}\label{sec3}

The result of F\"uredi, discussed below, gave a completely new direction to the Delta-system method. Although it appears to be  motivated by Frankl's ideas from \cite{Fra83} and as such was an organic development of the $\Delta$-system method, it would be fair to say that the Delta-system method as is used in the Deza--Erd\H os--Frankl paper and the one discussed in this section are two completely different methods.

\subsection{F{\"{u}}redi's theorem}\label{sec31}
A key development for the Delta-system method came with the paper of Zolt\'an F\"uredi from 1983 \cite{Fu1},\footnote{It is in this paper that $\Delta(s)$-systems are called $s$-stars, a name that appears in the follow-up papers of Frankl and F\"uredi.} who proved a much stronger Ramsey-type statement concerning large families of $k$-element sets. As F\"uredi writes in his paper, the following was conjectured by Peter Frankl:
\begin{quote}
  There is a positive constant $c(k,s)$ such that every family $\ff\subset {[n]\choose k}$ has a subfamily $\ff^*$ of size at least $c(k,s)|\ff|$ and such that for any two $F_1,F_2\in \ff^*$ the set $F_1\cap F_2$ is a kernel of a $\Delta(s)$-system.
\end{quote}
We can speculate that this conjecture was very natural to state after Frankl's paper \cite{Fra83}. Zolt\'an F\"uredi managed to prove it in a stronger form. In order to state the result, we need some preparations.

Let $\ff\subset {[n]\choose k}$ be a family. This family is {\it $k$-partite} if there exists a partition $[n] = X_1\sqcup X_2\sqcup\ldots \sqcup X_k$ such that for any $A\in \m F$ and $i\in[k]$ we have $|A\cap X_i|=1$. Further, for each set $B\in \ff$ define the {\it trace} $\ff|_B:=\{A\cap B: A\in \ff, A\ne B\}$. (We exclude $B$ itself from the trace for convenience.) If $\ff$ is $k$-partite with parts $X_1,\ldots, X_k$, then each set in $\ff|_B$ intersects each part $X_i$ in at most $1$ element. The following notion encodes, which parts are intersected. Define the {\it projection} $\pi: 2^{[n]}\to [k]$ as follows: $\pi(Y) = \{i\in [k]: Y\cap X_i\ne \emptyset \}$. We can naturally extend this definition to $\pi(\ff|_B)\subset 2^{[k]}$. 

\begin{thm}[F\"uredi \cite{Fu1}]\label{thmfur} Given integers $s>k\ge 2$ and a family $\ff$ of $k$-element sets, there exists $c(k,s)>0$ and a family $\ff^*\subset \ff$ such that
\begin{enumerate}
  \item $|\ff^*|\ge c(k,s)|\ff|$;
  \item $\ff^*$ is $k$-partite;
  \item $\pi(\ff^*|_A) = \pi(\ff^*|_B)$ for any $A,B\in \ff$;
  \item for each distinct $A,B\in \ff^*$ their intersection $A\cap B$ is a kernel of a $\Delta(s)$-system contained in $\ff^*$, i.e., there exist $F_1,\ldots, F_s\in \ff^*$ such that $F_i\cap F_j = A\cap B$ for all $i\ne j\in[s]$.
\end{enumerate}
\end{thm}
Let us call such $\ff^*$ a {\it $(k,s)$-homogeneous structure}. As we shall see, this lemma is a very powerful tool to discern structure in the family $\ff$.
Unfortunately, the bound on $c(k,s)$ that comes out of F\"uredi's proof is doubly-exponential in $k$: $c(k,s)>(sk2^k)^{-2^k}$. This limits the applicability of Theorem~\ref{thmfur} to $n>n_0(k,s)$, where the latter function is doubly-exponential in $k$. We present a more quantitative variant of this theorem due to Jiang and Longbrake (with a very similar proof) in Section~\ref{sec62}. Recently, Janzer, Jin, Sudakov and Wu \cite{JJSW} showed that the doubly-exponential dependency of $c(k,s)$ on $k$ is necessary. Moreover, they proved a weaker version of the theorem, sufficient for many applications, with a single-exponential bound on the constant. See Section~\ref{sec63} for details.

The family $\m M:=\pi(\ff^*|_A)$ for some $A\in \ff^*$ is called the {\it intersection structure} of $\ff^*$. Note that it does not depend on $A$.  Let us mention one important implication.
\begin{cor}\label{corfur}
  Suppose that $\ff^*\subset \ff$ is a  $(k,k+1)$-homogeneous structure and $\m M$ is the intersection structure of $\ff^*$. Then
  \begin{itemize}
  \item for each $A\in \ff^*$ and $Y_1,Y_2\in \ff^*|_A$ there are sets  $B_1,B_2\in \ff^*$ such that $B_1\cap B_2 = Y_1\cap Y_2$;
  \item If $\ff$ is an $(n,k,L)$-system, then all sizes and intersections of sets in $\m M$ lie in $L$;
  \item $\m M$ is {\em closed under intersection}, i.e., if $Y_1,Y_2\in \m M$ then $Y_1\cap Y_2\in \m M$.
  \end{itemize}
\end{cor}
\begin{proof}
  The proof of the first part is an application of Lemma~\ref{lemdeltadef} with $s=2$ to the  $\Delta(k+1)$-systems with kernels  $Y_1,Y_2$. The second  part immediately follows from the first. Let us prove the third part. Take $A\in \ff^*$ and note that $\m M = \pi(\ff^*|_{A})$. Take sets $A_1,A_2\in \ff^*$ such that $\pi(A\cap A_i)=Y_i$. By the first part, we can find $B_1, B_2\in \ff^*$ such that $B_1\cap B_2 = A\cap A_1\cap A_2$. Then $Y_1\cap Y_2 =\pi(A\cap A_1\cap A_2) = \pi(B_1\cap B_2)\in \pi(\ff^*|_{B_1}) = \m M$. 
\end{proof}

In the remainder of the section, we prove F\"uredi's structural theorem. First, we note that any $k$-uniform $\ff$ contains a large $k$-partite subfamily.

\begin{lem}[Erd\H os and Kleitman \cite{EKl}] Given $\ff\subset {[n]\choose k}$, it contains a $k$-partite subfamily $\ff'$ that satisfies $|\ff'|\ge \frac{k!}{k^k}|\ff|$.
\end{lem}

The proof is a simple probabilistic counting: take a random partition by including each element of the ground set in part $i$ with probability $1/k$, and calculate the expected number of edges that `respect' that partition. In what follows, we assume that $\ff$ is $k$-partite.

We now come to the heart of the proof, which is a delicate refinement argument.  Put \begin{align*}\m M(\ff)&:=\bigcup_{A\in \ff} \pi(\ff|_A),\\ \m B(\ff)&:=\{X\subset [k]: \text{there is $Y$ s.t. } \pi(Y) = X \\ &\text{\ \ \ \ \ \  and } Y \text{ is a kernel of a }\Delta(s)\text{-system } F_1,\ldots, F_s\in \ff\}\end{align*}

Note that any $X\in\m B(\ff)$ is actually an intersection of any two sets from the corresponding sunflower, and thus $X\in \m M(\ff)$. Thus, $\m B(\ff)\subset \m M(\ff)$. The proof of the theorem consists of at most $2^k$ refinement steps, where at each step we want to decrease the size of $|\m M(\ff)|$. Namely, we do the following:   at first we refine $\ff$ to $\ff'$ in order to find a set $I\in \m M(\ff)\setminus \m B(\ff')$, and then, if needed, further refine  $\ff'$ to $\ff''$ so that $I\in \m M(\ff)\setminus \m M(\ff'')$. The exact statement is as follows.
    \begin{lem}\label{lemfur1} Assume that $\ff$ does not satisfy the conditions of the theorem. Then we can find $\ff''\subset \ff$ such that $$|\ff''|\ge \frac{|\ff|}{1+s(k-1)|\m B(\ff)|}\ge \frac{|\ff|}{1+s(k-1)2^{k}}$$ and such that either $\ff''$ satisfies the conditions of the theorem or such that there is a set $I\in \m M(\ff)\setminus \m M(\ff'')$.\end{lem}
    \begin{proof}
      Let us define a partition
      $$\ff = \ff_0\cup \bigcup_{I\subset [k]}\ff_I$$
      as follows. At each iteration we go over all sets $A$ of $\ff$ and check if for every $X\subset A$ such that $\pi(X)\in \m M(\ff)$ the set $X$ is a kernel of a $\Delta(s)$-system. If it is not, then we delete this set $A$ from $\ff$ and include it into $\ff_{\pi(X)}$. At some point the procedure stops and we are left with a family $\ff_0$ in which each such $X\subset A$ is a center of an $s$-sunflower for every $A$.

      Next, we note two things. First, we claim that $\ff_0$ is a $(k,s)$-homogeneous structure. We see that $\ff_0$ satisfies conditions (2) and (4) of the theorem by definition. Let us also check that (3) holds for $\ff_0$. Suppose that, for some $A$ and $B$, there is $I\in \ \pi(\ff_0|_A)\setminus \pi(\ff_0|_B)$. Then take the set $X\subset B$ such that $\pi(X)=I$ and note that $X$ is not a kernel of an $\Delta(s)$-system $F_1,\ldots, F_s$: otherwise, by Observation~\ref{obsdelta2}, there would have been $i$ such that $F_i\cap B = X$ since $s>k$. Since $X$ is not a kernel of a $\Delta(s)$-system, such $B$ should have then been `filtered' into $\ff_{\pi(X)}$ during the initial procedure.\footnote{Note here the delicate detail of the initial procedure of `filtering' sets. We check the `kernel condition' for all $X\subset A$ such that $\pi(X)\in \m M(\ff)$, and not only, say, $\pi(X)\in \pi(\ff|_A)$.}  Therefore, $\pi(\ff_0|_A)= \pi(\ff_0|_B)$ for all $A,B$ and $\ff_0$ is as desired.

       Second, we claim that, for each $I\subset [k]$, $I\notin \m B(\ff_{I})$. Indeed, take any $X$ such that $\pi(X) = I$ and look at the first set $F$ containing $X$ that was included in $\ff_I$. At that point, we know that $X$ is not a center of a $\Delta(s)$-system in the remaining part of $\ff$. Only sets from the remainder of $\ff$ may end up in $\ff_I$, and thus $X$ will not become such a center.

       Now, for each $I\subset [k]$ and $X$ such that $\pi(X) = I$ and $\ff_I[X]$ is non-empty, we take a maximal $\Delta$-system $F_1,\ldots, F_\ell\in \ff_I$ with kernel $X$, where $\ell<s$. Any other set in $\ff_I[X]$ (the sets in $\ff_I$ containing X) must intersect one of $F_i\setminus X$, so we may take the element $v_X\in \cup F_i\setminus X$ that is contained in at least a $1/\ell k$-proportion of sets from $\ff_I[X]$. Recall that $\ff$ is $k$-partite, and thus $\ff[X_1]$, $\ff[X_2]$ are disjoint for distinct $X_1,X_2$ such that $\pi(X_i) = I$. Now put $$\ff'_I:=\bigsqcup_{X: \pi(X)=I} \ff_I[X\cup \{v_X\}].$$ We note that, first, $|\ff'_I|\ge \frac 1{(s-1)k}|\ff_I|$ and, second, $I\notin \m M(\ff'_I)$.

      Finally, we take as $\ff''$ the largest family out of  $\ff_0$ and $\{\ff_{I}'\}$. Since $\ff_0$ and $\{\ff_{I}'\}$ decompose $\ff$, we can see that one of the families in question has size at least $\frac{|\ff|}{1+s(k-1)|\m B(\ff)|}$.
    \end{proof}

After each application of Lemma~\ref{lemfur1} we either get a subfamily with desired properties, or pass from $\ff$ to $\ff''$, for which $|\m M(\ff'')|<|\m M(\ff)|$.  Repeating this at most $2^{k}$ times inevitably gives us a family with desired properties. One can easily track the value of $c(k,s)$ given by the proof.

\subsection{First applications: $(n,k,L)$-systems} In the paper \cite{Fu1}, Zolt\'an F\"uredi points out several applications to $(n,k,L)$-systems, which seems to have been the main motivation for finding and proving Theorem~\ref{thmfur}. We again assume $L = \{\ell_1,\ldots, \ell_r\}$. Recall that $m(n,k,L)$ is the size of the largest $(n,k,L)$-system. Remark that the bound $m(n,k,L) = \Omega_k(n)$ is evident for any non-empty $L$, by taking a $\Delta(n/k)$-system with core of size in $L$. He states without proof the following two theorems. The first one is as follows.
\begin{thm}\label{thmfur1}
  We have $m(n,k,L) = O_k(n)$ if and only if there is no family $\m M\subset 2^{[k]}$ that is closed under intersection, and such that $\cup_{M\in \m M} M = [k]$ and $|M|\in L$ for all $M\in \m M$.
\end{thm}
That is, he generalizes several previous results, including the result of Frankl \cite{Fra83} for $|L|=3$.  In order to see that the `if' direction of the theorem is valid, take any $(n,k,L)$-family $\ff$ and find a $(k,k+1)$-homogeneous structure $\ff^*\subset \ff$. Then by Corollary~\ref{corfur}, the family $\m M(\ff^*)$ is closed under intersection and satisfies $|M|\in L$ for all $M\in \m M(\ff^*)$. Therefore, we conclude that $\cup_{M\in \m M(\ff^*)} \ne [k]$. That is, there is a part $i$ such that each element $x$ from that part is contained in at most one set from $\ff^*$. It implies that there are at most $n$ sets in $\ff^*$. The `only if' direction is a generalization of the construction provided by Frankl \cite{Fra83} and mentioned in Section~\ref{sec25}.

The second theorem is the following reduction theorem.

\begin{thm}\label{thmfur2}
  Assume that $\ell_1,\ldots, \ell_r,k$ all have a common divisor $d$. Then we have $m(n,k,\{\ell_1,\ldots, \ell_r\}) = \Theta_k\big(m\big(\lfloor\frac nd\rfloor,\frac kd,\{\frac{\ell_1}d,\ldots, \frac{\ell_r}d\}\big)\big)$.
\end{thm}
The $\Omega_k$ part of the theorem is easy to see: take the largest $\big(\lfloor\frac nd\rfloor,\frac kd,\{\frac{\ell_1}d,\ldots, \frac{\ell_r}d\}\big)$-system $\ff$ and then replace each element of the ground set with $d$ elements. We get an $(n,k,\{\ell_1,\ldots, \ell_r\})$-system of the same size. The $O_k$ direction requires Theorem~\ref{thmfur}. Take the largest $(n,k,\{\ell_1,\ldots, \ell_r\})$-system $\ff$ and select a  $(k,k+1)$-homogeneous structure $\ff^*\subset \ff$.

Consider $\m M:=\m M(\ff^*)$ and define its {\it atoms} $P_1,\ldots, P_m$, with some integer $m$. There are disjoint subsets of $[k]$ defined by the property that the elements of each $P_i$ are indistinguishable w.r.t. $\m M$: for each $M\in \m M$ and $x\in P_i$, $x\in M$ iff $P_i\subset M$. It is easy to see that all atoms form a partition of $[k]$: $P_1\sqcup\ldots\sqcup P_m = [k]$, and that each $|P_i|$ is divisible by $d$. For the latter property, we use the fact that $\m M$ is closed under intersection.


Next, we use the intersection structure $\m M$ to analyze the structure of the family $\ff^*$. Note that if $F\in \ff^*$ and $Q\subset F$ is such that $\pi(Q) = P_i$, then any set $F'\in \ff^*$ that intersects $Q$ must contain $Q$. Therefore, the partition $P_1\sqcup\ldots\sqcup P_m = [k]$ induces a matching   $Q'\sqcup Q_1\sqcup\ldots \sqcup Q_{p}=[n]$ (with some integer $p$ that is potentially much larger than $m$) such that, first, $d$ divides $|Q_i|$ for each $i$ and, second, whenever a set $F$ intersects $Q_i$, it must contain it. In other words, elements in $Q_i$ are `indistinguishable' by sets in $\ff^*$. The set $Q'$ is the set of elements not contained in any set from the family. Now we can define a family $\m G$ as follows: replace each $Q_i$ by $|Q_i|/d$ elements and let $\m G$ be the image of the family $\ff^*$ under this transformation. We note that, by the above discussion, $\m G$ is correctly defined and is a $\big(n',\frac kd,\{\frac{\ell_1}d,\ldots, \frac{\ell_r}d\}\big)$-system for some $n'\le \frac nd$. Finally,  $|\m G| = |\m \ff^*|$.

\subsection{Families with $3$ intersection sizes} In the paper \cite{Fur85}, continuing the study of Frankl \cite{Fra83}, F\"uredi applied his structural theorem to the study of $(n,k,L)$-systems with $|L|=3$, $L = \{\ell_1,\ell_2,\ell_3\}$. For $|L|=1$ the situation has a perfect description by Deza's Corollary~\ref{corde}. For $|L|=2$ we either have $|\ff| = \Theta_k(n)$ or $|\ff| = \Theta_k(n^2)$ (the lower bound here is valid when Deza--Erd\H os--Frankl divisibility constraints hold, and is due to Babai and Frankl \cite{BF}). For $|L|=3$, however, the situation becomes much more complicated. Theorem~\ref{thmfur1} gives a characterization of the case when $|\ff| = O(n)$. We also know from Theorem~\ref{thmdef} that, unless $\ell_1|\ell_2-\ell_1|\ell_3-\ell_2|k-\ell_3$, we have  $m(n,k,L) = O_k(n^2)$. We, however, do not seem to know necessary and sufficient conditions for $m(n,k,L) = \Theta_k(n^3)$. Frankl and F\"uredi explored the range between $n$ and $n^2$ and showed that different non-integral exponents are possible.
Using that $m(n,k,\{\ell_1,\ell_2,\ell_3\}) = \Theta_k(m(n-\ell_1,k-\ell_1,\{0,\ell_2-\ell_1,\ell_3-\ell_1\}))$,  in what follows we assume for simplicity that $L = \{0,a,b\}$.

In the paper that we discussed in Section~\ref{sec25}, Frankl \cite{Fra83} showed that if we cannot find non-negative integers $\alpha,\beta$, such that $\alpha a+\beta (b-\lfloor b/a\rfloor a) =k$ and the divisibility conditions of Theorem~\ref{thmdef}  do not hold, then $m(n,k,L) = O(n^{3/2})$. At the same time, if we are not in the situation of Theorem~\ref{thmfur1}, then $m(n,k,L) = \Omega(n^{k/(k-1)})$. Frankl \cite{Fra83} also showed that $|\ff| = \Theta_k(n^{3/2})$ in some cases. At that time, there was hope that maybe $3/2$ is the only exceptional rational exponent. However, F\"uredi showed that the situation is even more complicated. We need some preparations in order to state the results.

 Let $\m M\subset 2^{[k]}$ be the prospective intersection structure,  $\m M=\{A_1,\ldots, A_f,$ $B_1,\ldots,B_g\}$, where $A_i$ and $B_j$ are the prospective intersections of sizes $a$ and $b$, respectively. We of course assume that $\m M$ is closed under intersection. Let $I = \{i\in[f]: A_i\subset B_j\text{ for at least two $B_j$'s}\}$. Let $B_j':=\{i\in I:A_i\subset B_j\}$ encode the indices of sets $A_i$ that are contained in $B_j$. Finally, form a family out of these sets: put $\m C:=\{B_j': j\in[g]\}$.
Consider the following condition.
\begin{quote} There exists $\m M$ as above and such that every pair of elements of $I$ is contained in exactly one set in $\m C$. That is, $\m C$ is a $2$-design on $I$.
\end{quote}
Note that any pair of elements of $I$ is contained in at most one set in $\m C$, otherwise, two sets of size $b$ intersect in more than $a$ elements. This condition was one of the requirements for the  $(n,k,L)$-families of size $\Theta_k(n^{3/2})$, constructed in \cite{Fra83}. F\"uredi in \cite{Fur85} shows that if this condition is not satisfied, then $|\ff| = O(n^{4/3})$ and, moreover, constructs such families of size $\Theta(n^{4/3})$ for some values of parameters.

Let us give some proof ideas. First, he passes to the  $(k,k+1)$-homogeneous structure $\ff^*$. As we assumed,  $\m M:=\m M(\ff^*)$ does not satisfy the displayed condition. Next, let $\m A:=\{F_1\cap F_2:F_1,F_2\in \ff^*, |F_1\cap F_2|=a\},$ and define $\m B$ analogously for intersection size $b$. It is not difficult to check that $|\m A|, |\m B|=O_k(n)$. E.g., sets in $\m A$ are pairwise disjoint. Then, using an argument similar to Proposition~\ref{prop77d}, he `cleans' the set system so that each set from the modified families $\m A, \m B$ is contained in at least $n^{1/3}$ sets from (the modified) $\ff^*$. (This is again the argument as in proving that a graph with many edges has a subgraph with high minimum degree.) Abusing notation, we use $\m A, \m B,\m F^*$ for these `cleaned' families. In this process, we removed $O_k(n^{4/3})$ sets. Consider the family $\m A_0\subset \m A$ of sets in $\m A$ that are contained in at least $n^{1/3}$ sets from $\m B$. A simple double-counting argument, together with the fact that sets in $\m A$ are disjoint and $|\m B|=O_k(n)$, shows that $|\m A_0| = O_k(n^{2/3})$.

Finally, we use the fact that the property is violated. Take an arbitrary $F\in  \ff^*$ and two sets $A_1,A_2\subset F$ that satisfy the following: both are contained in at least two sets from $\m B$, but there is no set $B\in \m B$ such that $A_1\cup A_2\subset B$. He then shows that, first, $A_1,A_2\in \m A_0$.\footnote{For this we need the fact that each $A_i$ is contained in at least two $B_j$'s and, ultimately, must be contained in many, since $B_j$'s have high degree.} Second, he shows that $A_1, A_2$ uniquely determines the set $F$.  Indeed, the latter is easy to see: $A_1\cup A_2$ is not contained in any intersection of sets from $\ff^*$ (in the terminology of the next section, we may call it an {\it own subset}). We conclude that $|\ff^*|\le {|\aaa_0|\choose 2}= O_k(n^{4/3})$.

\section{Forbidding just one intersection}\label{sec33}  One of the first and most prominent applications of Theorem~\ref{thmfur} is the paper of Frankl and F\"uredi \cite{FF1}, in which they solve for $n>n_0(k)$ the `forbidden one intersection problem'. Recall that the question is to determine the largest size of a family $\ff\subset {[n]\choose k}$ such that $|F_1\cap F_2|\ne \ell$. See the problem collection of Erd\H os \cite[page 18]{Erd75}. For $\ell=1$ this problem was proposed by Erd\H os and S\'os. Then Erd\H os in \cite{Erd75} suggested the generalization of this problem for all values of $\ell$. He  wrote:
\begin{quote}I just thought of (2)
while writing these lines and thus would not be surprised if
it would be completely false or at best not completely
accurate.\end{quote}
The equation (2) above stands for the two prospective extremal families. One is the family of all sets containing a given $(\ell+1)$-element set, and the other one is a (partial) Steiner system/design: the largest collection of $k$-element sets such that any two intersect in at most $\ell-1$ elements. The bound was indeed not completely accurate, but not completely false, as we shall see in Section~\ref{sec331} below.

We already mentioned the paper of Frankl \cite{Fra80a}, in which he managed to show that for $n\ge n_0(k)$, $k\ge 3\ell$ the largest $\ell$-avoiding family has size at most $c(k){n-\ell-1\choose k-\ell-1}$. Using Theorem~\ref{thmfur}, Frankl and F\"uredi managed to obtain great progress on this problem. In this section, we discuss the approximate result for all $\ell$, and in the next section we discuss the exact result for $k\ge 2\ell+2$.

\begin{thm}[Frankl and F\"uredi \cite{FF1}]\label{thmff1}
  If $\ff\subset {[n]\choose k}$ is an $(n,[0,k-1]\setminus\{\ell\},k)$-system, then
  $$|\ff| = O_k(n^{\max\{\ell,k-\ell-1\}}).$$
\end{thm}

\subsection{Examples}\label{sec331} Let us explain why the maximum in the exponent is for a good reason. This problem has very different prospective extremal examples depending on whether $k>2\ell+1$ or $k\le 2\ell+1$.

In the former case, it is of the Erd\H os--Ko--Rado type: the family of all sets containing a given $(\ell+1)$-element set. That is, it avoids all intersections smaller than $\ell$. It has size ${n-\ell-1\choose k-\ell-1}$.

In the latter case, it is based on a partial {\it design} (partial {\it Steiner system}). Consider the following partial design: a collection of $(2k-\ell-1)$-element sets such that any $\ell$-element set is contained in at most one of them.  A design is {\it complete} if each $\ell$-element set is contained in exactly one set. Any such family has size at most ${n\choose \ell}/{2k-\ell-1\choose \ell}$, and R\"odl \cite{Ro} proved that such families of size $(1+o(1)){n\choose \ell}/{2k-\ell-1\choose \ell}$ can be found using a greedy/random procedure called R\"odl Nibble Method. Keevash \cite{Kee} proved that (complete) designs exist for large $n$, provided that natural divisibility conditions hold. Back to the construction, take the $k$-shadow of such family, that is, all $k$-element sets contained in one of the  $(2k-\ell-1)$-element `blocks'. Note that any two $k$-sets inside one $(2k-\ell-1)$-element block intersect in at least $\ell+1$ elements, while sets from different blocks intersect in at most $\ell-1$ elements. The size of such construction is $(1+o(1)){n\choose \ell} {2k-\ell-1\choose k}/{2k-\ell-1\choose \ell}.$ \footnote{We may compare it with the example suggested by Erd\H os. The latter has size (at most) ${n\choose \ell}/{k\choose \ell}$, and the ratio of the sizes is $\frac 1{(k-\ell)!}\prod_{i=1}^{k-\ell}(2k-2\ell-i)$, which is strictly bigger than $1$ for $\ell\le k-2$. For $\ell=k-1$ the examples coincide.}

Let us compare the two constructions when $\ell = k-\ell-1$, i.e., when $k=2\ell+1$. The size of the latter one is $(1+o(1)){n\choose \ell} {k+\ell\choose k}/{k+\ell \choose \ell} = (1+o(1)){n\choose \ell}$, while the size of the former one is ${n-\ell-1\choose \ell}$. That is, the two constructions are roughly of the same size, but the design-based construction is potentially slightly bigger, at least when complete designs exist.

The conjecture of Erd\H os \cite{Erd75}, corrected by Frankl \cite{Fra83b}, is as follows.
\begin{conj}[Erd\H os \cite{Erd75}, Frankl \cite{Fra83b}]\label{conjff} Assume that $n\ge n_0(k)$ and $\ff\subset {[n]\choose k}$ is avoids intersection $\ell$.
  \begin{itemize}\item If $k\ge 2\ell+2$ then $$|\ff|\le {n-\ell-1\choose k-\ell-1}.$$
  \item If $k\le 2\ell+1$ then
  $$|\ff|\le (1+o(1)){n\choose \ell} {2k-\ell-1\choose k}/{2k-\ell-1\choose \ell}.$$
  \end{itemize}
\end{conj}
Actually, Frankl \cite{Fra83b} conjectured that the second part holds without the $n\ge n_0(k)$ assumption. He proved it for $k-\ell$ being a prime number. In Section~\ref{sec333} we will see the proof by Frankl and F\"uredi of the first part of the conjecture.

\subsection{Proof of Theorem~\ref{thmff1}}\label{sec322} In this section, we present the proof of Theorem~\ref{thmff1}. In the theorem, we assume that $k$ is constant. Thus, we may apply Theorem~\ref{thmfur} and select $\ff^*\subset \ff$ that is a $(k,k+1)$-homogeneous structure. In what follows, we bound the size of $\ff^*$, which then implies the bound on $|\ff|$.  Informally, passing to $\ff^*$ allows us to work with the intersection structure $\m M:=\m M(\ff^*)$ instead of the family $\ff$ itself. Let us say that the {\it rank} $r(\m M)$ of $\m M$ is the  size of the smallest set in $2^{[k]}$ not contained in any set from $\m M$. The next lemma contains the information we need concerning the structure of $\m M$.
\begin{lem}\label{lemmanyint}
  We have $r(\m M)\le \max\{\ell,k-\ell-1\}$.
\end{lem}
Given this lemma, the proof of Theorem~\ref{thmff1} is easy to complete. Indeed, recall that $\m M = \pi(\ff^*|_F)$ for any set $F\in \ff^*$. By  Lemma~\ref{lemmanyint}, the set $F$ contains a subset $X(F)$ of size  $m:=r(\m M)\le \max\{\ell,k-\ell-1\}$ that is not contained in any other $B\in \ff^*$. (We call it an {\it own subset}.) Otherwise, $B\cap F\supset X(F)$, which contradicts the definition of the rank. Therefore, there is an injection from $\ff^*$ to the sets of size $m$, and thus $|\ff^*|\le {n\choose m}$. Consequently, $|\ff|\le c(k,k+1)^{-1}  {n\choose m}\le c(k,k+1)^{-1} n^{m}$. We are left to prove Lemma~\ref{lemmanyint}. The proof of the lemma relies on  a result of Frankl and Katona \cite{FrKa} on `sets with many intersection conditions'.  As Frankl and Katona write in their paper, this problem arose in the context of a problem on database systems.

\begin{thm}[Frankl and Katona \cite{FrKa}]\label{thmmanyint}
    Let $A_1,\ldots, A_m$ be not necessarily distinct subsets of $[n]$ and assume that $|\cap_{j=1}^t A_{i_j}|\ne t-1$ for any distinct $i_1,\ldots, i_t$. Then $m\le n$.
\end{thm}
At the first glance, the theorem has an unusual condition for an extremal set theoretic question. But, actually, it is kind of a Hall's condition in disguise. We shall comment on this in the proof.

The theorem is sharp: the bound is attained on $A_1 = \ldots = A_m = [n]$, in which case all intersections are larger than the right hand side of the condition. This is, however, not the only possibility. In \cite{FF1} Frankl and F\"uredi discuss well-intersection designs: other families that could attain equality in Theorem~\ref{thmmanyint}.

\begin{proof} It is convenient for us to think of a (containment) bipartite graph $G$ between sets $A_i$ on one side and elements of the ground set $[n]$ on the other side. For an element $x\in[n]$, let us denote its neighborhood in $G$ by $N(x) = \{i: x\in A_i\}$ and its degree by $d(x)  =|N(x)|$. We essentially want to show is there is a matching covering the set part. In order to do that, we may try to verify Hall's conditions. The condition in the theorem for $t=1$ gives that all sets are non-empty. For $t=2$, it prevents two sets to intersect in exactly $1$ element, which excludes the only possible violation of Hall's condition for $2$ sets: that both sets consist of the same singleton. Continuing further this way, however, becomes complicated.

It is convenient to prove the theorem by induction on $m$. For any $m'<m$, we can apply inductive assumption to any subcollection of $m'$ sets, and it then implies that the union of any $m'$ sets $A_i$ has size at least $m'$. That is, Hall's conditions for all proper subsets of $\{A_1,\ldots, A_m\}$ are verified. Arguing indirectly, assume that $m>n$. By induction for $m'= m-1$, the union of the first $m-1$ sets has size at least $m-1$, but it is at most $n$, and thus $m = n+1$. Then there is a matching covering the element side in $G$ and, say $A_1,\ldots, A_n$. We may w.l.o.g. assume that $i$ is matched to $A_i$.

Below we show that for each $A_i$ and $x\in A_i$ we have $d(x)\le |A_i|$. This concludes the proof, since, double counting the edges of $G$, we have $\sum_{i=1}^m|A_i| = \sum_{x=1}^n d(x)$, but $d(i)\le |A_i|$ for all $i\in[n]$. It implies that $|A_{m}| = 0$, which is impossible. 

In order to prove $d(x)\le |A_i|$, we again apply induction, but this time we change the sets to essentially their intersections with $A_i$. Take distinct $j_1,\ldots, j_t\in  N(x)\setminus \{i\}$ and observe that $$\Big|A_i\cap \bigcap_{\ell=1}^t (A_i\cap A_{j_\ell})\Big|\ne t$$
by assumption. Removing $A_i$ from the list and $x$ from the sets, we get
$$\Big|\bigcap_{\ell=1}^t (A_i\cap A_{j_\ell}\setminus \{x\})\Big|\ne t-1.$$
We see that the sets $A_i\cap A_j\setminus \{x\}$  for $j\in N(x)\setminus \{i\}$    satisfy the condition of the theorem. Now apply inductive hypothesis to these sets and get that $|N(x)\setminus \{i\}|\le |A_i\setminus \{x\}|$, which is equivalent to $d(x)\le |A_i|$.
\end{proof}
The proof we presented is essentially the proof of Frankl and Katona. Frankl and F\"uredi \cite{FF1} give a different proof.

\begin{cor}\label{cormanyint}
   Let $r\ge 1$ be an integer and $A_1,\ldots, A_m$ be not necessarily distinct subsets of $[n]$. Assume that $|\cap_{j=1}^t A_{i_j}|\ne t-1-r$ for any distinct $i_1,\ldots, i_t$. Then $m\le n+r$.
\end{cor}

The corollary is again tight: again, take all $A_i$ to be equal to $[n]$. Interestingly, in contrast to $r=0$ case, for $r\ge 1$ this is the only extremal example. We shall need this fact later.
\begin{proof}
  We modify the sets: consider  $(\cap_{i=1}^r A_i)\cap A_j$ for $j=r+1,\ldots,m$. These sets satisfy the assumption of Theorem~\ref{thmmanyint}, and thus their number is at most $n$, i.e., $m-r\le n$.
\end{proof}

We are now ready to prove Lemma ~\ref{lemmanyint}.

\begin{proof}[Proof of Lemma~\ref{lemmanyint}]
  Take a minimal subset $B$ not contained in any set from $\m M$ and assume that $|B|>\ell$ holds. Assume that $B = \{x_1,\ldots, x_b\}$ and for each $i\in [b]$ define $B_i:=B\setminus \{x_i\}$. By minimality of $B$, $B_i=B\cap A_i$ for some $A_i\in \m M$.  Put $Y = [k]\setminus B$ and define $D_i = Y\cap A_i$.

  We claim that for any $i_1<\ldots<i_t$ we have
  \begin{equation}\label{eqmanyint}|D_{i_1}\cap\ldots \cap D_{i_t}|\ne t-(b-\ell).\end{equation}
  Indeed, otherwise $A = A_{i_1}\cap\ldots \cap A_{i_t}$ satisfies $|A\cap B| = b-t$ and $A\cap Y = t-(b-\ell)$, implying $|A| =\ell$. This is a contradiction, since $\ff$ avoids intersection $\ell$ and thus $\m M$ has no sets of size $\ell$.

  Finally, condition \eqref{eqmanyint} allows us to apply Corollary~\ref{cormanyint} to sets $D_i$ with $r=b-\ell-1$, $n=|Y| =k-b$ and $m=b$. (Note that $r\ge 0$ because $b>\ell$ by assumption.) We conclude that $b\le (k-b)+b-\ell-1$, that is,  $b\le k-\ell-1$. This concludes the proof.
\end{proof}

\subsection{A sharp result in the case $k>2\ell+1$}\label{sec333} Frankl and F\"uredi also managed to prove an exact result in the case when $k>2\ell+1$, i.e., when the tentative extremal family consists of all sets containing a given $\ell+1$-element set.
\begin{thm}[Frankl and F\"uredi \cite{FF1}]\label{thmff2}
  If $n>n_0(k)$, $k\ge 2\ell+2$ and $\ff\subset {[n]\choose k}$ avoids intersection $\ell$, then
  $$|\ff| \le {n-\ell-1\choose k-\ell-1}.$$
\end{thm}
The proof follows the same logic, but requires several refinements. For some technical reasons, it is easier to leave out the case $k = 2\ell+2$ and treat the case $k\ge 2\ell+3$. We sketch the proof in that assumption.

The first step is a refinement of Corollary~\ref{cormanyint}, which states that for $r\ge 1$ equality $m=n+r$ is only possible if the sets $A_1,\ldots, A_m$ are all equal to the ground set $[n]$. It is essentially done by carefully running the proof of Theorem~\ref{thmmanyint}, and we skip it for conciseness.\footnote{This is the part where it is helpful to assume that $k\ge 2\ell+3$. In this regime, we apply the corollary, while for $k=2\ell+2$, we apply Theorem~\ref{thmmanyint} itself. Thus, in the proof for $k\ge 2\ell+3$ we can rely on the characterisation of the extremal families. But no such characterisation is valid for  Theorem~\ref{thmmanyint}, and thus for $k=2\ell+2$.} Arguing as in the proof of Lemma~\ref{lemmanyint}, it implies the following lemma:

\begin{lem}\label{lemrankref} In the setting of Lemma~\ref{lemmanyint} and for $k\ge 2\ell+3$, we have $r(\mathcal M)\le k-\ell-2$, unless $\mathcal M$ has the following special form: there is a set $C$, $|C| = \ell+1$, such that $\mathcal M\supset \{F: C\subset F \subset [k]\}$. (In which case, $r(\mathcal M) = k-\ell-1$.)
\end{lem}

Let us comment on the proof. Looking into the proof of Lemma~\ref{lemmanyint} and using equality characterization of Corollary~\ref{cormanyint}, we actually see that, for $b=k-\ell-1$ to hold, all $D_i$'s must be equal to $Y$, and $Y$ plays the role of $C$ from the lemma. We also note that the `own' subset $B$ is the subset complementary to $Y$.

Equipped with this lemma, we take the largest $\ell$-avoiding family $\ff\subset {[n]\choose k}$ and do the following iterative decomposition: find a  $(k,k+1)$-homogeneous structure $\ff_1^*\subset \ff$, remove it, then find another such structure $\ff_2^*$ in $\ff\setminus \ff_1$ etc. We stop at step $m$ when the rank of $\ff_m^*$ is at most $k-\ell-2$. As in the proof of Theorem~\ref{thmff1}, it implies that $|\ff_m^*|\le n^{k-\ell-2}$, and thus
\begin{equation}\label{eqrem1}|\ff\setminus (\ff^*_1\cup \ldots\cup \ff^*_{m-1})|\le C_k n^{k-\ell-2},\end{equation}
that is, an order of magnitude smaller than $\ff$. Moreover, using Lemma~\ref{lemrankref}, we know  that $\mathcal M(\ff_i^*)$ contains all supersets of $C_i$ for some $C_i$ of size $\ell+1$.

For any $F\in \ff_i^*$, recall the definition of an own subset $X(F)$ from Section~\ref{sec322}: $X(F)$ is of size $k-\ell-1$ and  is not contained in $F\cap F'$ for $F'\in \ff_i^*$, $F'\ne F$.  Let us also define $C(F) = F\setminus X(F)$. Recall that the own subset and the center $C$ from Lemma~\ref{lemrankref} are complementary, and thus $C(F)$ projects into $C_i\in \mathcal M(\ff_i^*)$. In particular, for any set $A$, such that $C(F)\subset A\subsetneq F$, there is a $\Delta(k+1)$-system in $\ff_i^*$ with center $A$. We use this property to prove the following proposition.

\begin{prop}\label{propglob1}
  For any $F\in \ff_i^*$, $i\in [m-1]$ and $F'\in \ff$ we have $X(F)\not\subset F'$.
\end{prop}
That is, an own subset of $F$ within $\ff_i^*$ is `globally' own.
\begin{proof}
  Arguing indirectly, assume that $X(F)\subset F'$. Then $C(F)\not\subset F'$, and thus $|C(F)\cap F'|\le \ell$. At the same time, $|F'\cap F|\ge |X(F)|\ge k-\ell-1\ge \ell+1$.  Thus, we may choose a set $A$, $C(F)\subset A\subsetneq F$ such that $|A\cap F'| = \ell$. By the definition of a $(k,k+1)$-homogeneous system, $A$ is the center of a $\Delta(k+1)$-system in $\ff_i^*$, and thus we may find a set $F''\in \ff_i^*$ such that $F''\cap F' = A\cap F'$. This contradicts the fact that $\ff$ avoids intersection $\ell$.
\end{proof}
Actually, using the same argument, we can prove the following proposition, which we shall need a bit later.

\begin{prop}\label{propglob2}
  For any $F\in \ff_i^*$, $i\in [m-1]$ and $F'\in \ff$ the following holds. If $|F\cap F'|\ge \ell+1$ then $C(F)\subset F'$ holds. If additionally $F'\in \ff_j^*$ with $j\in [m-1]$ then we have $C(F)=C(F')$.
\end{prop}
\begin{proof}
  The first part is verbatim the same argument as in Proposition~\ref{propglob1}. For the second part, we note that if $C(F)\ne C(F')$ then $|C(F)\cap C(F')|\le \ell$, but $|F\cap F'|\ge \ell+1$, and thus we can find sets $A, A'$ such that $C(F)\subset A\subsetneq F$,  $C(F')\subset A'\subsetneq F'$, such that $|A\cap A'|=\ell$. Both are cores of $\Delta(k+1)$-systems in $\ff$. Then we use the same $\Delta$-system argument and find two sets $G,G'\in \ff$ that intersect in $\ell$ elements, a contradiction.
\end{proof}

Proposition~\ref{propglob1} is enough to prove an asymptotic version of Theorem~\ref{thmff2}. Since $X(F)$ are distinct for all $F\in \ff_1^*\cup\ldots\cup \ff_{m-1}^*$, we have, using \eqref{eqrem1}
\begin{equation}\label{eqff14}|\ff|\le |\ff_1^*\cup\ldots\cup \ff_{m-1}^*|+O(n^{k-\ell-2})\le {n\choose k-\ell-1}+O(n^{k-\ell-2}).\end{equation}
In order to prove the exact result, we need to continue our bootstrapping argument. Slightly abusing notation, let $C_1,\ldots, C_h$ be all sets such that $C_i = C(F)$ for some $F\in \ff_1^*\cup\ldots\cup \ff_{m-1}^*$. Define
$$\m G_i:= \{G\in \ff_1^*\cup\ldots\cup\ff_{m-1}^*: C_i\subset G\}.$$ Also, for a family $\m G$ and integer $m$, recall the notion of the {\it shadow on level $m$}: $\partial^{(m)}\m G:= \cup_{G\in \m G}{G\choose m}$. Proposition~\ref{propglob2} immediately implies the following property:
\begin{prop}\label{propglob3}
  The families $\partial^{(\ell)}(\m G_i)$ are disjoint.
\end{prop}
Consider the families $\m G_i(C_i)= \{G\setminus C_i: G\in \m G_i\}$ of $(k-\ell-1)$-element sets. We note that $\partial^{(\ell)}(\m G_i(C_i))\subset\partial^{(\ell)}(\m G_i)$, and thus the last proposition implies 
\begin{equation}
\label{eqshadowsum} \sum_{i\in [h]}|\partial^{(\ell)}(\m G_i(C_i))|\le {n\choose \ell}.
\end{equation}
However, since $\ff$ is extremal and $\m G_i$ decompose a bulk of $\ff$, we must have $$\sum_{i\in [h]}|\m G_i(C_i)|\ge {n\choose k-\ell-1}-O(n^{k-\ell-2}).$$
Recall the Kruskal--Katona theorem.
\begin{thm}[Kruskal, Katona \cite{Kr,Ka}, in Lov\'asz' form] \label{thmkk} Given a family $\mathcal X\subset {[n]\choose k}$ with $|\m X| = {x\choose k}$ for some real $x\ge k$, we have $|\partial^{(\ell)}\m X|\ge {x\choose \ell}$.\end{thm}
We then apply the Kruskal--Katona theorem to $\m G_i(C_i)$ and conclude that the last two displayed inequalities are compatible only if one of the families, say $\m G_1$, contains a bulk of the sets: if $|\ff\setminus \m G_1|= O(n^{k-\ell-2})$. Thus, we discern an approximate structure of the extremal family: almost all sets in $\ff$ must contain a given $(\ell+1)$-element set $C_1$, i.e., lie in $\m G_1$. We present this argument explicitly and in a stronger form in Section~\ref{sec6}.

In the remainder of the proof we need to show that for extremal $\ff$, there are no other sets but those that contain $C_1$. It actually requires a non-trivial effort. Let us partition $\ff$ into three parts $\ff = \mathcal A\cup \mathcal B\cup \mathcal K$, where \begin{align*}\mathcal K&:=\{F\in \ff[C_1]: \text{ every set }B\text { s.t. } C_1\subset B\subsetneq F \\ &\ \ \ \ \ \text{ is a kernel of a }\Delta(k+1)-\text{sunflower in } \ff\},\\
\m A&:=\ff[C_1]\setminus \m K,\\
\m B&:=\ff\setminus \ff[C_1].
\end{align*}
Clearly, $\m G_1\subset \m K$, and thus $\m A\cup \m B$ are small. The goal is to show that if $\m A\cup \m B$ is non-empty then $\ff$ is smaller than the (supposedly) extremal example $\ff_0$ of all sets containing $C_1$. There are two cases to consider, and in both cases we shall use the arguments concerning shadows in the spirit of the argument above. The idea is to show that, first, the shadows of $\m K$ and one of $\m A, \m B$ must be disjoint and, second, to show that the shadows of $\m A$ or $\m B$ are proportionally much larger than the corresponding shadow of $\m K$.

First, assume that $|\m B|\ge |\m A|$. Using Proposition~\ref{propglob2}, we conclude that $\partial^{(\ell)}(\m B)$ and $\partial^{(\ell)}(\m K)$ are disjoint, and thus $|\partial^{(\ell)}(\m B)| +|\partial^{(\ell)}(\m K)|\le {n\choose \ell}$. Looking at the extremal example $\ff_0$, we note that $|\partial^{(\ell)}\ff_0|/|\ff_0| = {n\choose \ell}/{n-\ell-1\choose k-\ell-1}$, and using the Kruskal--Katona theorem, it is not difficult to show  $|\partial^{(\ell)}\m K|/|\m K| \ge {n\choose \ell}/{n-\ell-1\choose k-\ell-1}$. If we manage to show that \begin{equation}\label{eqsizeb}|\partial^{(\ell)}\m B|/|\m B| \gg {n\choose \ell}/{n-\ell-1\choose k-\ell-1},\end{equation} then we get the following chain of inequalities:
$$|\m A|+|\m B|+|\m K|\le 2|\m B|+|\m K|<\frac{{n-\ell-1\choose k-\ell-1}}{{n\choose \ell}}\left(|\partial^{(\ell)}\m B|+|\partial^{(\ell)}\m K|\right)\le {n-\ell-1\choose k-\ell-1}.$$
In order to show \eqref{eqsizeb}, we first note that $\m B$ is $\ell$-avoiding, and so we can use (the proof of) Theorem~\ref{thmff1} in order to get $|\m B|\le C_k |\partial^{(k-\ell-1)}\m B|$ (by a bijection from a homogeneous structure $\m B^*\subset \m B$ into own subsets of size $k-\ell-1$). Then we remark that $\m B$, and consequently $\partial^{(k-\ell-1)}\m B$, are  small (both have size $O(n^{k-1})$), and, when writing down the size of $|\partial^{(k-\ell-1)}\m B|$ as ${x\choose k-\ell-1}$, we see that $x<n^{1-1/k}$. Thus, using the Kruskal--Katona theorem, we conclude that the ratio |$\partial^{(\ell)}\m B|/|\partial^{(k-\ell-1)}\m B|$ is much bigger  than  ${n\choose \ell}/{n-\ell-1\choose k-\ell-1}$. Combining both observations, we get  \eqref{eqsizeb}.

Next, assume that $|\m A|\ge |\m B|$. We start by selecting a $(k,k+1)$-homogeneous system $\m A^*\subset \m A$ with $|\m A^*|\ge c_k|\m A|$. For each set  $F\in \m A^*$ we find an own\footnote{not contained in other sets from $\m A^*$} subset $H(F)$ of size $k-1$ such that $H\notin \partial^{(k-1)}\m K$.\footnote{To see this, take any $H'$, $C_1\subset H'\subset F$, that does not belong to $\m M(\m A^*)$. Then $H'$ and none of its supersets can belong to the shadow of $\m K$. We only have to pick a superset of size $k-1$ that is not in $\m M(\m A^*)$. It is easy to see that there is one.}   The family $\m H:=\{H(F):F\in \m A^*\}$ has the same size as $\m A^*$ and, moreover, all sets in $\m H$ and $\m K$ contain $C_1$. Now we look at the families $\m H(C_1), \m K(C_1)$ and note that $\m H(C_1)$ and $\partial^{(k-\ell-2)}\m K(C_1)$ are disjoint by the construction of $\m H$. Thus, the sum of their sizes is at most ${n-\ell-1\choose k-\ell-2}$. But, while $|\partial^{(k-\ell-2)}\m \ff_0(C_1)|/|\m \ff_0| = (k-\ell-1)/(n-k-1)$ and $|\partial^{(k-\ell-2)}\m K(C_1)|/|\m K| \ge (k-\ell-1)/(n-k-1)$, we have
$|\partial^{(k-\ell-2)}\m A(C_1)|/|\m A|\ge |\m H|/|\m A| \ge 1/C_k$. Then we can conclude that $|\ff|<|\ff_0|$ using a chain of inequalities similar to the last display.

\section{Exact solution of some Tur\'an-type problems}\label{sec4}
The next joint paper of Frankl and F\"uredi \cite{FF2} on the subject is a pinnacle of the applications of the $\Delta$-system method. From the $\Delta$-system method point of view, it is in a sense harvesting the fruits of their previous work \cite{FF1} and delineating the scope of applications of the method. But even besides the applications of the $\Delta$-system method, it contains many influential ideas that affected the course of a big part of extremal set theory. In the paper, they solve (for $n\ge n_0(k)$) several hard and general extremal set theoretic (or Tur\'an-type, as the name of the paper suggests) problems. In this section, we go over some of the contributions of the paper that are related to the $\Delta$-system method.

\subsection{$\Delta$-systems with a fixed kernel}\label{sec41} The following question was asked by Duke and Erd\H os \cite{DE}: what is the largest size $f(n,k,\ell, s)$ of the family $\ff\subset {[n]\choose k}$ that does not contain a $\Delta(s+1)$-system with kernel of size $\ell$? For $s=1$ this is the forbidden one intersection problem that we discussed before, and for $\ell=0$ this is the famous Erd\H os Matching Conjecture \cite{E} (see \cite{FK21} for a survey). One general question that they asked is whether it is true if $f(n,k,\ell,s)=O_{k,s} (f(n,k,\ell,1))$.

An observation of F\"uredi from \cite{Fu1} is that Theorem~\ref{thmfur} immediately implies that.  Take any family $\ff$ that does not contain a $\Delta(s+1)$-system with kernel of size $\ell$ and select a $(k,s+1)$-homogeneous system $\ff^*$ in $\ff$. If $|F_1\cap F_2|=\ell$ for some sets $F_1,F_2\in \ff^*$, then there is a $\Delta(s+1)$-system with kernel $F_1\cap F_2$, a contradiction with the choice of $\ff$. That is, $\ff^*$ avoids intersection $\ell$. At the same time, as Theorem~\ref{thmfur} guarantees, $|\ff^*|\ge c(k,s+1) |\ff|$ for some positive constant $c(k,s+1)$ that depends on $k,s$ only.

That is, in view of Theorem~\ref{thmff1}, we have that $$f(n,k,\ell,s) = O_k(\max\{n^{k-\ell-1}, n^{\ell}\}).$$ Frankl and F\"uredi prove a stronger result for $k\ge 2\ell+3$. Recall that $\phi(k,s)$ is the largest number of sets in a family of $k$-element sets that does not have a $\Delta(s+1)$-system.
\begin{thm}[Frankl and F\"uredi \cite{FF2}]\label{thmff3}
For fixed $k,s$ and $k\ge 2\ell+3$ we have
$$f(n,k,\ell,s) = (\phi(\ell+1,s)+o(1)){n-\ell-1\choose k-\ell-1}.$$
\end{thm}
\subsubsection{Examples}\label{sec411} The examples in this section generalize the examples from Section~\ref{sec331}. The example that shows tightness of the bound in Theorem~\ref{thmff3} is as follows. Take a family $\m H$ of $(\ell+1)$-element sets that has size $\phi(\ell+1,s)$ and has no $\Delta(s+1)$-system. Let $[m]$ be the ground set of $\m H$ and note that it has constant size, assuming $k,s$ are constant. Then consider the family of sets
$$\mathcal A_1:=\Big\{F\in {[n]\choose k}: F\cap [m]\in \m H\Big\}.$$
First, we have $|\mathcal A_1| = |\m H|{n-m\choose k-\ell-1}$, which has the same asymptotic as the bound in Theorem~\ref{thmff3} (thanks to the fact that $m$ is constant). Second, assuming that $F_1,\ldots, F_{s+1}$ form a $\Delta(s+1)$-system with kernel of size $\ell$, we get that $F_1\cap [m],\ldots, F_{s+1}\cap [m]$ are distinct and should form a $\Delta(s+1)$-system. But these are elements of $\m H$, and $\m H$ has no $\Delta(s+1)$-systems. Therefore, there are no $\Delta(s+1)$-systems with kernel of size $\ell$ in $\m A_1$.

The second example is as follows: Take a partial Steiner system of $((s+1)(k-\ell)+\ell-1)$-element sets such that any $\ell$-element set is contained in at most $1$ of them, and consider all $k$-element sets that are contained in one of the sets from the Steiner system. This is the family $\m A_2$. It has size
$$\frac{\Big({(s+1)(k-\ell)+\ell-1\choose k}+o(1)\Big){n\choose \ell}}{{(s+1)(k-\ell)+\ell-1\choose \ell}}.$$
To see that it has no $\Delta(s+1)$-systems with kernel of size $\ell$, note that, if one exists, then all sets must come from the same $((s+1)(k-\ell)+\ell-1)$-element set of the Steiner system. But such a $\Delta$-system spans $\ell+(s+1)(k-\ell)$-elements, and thus cannot fit inside one $((s+1)(k-\ell)+\ell-1)$-element set.

Frankl and F\"uredi conjectured that one of these examples is (approximately) best possible.
\begin{conj}[Frankl and F\"uredi \cite{FF2}] For $k,\ell, s$ fixed the following holds.
\begin{itemize}\item If $k\ge 2\ell+1$ then $$f(n,k,\ell,s)\le (\phi(\ell+1,s)+o(1)){n-\ell-1\choose k-\ell-1}.$$
\item If $k\le 2\ell$ then
$$f(n,k,\ell,s)\le \frac{\Big({(s+1)(k-\ell)+\ell-1\choose k}+o(1)\Big){n\choose \ell}}{{(s+1)(k-\ell)+\ell-1\choose \ell}}.$$
\end{itemize}
\end{conj}
\subsubsection{Proof of Theorem~\ref{thmff3}}
The proof of the theorem is mostly a simplified version of the proof of Theorem~\ref{thmff2}, and we assume that the reader is familiar with the first page of that proof. We start with the same decomposition of $\ff= \ff_1^*\sqcup \ff_2^*\sqcup\ldots$ into $(k,(s+1)k)$-homogeneous structures.\footnote{The parameters of the F\"uredi structure is the only difference between the first part of the two proofs.} We stop when $\ff_m^*$ is either empty or has rank at most $k-\ell-2$. Since we are only proving an asymptotic result, we may ignore this remainder and w.l.o.g. assume that $\ff = \ff_1^*\sqcup\ldots \sqcup \ff_{m-1}^*$. Now each set $F\in \ff_i^*$ has an own subset $X(F)$ of size $k-\ell-1$. At the same time, the set $C(F) = F\setminus X(F)$ by Lemma~\ref{lemrankref} has the property that, for any $A$ such that $C(F)\subset A\subsetneq F$, there is a $\Delta((s+1)k)$-system in $\ff_i^*$ with center $A$. 

It is in this last part that the proofs are somewhat different. For each set $X\in {[n]\choose k-\ell-1}$ define the family $\mathcal G_X:=\{F\setminus X(F): F\in \ff, X(F) = X\}$ of truncated sets, whose own subset is $X$. These are  sets that have the form $C(F)$ for the corresponding $F$.  Note that each $\ff^*_i$ contributes at most one set to $\m G_X$. Each set in $\ff$ has an own subset, and so we have $$|\ff| = \sum_{X\in {[n]\choose k-\ell-1}} |\m G_X|.$$ In order to complete the proof of the theorem, it is sufficient to show that $|\m G_X|\le \phi(\ell+1,s)$ for each $X$. Note that $\m G_X$ consists of $(\ell+1)$-element sets.

Assume that $|\m G_X|> \phi(\ell+1,s)$ for some $X$. Take a $\Delta(s+1)$-system $G_1,\ldots, G_{s+1}\in \m G_X$ and assume that its kernel has size $t$. Note that $t\le \ell$. Take an arbitrary subset $Y\subset X$ of size $\ell-t$. For each $i\in[s+1]$ we have $G_i = C(F)$ for some $F\in \ff$, and thus $G_i\cup Y$ is a kernel of a $\Delta((s+1)k)$-system. Moreover, $|(G_i\cup Y)\cap (G_j\cup Y)|=\ell$. Now, say, using Lemma~\ref{lemdeltadef}, for each $i\in [s+1]$ we can find  $F_i\in \ff$, such that $F_i\supset G_i\cup Y$ and such that $F_i\cap F_j = (G_i\cup Y)\cap (G_j\cup Y)$. They form  a $\Delta(s+1)$-system with kernel of size $\ell$ in $\ff$, a contradiction.

\subsection{Expanded hypergraphs}\label{sec42}
Let $\m A$ be an arbitrary family of $k$-sets. We call $K(A):=\cap_{A\in \m A}A$ the {\it kernel} of $\m A$. Let the center $C(\m A)$ be the set of all elements that are contained in at least two sets in $\m A$. Note that $K(\m A)\subset C(\m A)$. For a set $Y$, let $\m A|_{Y}$ stand for the {\it trace} of $\m A$ on a set $Y$, that is, $\m A_Y:=\{A\cap Y: A\in \m A\}$.

In the following papers on the subject, a slightly different perspective was taken on $\m A$.
\begin{defn}[Expanded hypergraphs]\label{defexp}
Given a set family (a hypergraph) $\m H= \{H_1,\ldots, H_m\}$ of $(\le k)$-element sets, a family of $k$-element sets $\m A = \{A_1,\ldots, A_m\}$ is an {\em expansion} (or {\em $k$-expansion}) of $\m H$, if $A_i$ is obtained from $H_i$ by adding several elements so that, first, $|A_i|=k$ and, second, all added elements are pairwise distinct and distinct from $\cup \m H$.
\end{defn}
That is the family $\m A$ above is a $k$-expansion of the family $\m A|_{C(\m A)}$. The most well-studied case is when $\m H$ is a graph, i.e., when all sets have uniformity~$2$.

Frankl and F\"uredi studied the general quantity
$$ex(n, \m A) = \max\Big\{|\ff|: \ff\subset {[n]\choose k}: \ff \text{ contains no copy of }\m A\Big\}.$$
While many papers at the time suggested very general extremal quantities as the one above, rarely there were general results proved concerning them. Frankl and F\"uredi, however, managed to prove one. The first big issue to resolve before proving the extremal result is to decide on the prospective optimal construction.

\subsubsection{Examples}
Suppose that $|K(\m A)| = p$, $|C(\m A)|=q$.  One example of a family avoiding $\m A$ is the family of sets containing a fixed $(p+1)$-element set. This is, however, suboptimal in many cases. Here, they introduced the very influential notion of a {\it $t$-crosscut}.
\begin{defn}[$t$-crosscut]\label{defcc} A set $Y$ is a {\it $t$-crosscut} of $\m A$ if $|A\cap Y|=t$ for all $A\in\m A$.
\end{defn}
Note the difference between $1$-crosscut and a cover. The set $K(\m A)$ is a $p$-crosscut. We will, however, be interested in $(p+1)$-crosscuts. Let us now introduce the key quantity.
 We note that it is not guaranteed that a $(p+1)$-crosscut exists for $\m A$, but if $\m A$ is $k$-partite, then it does: we may take the union of any $p+1$ parts. We are discussing Tur\'an-type problems in the sparse regime, and so we can limit ourselves to the $k$-partite case. Indeed, if the forbidden hypergraph is not $k$-partite, then we can take a complete $k$-partite hypergraph as an example of the hypergraph that avoids it. It has $(n/k)^k = \Theta(n^k)$ edges. Another scenario in which a $(p+1)$-crosscut always exists if $\m A$ is an expanded hypergraph: if each edge has an `own' vertex of degree $1$. We may then form a $(p+1)$-crosscut by taking the union of $K(\m A)$ and these `own' vertices for each edge.

 Let $\sigma_{p+1}(\m A)$ be the maximum size of a family of $(p+1)$-element sets that avoids copies of the trace $\m A|_Y$ of any $(p+1)$-crosscut $Y$ of $\m A$, if such a maximum exists.\footnote{Note that $\sigma_1(\m A)$ is one smaller than the size of the smallest $1$-crosscut. That is, the extremal family in this case is just the collection of $\sigma_1(\m A)$ singletons.} \vskip+0.1cm

 {\bf Example. } If $k\ge p+q$, a $k$-expanded hypergraph $\m A$ has many $(p+1)$-crosscuts. Importantly, any $\Delta(|\m A|)$-system of sets of size $p+1$ is a crosscut. This is due to the fact that each edge of $\m A$ has at least $p$ `own' vertices.   This implies that $\sigma_{p+1}(\m A)\le \phi(p+1,|\m A|-1)$.
 \vskip+0.1cm

Equipped with this definition and assuming that $\sigma_{p+1}(\m A)$ is finite, we can give an example of an $\m A$-free family in ${[n]\choose k}$, similar to the example for Theorem~\ref{thmff3}. Let $\m B\subset {[n]\choose p+1}$ be the family that achieves equality in the definition of $\sigma_{p+1}(\m A)$, and let $Y = \cup \m B$. Note that $|\m B|$ is finite, and thus $|Y|$ is independent of $n$.  Put
$$\ff_{\m B}= \Big\{F\in {[n]\choose k}: F\cap Y\in \m B\Big\}.$$
Then, first, $|\ff| = (\sigma_{p+1}(\m A)+o(1)){n\choose k-p-1}$ and, second, $\ff_{\m B}$ is $\m A$-free. Indeed, to see this, assume that $\{F_1,\ldots, F_s\}$ form a copy of $\m A$ in $\ff_{\m B}$. Since $|Y\cap F_i|=p+1$, it implies that $Y$ is a $(p+1)$-crosscut for $\{F_1,\ldots, F_s\}$, and thus a $(p+1)$-crosscut for $\m A$, a contradiction with the definition of $\m B$.

\subsubsection{The result for $k\ge 2p+q+2$} Frankl and F\"uredi showed that the example above is essentially optimal in {\it intersection condensed} case: when $k$ is somewhat large w.r.t. $C(\m A)$. Recall that $|K(\m A)| = p$, $|C(\m A)|=q$.

\begin{thm}[Frankl and F\"uredi, \cite{FF2}]\label{thmff4} Suppose that $k\ge 2p+q+2$. Then $ex(n,\m A) = (\sigma_{p+1}(\m A)+o(1)){n-p-1\choose k-p-1}$.
\end{thm}

The proof of this theorem is very similar to the proof of Theorem~\ref{thmff3}. They are actually presented in parallel in that paper. One necessary ingredient is a substitute for Lemma~\ref{lemrankref}. In order to get prove this analogue of Lemma~\ref{lemrankref}, Frankl and F\"uredi need a possibly suboptimal bound $k\ge 2p+q+2$. The conclusion of the lemma is the same: for a bulk of the sets $F\in \ff$, there is a set $C(F)$ with the property that, for any $A$ such that $C(F)\subset A\subsetneq F$, there is a $\Delta(k^2)$-system in $\ff_i^*$ with center $A$. We then do the same double-counting over possible sets $X(F)$ and show that, if $\m G_X$ is larger than $\sigma_{p+1}(\m A)$ and, consequently, contains a trace of a $(p+1)$-crosscut of $\m A$, then we can complete it to a copy of $\m A$ in $\ff$ by first embedding the rest of $C(\m A)$ into  $C(X)$ and then completing with disjoint sets using $\Delta$-systems. For this, $C(X)$ must be somewhat large, but $|C(X)|\ge |C(\m A)|$ is enough, which requires a bound $k\ge p+q+1$.

\subsection{Simplices}\label{sec43}
The following definition generalizes the notion of a triangle in a graph.
\begin{defn}\label{defsim}
  A {\it $d$-simplex} is a collection of $d+1$ sets $F_1,\ldots, F_{d+1}$, such that $\cap_{i\in[d+1]} F_i=\emptyset$, but for any proper subset $S\subsetneq [d+1]$ we have $\cap_{i\in S} F_i\ne \emptyset$.
\end{defn}
By analogy with the graph case, a $2$-simplex is called a {\it triangle}. A natural example of a family avoiding any simplex is a star: a family of all sets containing a fixed element. Erd\H os \cite{E4} conjectured that for $n\ge 3k/2$ and $k\ge 3$, any family that contains no triangle, has size at most ${n-1\choose k-1}$. One year later, Chv\'atal \cite{Chv} generalized this as follows.
\begin{conj}[Chv\'atal \cite{Chv}]\label{conjchv} Suppose that $\ff\subset {[n]\choose k}$, $\ff$ contains no $d$-dimensional simplex, $k>d$,  and $n\ge (d+1)k/d$. Then $|\ff|\le {n-1\choose k-1}$.
\end{conj}
Let us denote $sim(n,k,d)$ the largest family as in the conjecture. The bound in the conjecture in these terms states that $sim(n,k,d)\le {n-1\choose k-1}$.
 Chv\'atal himself proved the conjecture for $d=k-1$. Prior to the paper \cite{FF2}, Frankl \cite{Fra76} proved the conjecture for $n\le dk/(d-1)$, and for $n\ge n_0(k)$ in the following cases \cite{Fra81}: $d=2$, $k\ge 5$; $k\ge 3d+1$. He also proved the bound $sim(n,k,d)\le {n-1\choose k-1}+c_kn^{d-2}$ for any $k>d$.

Using linear algebraic methods, Frankl and F\"uredi \cite{FF2} improved the last bound to $sim(n,k,d)\le {n\choose k-1}$. They proved the following theorem.
\begin{thm}[Frankl and F\"uredi \cite{FF2}]\label{thmff5}
Conjecture~\ref{conjchv} holds for any $k>d$, $n>n_0(k)$. Moreover, the equality is only possible if $\ff = \{F\in {[n]\choose k}: x\in F\}$ for some $x\in[n]$.
\end{thm}

Erd\H os--Chv\'atal simplices are, actually, families of hypergraphs. Frankl and F\"uredi showed that essentially the same results hold if one restricts to one particular type of simplices.
\begin{defn}\label{defssim}
  We say that $\{F_1,\ldots, F_{d+1}\}$ is a {\em special $d$-simplex} if there is a $(d+1)$-element set $C =\{x_1,\ldots, x_{d+1}\}$, such that $C\cap F_i = C\setminus \{x_i\}$ and that, moreover, $F_i\setminus C$, $i\in[d+1]$ are pairwise disjoint.
\end{defn}
Let us denote $ssim(n,k,d)$ the analogue of $sim(n,k,d)$, but for special simplices.
\begin{thm}[Frankl and F\"uredi \cite{FF2}]\label{thmff6}
  Let $n\ge n_0(k)$. Then $ssim(n,k,d)\le {n-1\choose k-1}$, provided $k>d+2$, or $d=3$ and $k=4,5$. Moreover, the equality is only possible if $\ff=\{F\in {[n]\choose k}: x\in F\}$ for some $x\in[n]$.
\end{thm}

Frankl and F\"uredi also discussed the following generalization of the notions of a simplex and special simplex.
\begin{defn}
A {\em $(d,\ell)$-simplex} is defined as a $d$-simplex, except we ask all sets to have intersection of size at most $\ell-1$, but any $d$ sets to have intersection size at least $\ell$.   A special {\em $(d,\ell)$-simplex} is defined as a special $d$-simplex, to which we append $\ell-1$ new elements that are contained in all sets.
\end{defn}
That is, a (special) $(d,1)$-simplex is the same as a (special) $d$-simplex. We denote the corresponding extremal functions by $sim_\ell(n,k,d)$ and $ssim_\ell(n,k,d)$. It may seem just like a generalization of the previous notions, but actually it brings several very interesting connections.

First, it is easy to see that determining $sim_\ell(n,d+\ell-1,d)$ is equivalent to a notoriously difficult problem of Brown Erd\H os and S\'os of determining the largest family such that any $\ell+d$ vertices span at most $d$ sets.

Second, Frankl and F\"uredi showed that $sim_\ell(n,5,2)=o(n^4)$, but that $sim_3(n,5,2)/n^{4-\epsilon}\to \infty$ for any $\epsilon>0$. (For these parameters, $sim_3(n,5,2) = ssim_3(n,5,2)$ since the only such simplex is the special one.) This result anwered a question of Erd\H os on the existence of a single hypergraph with no exponent in the corresponding Tur\'an function. It is obtained by a reduction to a result of Ruzsa and Szemer\'edi \cite{RS} on linear $3$-uniform hypergraphs without a triangle, which, in turn, is connected to subsets of integers without $3$-term arithmetic progressions.

Frankl and F\"uredi also noted that, restricting to the most popular $(k-5)$-element set, the bound above implies $s(n,k,2,k-2)=o(n^{k-1}),$ and asked whether $s(n,k,2,k-2)/n^{k-1-\epsilon}\to \infty$ for all $k\ge 6$ and any $\epsilon>0$. F\"uredi and Gerbner proved it in \cite{FG21}.

They proved the following theorem concerning $(d,\ell)$-simplices.

\begin{thm}[Frankl and F\"uredi \cite{FF2}]\label{thmff7}
  Suppose $k\ge \ell+d$ and $n\ge n_0(k)$. Then we have
  \begin{itemize}
    \item[(i)] $sim_\ell(n,k,d) = O\big(n^{\max\{k-\ell, k-d-1\}}\big)$;
    \item[(ii)] If $k\ge 2\ell+d$, then $sim_\ell(n,k,d) ={n-\ell\choose k-\ell}$;
    \item[(iii)] If $k\ge 2\ell+d$, then $ssim_\ell(n,k,d) = (\sigma_{\ell}(\m S)+o(1)){n-\ell\choose k-\ell}$,
        where $\m S$ stands for the special $(d,\ell)$-simplex.
  \end{itemize}
\end{thm}

\subsubsection{Examples} One example for all these theorems is, of course, the family of all sets containing a given element or a set of $\ell$ elements. The other example for Theorem~\ref{thmff7} for larger values of $\ell$ is, as usual, coming from a packing: consider a packing of $k$-sets so that any $\ell+d-2$ elements are contained in at most one set. It gives an example of the size $(1-o(1)){n\choose \ell+d-2}/{k\choose \ell+d-2}$.

In the part (i) of Theorem~\ref{thmff7} for $k\le 2\ell+d-2$, the second expression gives the maximum, and in this case there is a gap of $1$ in the exponent between the upper bound and the packing example. The case of $s(n,5,2,3)$ discussed above partially explains this phenomenon.

\subsubsection{Proofs}
Asymptotic bounds in Theorem~\ref{thmff6} for $k>d+2$, as well as Theorem~\ref{thmff7} (iii) are actually direct consequences of Theorem~\ref{thmff4}. (One just needs to calculate the sizes of the kernel and the center of these families.) In order to get exact results in all cases except the case of Theorem~\ref{thmff6} for $k=3,d=2$, Frankl and F\"uredi employ a strategy which is very similar to the strategy in the proof of Theorem~\ref{thmff2}.\footnote{We may say that the case $k=4,d=2$ corresponds to the harder case of Theorem~\ref{thmff2} when $k=2\ell+2$, which we haven't discussed. In this case, there is no uniqueness in Lemma~\ref{lemrankref}, and one requires extra arguments. The key idea is that any family not of the type as in Lemma~\ref{lemrankref} contributes not one, but two own subsets, which allows to bound the overall contribution of such families.}
In general, Delta-system proofs follow the same strategy, with one important detail: in each case, one needs an appropriate analogue of Lemma~\ref{lemrankref}.

The case $k=3, d=2$ is treated using a clever weighting (double-counting) argument, in which weights are assigned to the sets in $(k-1)$-shadows.

\subsection{Trees} We define {\em hypertrees} recursively as follows.
\begin{defn}[Hypertrees]\label{deft1}  A single edge $E$ is a hypertree. A family $\m T$  is a {\em hypertree} if there is an ordering of edges $\m T =\{E_1,\ldots, E_m\}$ such that for each $i\in[m-1]$ the collection of sets $\{E_1,\ldots, E_i\}$ is a hypertree and $E_{i+1}\cap \cup \m T = E_{i+1}\cap E_j$ for some $j\le i$. A {\em $k$-tree} is a hypertree with all edges of having size $k$.
\end{defn}

\begin{defn}[Tight hypertrees]\label{deft2} In terms of Definition~\ref{deft1}, we call a $k$-tree {\em tight} if $|E_{i+1}\cap E_i| = k-1$ for all $i\in [m-1]$.
\end{defn}

Generalizing the  Erd\H os--S\'os conjecture, which states that any graph on $n$ vertices with at least $\big\lfloor\frac{(k-1)n}2\big\rfloor+1$ edges contain a copy of any tree with $k$ edges (see \cite[page 2]{Erd64}), Kalai (personal communication to the authors of \cite{FF2}) conjectured that for any tight $k$-tree $\m T$ on $v$ vertices (i.e., with $|\cup \m T| = v$) one has
\begin{equation}\label{eqkal} ex(n, \m T)\le \frac {v-k}k {n\choose k-1}.\end{equation}
Note that $v-k = |\m T|-1$. R\"odl's packing theorem \cite{Ro} allows to show that, if true, this bound is tight up to $(1+o(1))$ term. Indeed, consider a packing of $(v-1)$-element sets such that any $(k-1)$-element set is contained in at most one of them. There are $(1+o(1)){n\choose k-1}/{v-1\choose k-1}$ sets in such a packing. Then, consider the shadow on level $k$ of this packing (cf. Sections~\ref{sec331},~\ref{sec411}). It will have size $(1+o(1)){n\choose k-1}{v-1\choose k}/{v-1\choose k-1}=(1+o(1))\frac{v-k}k{n\choose k-1}$.  Note that this example will have the size equal to the RHS of \eqref{eqkal} if it is a perfect packing.

Using a weighting argument, Frankl and F\"uredi proved the conjecture \eqref{eqkal} for {\it star-shaped trees}.
\begin{thm}[Frankl and F\"uredi \cite{FF2}]
  The conjecture \eqref{eqkal} holds for all $k$-trees $\mathcal T$ for which there is an edge $E$ such that any other edge $F$ satisfies $|E\cap F|=k-1$.
\end{thm}
Moreover, they managed to show that equality holds only if the family is the $k$-th shadow of a perfect packing of the type described above!

\section{Stability, structure, and supersaturation via Delta-systems}\label{sec6}
\subsection{Stability} A powerful and very popular method of proving results in extremal combinatorics is via stability results. In a large sense, it can be any `bootstrapping' strategy. We start by discerning some structure in all large/nearly-optimal examples for a particular problem. This structure, sometimes together with more restrictive assumptions on how close is the size of the structure to extremal, allows us to discern even more structure. We may repeat this process several times until eventually getting the exact extremal object. For results in extremal set theory, typically, there are at most $3$ steps in such proof strategy:  (i) a coarse structural result, which in some way allows to compactly describe the structure of most of the object; (ii) a $99\%$-stability result, which states that almost all nearly-extremal families fall inside the expected extremal example; (iii) proof of the exact result.

Stability approach has rich history, going back to works of Erd\H os and Simonovits (see \cite{Simo}). Their method is discussed in the survey of Keevash \cite{Kee2}.  For the problems discussed in this survey, stability was pushed forward, extensively used and popularized in the works of Mubayi and coauthors, see \cite{MV}. In these works, stability results are made into a separate statement, which since then became a common practice. Stability results for extremal questions in a different parameter regime were also developed in the Boolean Analysis perspective in the works of Keller, Lifshitz and coauthors (see \cite{KL}, \cite{EKL}) and by Zakharov and the author \cite{KuZa}. (See Section~\ref{sec7} for a brief discussion of the last two methods.) In this survey, I wanted to point out that stability results are a natural part of the Delta-system method and, as such, stability method was used from the early days of the Delta-system method.

The earlier variant of the Delta-system method decomposition (such as, say, in Theorem~\ref{thmdef}) by its nature gives a low-uniformity family that contains most of the family. I.e., if the $(n,k,L)$-family in Theorem~\ref{thmdef} in the Delta-system decomposition does not have a set of uniformity $\ell_1$, then it has size $O_k(n^{|L|-1})$. Next, all but $O_k(n^{|L|-2})$ sets of the family are covered by sets of sizes $\ell_1,\ell_2$ etc. This opens up possibilities for a series of coarse structural results.  We also refer the reader to Section~\ref{secbase}, where we touch upon decompositions as in the early variant of the Delta-system method, as well as to the paper of Kostochka and Mubayi \cite{KM2} for more examples.

I wanted to focus on post-Theorem~\ref{thmfur} variant of the method, since it seems that inherent stability method, contained in the proof of Frankl and F\"uredi \cite{FF1} is somewhat overlooked.  As an example, Keevash, Mubayi, and Wilson in \cite{KMW} prove a $99\%$ stability result for the forbidden intersection $1$, in the following form:
\begin{thm} For any $\epsilon>0$ there exists $\delta>0$ such that if $|\ff|\ge(1-\delta) {n-2\choose k-2}$ and $\ff$ avoids intersection $1$, then there is $X$, $|X|=2$, such that $|\ff\setminus \ff[X]|\le \epsilon n^{k-2}$.
\end{thm}

Let us explain, how the proof of Frankl and F\"uredi from Section~\ref{sec333} almost immediately implies the following stronger result.

\begin{thm}\label{thmstab} Assume that $n\ge n_0(k)$. There exists $C = C(k)$ such that the following holds for any $\epsilon =\epsilon(n)$. Assume that $k\ge 2\ell+3$ and $\ff\subset {[n]\choose k}$ avoids intersection $\ell$ and $|\ff|\ge (1-\epsilon){n-\ell-1\choose k-\ell-1}$. Then there is $X$, $|X|=\ell+1$, such that \begin{equation}\label{eqstab00}|\ff\setminus \ff[X]|\le \Big(\frac{\ell\epsilon}{(1-\epsilon)(k-2\ell-1)}\Big)^{\frac{k-\ell-1}\ell} {n-\ell-1\choose k-\ell-1}+C_k n^{k-\ell-2}.\end{equation}
\end{thm}
We remark that the same could be done for $k=2\ell+2$, but since we haven't treated this case in Section~\ref{sec333}, we omit it here as well.
The proof we present is essentially the same as in \cite{FF1}, just with some explicit calculations filled in.

Note the dependence on $\epsilon$ in the error term in \eqref{eqstab00}. For example, in the Keevash--Mubayi--Wilson setting, we get $\epsilon^{k-2}{n-2\choose k-2}+O(n^{k-3})$ error term. We exploit the same phenomenon as for the proof of the fact that every monotone property has a threshold by Bollob\'as and Thomason \cite{BT}. A very similar argument was used by Keller and Lifshitz in \cite{KL}.

\begin{proof} In the proof, we use $O$-notation as if $k$ is fixed and $n\to \infty$.
  Going through Section~\ref{sec333}, we first can limit ourselves to $\ff^*:=\ff_1^*\cup\ldots \cup \ff_{m-1}^*$ by removing at most $C_k(n^{k-\ell-2})$ sets (cf. \eqref{eqrem1}). Then we group the sets in $\ff^*$ by their centers $C_i$, and form families $\m G_i$. We see that their shadows are disjoint, and get the following two inequalities. First, as in \eqref{eqshadowsum}, 
  \begin{equation}\label{eqstab0}\sum_{i\in [h]}|\partial^{(\ell)}(\m G_i(C_i))|\le {n\choose \ell}.\end{equation}
  At the same time, the inequality $|\ff^*|\ge (1-\epsilon){n-\ell-1\choose k-\ell-1}-C_kn^{k-\ell-2}$ translates into
  \begin{multline}\label{eqstab1}\sum_{i\in [h]}|\m G_i(C_i)|\ge (1-\epsilon){n-\ell-1\choose k-\ell-1}-O(n^{k-\ell-2})\\= (1-\epsilon){n\choose k-\ell-1}-O(n^{k-\ell-2}).\end{multline}
In Section~\ref{sec333} we said that \begin{quote} We then apply the Kruskal--Katona theorem to $\m G_i(C_i)$ and conclude that the last two displayed inequalities are compatible only if one of the families, say $\m G_1$, contains a bulk of the sets: if $|\ff\setminus \m G_1|= O(n^{k-\ell-2})$.\end{quote}
Here, we will give the precise argument (in our slightly more general setting). Assume that $|\m G_i(C_i)| = {x_i\choose k-\ell-1}$ with $x_1\ge x_2\ge\ldots\ge x_h$. Then, by the Kruskal--Katona theorem we get that $$|\partial^{(\ell)}(\m G_i(C_i))|\ge {x_i\choose \ell}= \prod_{j=\ell}^{k-\ell-2}\frac{k-1-j}{x_i-j}{x_i\choose k-\ell-1}=\prod_{j=\ell}^{k-\ell-2}\frac{k-1-j}{x_i-j}|\m G_i(C_i)|.$$
In order to conveniently substitute it into \eqref{eqstab0} and \eqref{eqstab1}, we multiply \eqref{eqstab0} by ${n\choose k-\ell-1}/{n\choose \ell}=\prod_{j=\ell}^{k-\ell-2}\frac{n-j}{k-j-1}$:
\begin{align}
 \notag {n\choose k-\ell-1} &\ge \sum_{i\in [h]}\frac{{n\choose k-\ell-1}}{{n\choose \ell}}|\partial^{(\ell)}(\m G_i(C_i))| \\
 \notag &\ge  \sum_{i\in [h]}\prod_{j=\ell}^{k-\ell-2}\frac{n-j}{k-j-1}\cdot \frac{k-j-1}{x_i-j} |\m G_i(C_i)|\\
 \notag &\ge   \prod_{j=\ell}^{k-\ell-2}\frac{n-j}{x_1-j} \sum_{i\in [h]}|\m G_i(C_i)|\\
 \notag &\ge  \Big(\frac{n}{x_1}\Big)^{k-2\ell-1} \sum_{i\in [h]}|\m G_i(C_i)|\\
\label{eqstab2}  &\ge \Big(\frac{n}{x_1}\Big)^{k-2\ell-1} (1-\epsilon-O(1/n)){n\choose k-\ell-1}.
\end{align}
The last inequality is \eqref{eqstab1}. These simple calculations imply that $\Big(\frac{n}{x_1}\Big)^{k-2\ell-1} (1-\epsilon-O(1/n)) \le 1$, which  implies that $x_1\ge (1-\frac{\epsilon}{(1-\epsilon)(k-2\ell-1)})n - O(1)$. From here we can derive that $|\m G_1|\ge \big(1-\frac{(k-\ell-1)\epsilon}{(1-\epsilon)(k-2\ell-1)}\big){n\choose k-\ell-1} - O(n^{k-\ell-2})$. Combined with \eqref{eqff14}, we infer a stability with the number of sets not contained in $\m G_1$ bounded by $\frac{(k-\ell-1)\epsilon}{(1-\epsilon)(k-2\ell-1)}{n\choose k-\ell-1} + O(n^{k-\ell-2})$. Instead, let us use the Kruskal--Katona theorem once again and get a much better bound.

We have 
\begin{align*}
|\m \partial^{(\ell)}\m G_1(C_1)|\ge {x_1\choose \ell}&\ge {(1-\frac{\epsilon}{(1-\epsilon)(k-2\ell-1)}n -O(1)\choose \ell}
\\
 &\ge {n\choose \ell}-\frac{n\epsilon+O(1)}{(1-\epsilon)(k-2\ell-1)} {n\choose \ell-1}.\end{align*}
 Thus, $$\sum_{i=2}^h |\partial^{(\ell)}(\m G_i(C_i))|\le \frac{n\epsilon+O(1)}{(1-\epsilon)(k-2\ell-1)} {n\choose \ell-1}=\frac{\ell(\epsilon+O(1/n))}{(1-\epsilon)(k-2\ell-1)} {n\choose \ell}={y\choose \ell},
$$
where $y\le n \Big(\frac{\ell\epsilon+O(1/n)}{(1-\epsilon)(k-2\ell-1)})\Big)^{1/\ell}$. Using the Kruskal--Katona theorem, for each $i\ge 2$ we have $x_i\le y$. We substitute the inequality above into the same chain of inequalities as \eqref{eqstab2}:
\begin{align}
 \notag \frac{\ell(\epsilon+O(1/n))}{(1-\epsilon)(k-2\ell-1)}{n\choose k-\ell-1} &\ge \sum_{i=2}^h\frac{{n\choose k-\ell-1}}{{n\choose \ell}}|\partial^{(\ell)}(\m G_i(C_i))| \\
 \notag &\ge  \sum_{i=2}^h\prod_{j=\ell}^{k-\ell-2}\frac{n-j}{k-j-1}\cdot \frac{k-j-1}{x_i-j} |\m G_i(C_i)|\\
 \notag &\ge   \prod_{j=\ell}^{k-\ell-2}\frac{n-j}{y-j} \sum_{i=2}^h|\m G_i(C_i)|\\
 \notag &\ge  \Big(\frac{n}{y}\Big)^{k-2\ell-1} \sum_{i=2}^h|\m G_i(C_i)|
\end{align}
Rearranging the resulting inequality and substituting the bound on $y$, we get
$$\sum_{i=2}^h|\m G_i(C_i)|\le \Big(\frac{\ell(\epsilon+O(1/n))}{(1-\epsilon)(k-2\ell-1)}\Big)^{\frac{k-\ell-1}\ell}{n\choose k-\ell-1}.$$
Adding the remainder $|\ff\setminus \ff^*|$, we get the statement.
\end{proof}

Can we say something similar for more complicated hypergraphs? Both coarse structure results and $99\%$ stability for a large class of graphs that are embeddable into $2$-contractible $k$-trees (see Definition~\ref{defcont}) was obtained in a work of F\"uredi and Jiang \cite[Theorem~4.2]{FJ15}. For this problem, the extremal function is of the order $\sigma_1(\m H){n-1\choose k-1}$, where $\sigma_1(\m H)$ is the size of cross-cut minus $1$, as defined in Section~\ref{sec42}. In this `quasi-dense' regime the intersection structure $\m M$ is easier to control.

Can this be extended further? Let us take an example of Theorem~\ref{thmff3}. The problem is that the argument at the end of the proof is a certain double-counting argument, which arrives at a contradiction in case some of the $(k-\ell-1)$-sets  that are own in some $\ff_i^*$ can be extended by too many centers (i.e., the same subset is appearing as an own subset in too many $\ff_i^*$). In an upcoming note with Noskov, we are going to show, that this obstacle is possible to bypass. Let us sketch the argument, which shows that, provided the family  $\ff\subset{[n]\choose k}$ as in Theorem~\ref{thmff3} satisfies \begin{equation}\label{eqboundsun} |\ff|\ge (\phi(\ell+1,s)-\epsilon){n-\ell-1\choose k-\ell-1},\end{equation} then there is a family $\m C\subset {[n]\choose \ell+1}$ with no $\Delta(s+1)$-system, and such that $|\ff\setminus \ff[\m C]|<\delta{n-\ell-1\choose k-\ell-1}$.

Consider the Johnson graph $J(n,k-\ell-1)$ on vertices $X\in {[n]\choose k-\ell-1}$, that is, connect two sets if their symmetric difference has size $2$. Take two sets $X_1,X_2$ that form an edge in $J$. Then, looking at the argument in Section~\ref{sec41}, we can infer that not only $\m G_{X_1}$, $\m G_{X_2}$ cannot contain a $\Delta(s+1)$-system individually, but also does their union $\m G_{X_1}\cup \m G_{X_2}.$ Provided that the family $\ff$ satisfies \eqref{eqboundsun}, for all but an $\epsilon$-fraction of sets $X\in {[n]\choose k-\ell-1}$ we must have $|\m G_{X}| = \phi(\ell+1,s)$. But then any such two sets $X_1, X_2$ that form an edge must also have $\m G_{X_1}= \m G_{X_2}.$ We can then use expansion properties of $J(n,k-\ell-1)$ that there should be a `giant component' in $J$ of size $(1-\delta){n\choose k-\ell-1}$ on the sets $X$ such that $|\m G_X| = \phi(\ell+1,s)$. All of them will have the same family $\m G_X$. This family $\m G_X$ is the desired family $\m C$.

\subsection{A more quantitative view, structure, and supersaturation}\label{sec62}
In a series of papers \cite{Fu14}, \cite{FJ15}, \cite{JL}, Delta-system method was adapted to give coarse structural results, in the sense of the previous section. Let us discuss the paper of Jiang and Longbrake \cite{JL}, in which they took a new  point of view on the Delta-system method, which allows to get more quantitative results.  First, they managed to prove a structural result in the spirit of Lemma~\ref{lemrankref}, but gave a decomposition of the family without directly using any forbidden structure. They also made F\"uredi's Theorem~\ref{thmfur} more quantitative. These two developments together allowed them to prove {\it supersaturation} results. A supersaturation result for the function $ex(n,\m H)$ is a result that states that, whenever we have a family $\ff$ of size just above $ex(n,\m H)$, not only we have one copy of $\m H$, but we actually have many. Precise definitions of `just above' and `many' depend on the situation. These questions were popular in combinatorics in the recent years, and in the context of extremal set theory, it was particularly well-studied for Kneser/disjointness graphs. See, e.g., \cite{BL}, \cite{DGS}.

In this section, we describe the contributions of \cite{JL}. First, the authors provide the following quantitative variant of F\"uredi's theorem.

\begin{thm}\label{thmfurjl}
  Let $s,k\ge 2$ be integers with $s\ge 2k$. Then there exists a positive constant $c(s,k)$ (akin to the constant in Theorem~\ref{thmfur}), such that the following holds. Let $\ff\subset {[n]\choose k}$. Then there is a $k$-partite subfamily $\ff^*\subset \ff$, such that the following holds:
  \begin{itemize}
    \item $|\ff^*|\ge c(k,s)|\ff|$;
    \item $\pi(\ff^*|_A) = \pi(\ff^*|_B)$ for any $A,B\in \ff$;
    \item For every $J = A\cap B,$ such that $A,B\in \ff^*$, and every $x\in [n]\setminus J$ we have $|\ff^*(J\cup \{x\})|\le \frac 1s |\ff^*(J)|$;
    \item For each $J$ as above we have $|\ff^*(J)|\ge \max \big\{s,\frac 1{2k}\frac {|\ff^*|}{n^{|J|}}\big\}$.
  \end{itemize}
\end{thm}
The proof is the same as that of Theorem~\ref{thmfur}, with the additional use of the fact that the sets which violate $|\ff^*(J)|\ge \frac 1{2k}\frac {|\ff^*|}{n^{|J|}}$ for $J$ all together  cover no more than half of the sets from the family, and so we can get rid of them.

Second, the authors prove the following intersection structure lemma.

\begin{lem}\label{lemstrucjl}
  Let $k\ge 3$ be an integer. Let $\m J\subset 2^{[k]}$ be a family of proper subsets of $[k]$ that is closed under intersection and of rank at least $k-1$. Then one of the following must hold.
  \begin{enumerate}
    \item There exists $B\subset [k]$ of size $k-2$, such that $2^{B} \subset \m J$.
    \item there exists a unique $i$, such that for any $F$ such that $i\in F$, we have $F\in \m J$. Moreover, for every $D\subset [k]\setminus \{i\}$ with $|D|\ge k-2$ we have $D\notin \m J$.
  \end{enumerate}
\end{lem}
Previously, very similar statements were proved in \cite{FF2} and \cite{FJ15}, with the small difference that in the second case nothing was stated concerning sets of size $k-2$. (We note that this was mentioned in passing in its analogue for the forbidden one intersection \cite[Lemma 5.5a]{FF1}.) This is a key point for \cite{JL}. Using it, the authors obtain a variant of Propositions~\ref{propglob2} and~\ref{propglob3}, and then give a variant of the disjoint shadows argument that exploits the Kruskal--Katona theorem as in that same Section~\ref{sec333}. As a result, they obtain the following structural result.
\begin{thm}
Let $k,a,\epsilon>0$ be constants. Consider a family $\ff\subset {[n]\choose k}$ of size $a{n\choose k-1}$. Then one of the following should hold.
\begin{enumerate}
  \item There is a subfamily $\ff^*\subset\ff$ as in Theorem~\ref{thmfurjl} and with intersection structure as in part 1 of Lemma~\ref{lemstrucjl}.
  \item There is a subfamily $\ff'\subset \ff$ and $W\subset [n]$ such that (i) $|W|=O_k(1)$; (ii) $|\ff'|\ge (1-\epsilon)|\ff|$; (iii) for every $F\in \ff'$ we have $|W\cap F| = 1$.
\end{enumerate}
\end{thm}
As the authors point out, previously similar results were obtained in \cite{Fu14}, \cite{FJ15}, but under the assumption that $\ff$ avoids a certain subfamily, which we cannot assume for the supersaturation problem.
Now assume that we are trying to prove that in a family $\ff$ of size $(1+\epsilon)ex(n, \m H)$ for some $\m H$ there are many copies of $\m H$. If $\m H$ is chosen right, then in the first case  we should be able to find many copies of $\m H$ because of the form of intersection structure (having the $2^B$ part comes handy). In the second case, we get a good understanding of the coarse structure of the family $\ff$, and then exploit it to find many copies of $\m H$. Note that for any $x\in W$ the ratio $|\ff(x)|/{n-1\choose k-1}$ is constant, i.e., we are in the `dense regime'. In the dense regime, it is often easy to provide embeddings or cross-embeddings of most structures using counting arguments.  

\subsection{A weaker version of F\"uredi's structural theorem with better dependencies}\label{sec63}
In this subsection, we discuss some of the contributions of the recent paper by Janzer, Jin, Sudakov and Wu \cite{JJSW}. We focus on two major developments for the delta-system method.

First, the authors showed that $c(k,s)$ in Theorem~\ref{thmfur} can be at most $2^{-2^{(1+o(1))k}}$ in some situations. The reason for this is the property that all sets in $\ff^*$ from F\"uredi's theorem should have the same intersection structure. The authors of \cite{JJSW} basically combine families on disjoint ground sets with different possible intersection structures (specifically, of the form $\bigcup_{i=0}^{\lceil k/2\rceil-1}{[k]\choose i}\cup \mathcal I$ for an arbitrary $\m I\subset {[k]\choose \lceil k/2\rceil}$) and show that essentially all sets in $\ff^*$ must come from one of such families. In order to show the `essentially' part, they require that, for any given set  $A\in \m I$, the families can be decomposed into large delta-systems with centers that are projected into $A$. The construction of such families is, however, non-trivial. The latter `delta-system decomposition' property suggests that families defined by subspaces should be handy. Indeed, they construct their families using subspaces in high-dimensional spaces $\mathbb F_p^d$. They encode which sets from ${[k]\choose \lceil k/2\rceil}$ to include in $\m A$ via linear dependencies.

Second, they prove the following theorem.
\begin{thm}[\cite{JJSW}]\label{thmjjsw} Suppose $\ff\subset {[n]\choose k}$ and let $s\ge 2$. Then, writing $\beta_{k,s} = (25\cdot 2^k ks)^{-k}$, there exists a subfamily $\ff'\subset \ff$ with $|\ff|\ge \beta_{k,s}|\ff|$ such that
\begin{itemize}
  \item For every distinct $F_1,F_2\in \ff'$ there is a $\Delta(s)$-system in $\ff$ with kernel $F_1\cap F_2$;
  \item in fact, for every $F\in \ff'$, there exists an intersection-closed set system $\m I_F$ such that $\m I(F,\m F'):=\{F\cap F': F'\in \ff'\}\subset \m I_F\subset 2^F\setminus \{F\}$, and every element of $\m I_F$ is the kernel of a $\Delta(s)$-system in $\ff$.
\end{itemize}
\end{thm}
It should be clear that the second property implies the first one. The proof of the theorem is actually simple, and we sketch it here, with a slightly worse constant. Somewhat simplifying, we produce a random partition into much more than $k$ parts and choose only sets between some $k$ of these parts, showing that they satisfy the properties.
\begin{proof}
  First, color the ground set randomly into $N = 25\cdot 2^{k} ks$ colors and note that, for a set $F\in \ff$ with probability at least $1/2$ we have the following:
  \begin{itemize}
    \item the set $F$ is rainbow, i.e., all vertices get different colors;
    \item Assume that $A\subset F$ is such that $\ff(A)$ has no matching of size $s$, and let $X = X(A)$ be a hitting set for $\ff(A)$, $|X|\le k(s-1)$. Then, for all $A\subset F$ and for all corresponding $X(A)$ the colors of $X(A)\setminus F$ are disjoint from colors used for $F$.
  \end{itemize}
 Next, take a subfamily $\ff''\subset \ff$ consisting only of sets that satisfy the two properties above, and note that $2|\ff''|\ge |\ff|$. Among possible ${N\choose k}$ palettes of $k$ colors, choose the palette with the largest number of sets in $\ff''$. Let $\ff'\subset \ff''$ be a subfamily of these sets with a fixed palette, and note that $|\ff'|\ge |\ff|/2{N\choose k}$. We claim that $\ff'$ satisfies the claimed properties. Of course, the condition on the size is satisfied.

 Now, put $$\m I_F:= \big\{F\cap F_1\cap \ldots\cap F_t: F_1,\ldots, F_t \in \ff'\setminus \{F\}\big\}.$$
 Clearly, $\m I(F,\m F')\subset \m I_F$ and the latter is intersection-closed. Next, take $A\in \m I_F$ and assume that $A$ is not a kernel of a $\Delta(s)$-system. Note that, for any $F'\in \ff'$ such that $F'\supset A$, the set $X(A)\setminus F'$ must have colors that are disjoint from colors of $F'$. But $F'$ is rainbow and we have $k$ colors overall. Therefore, the part $X'(A)$ of $X(A)$ that is colored with the chosen $k$ colors must lie inside $F'$. We have that if $A = F\cap F_1\cap\ldots \cap F_t$ for some $F_1,\ldots , F_t\in \ff'$, then $X'(A) \subset F\cap F_1\cap\ldots \cap F_t$, a contradiction since $X'(A)\ne \emptyset $ and $X'(A)\cap A = \emptyset$. We conclude that any set in $\m I_F$ is a kernel of a $\Delta(s)$-system.
 \end{proof}

The family $\ff'$ from Theorem~\ref{thmjjsw} has weaker properties than the family $\ff^*$ from Theorem~\ref{thmfur}, but for many applications of the $\Delta$-system method it is sufficient. In particular, it is sufficient for the forbidden one intersection Theorem~\ref{thmff2} and for the forbidden $\Delta$-system with a fixed core Theorem~\ref{thmff3}.

 \section{Approaches to constructing bases}\label{secbase}
We have given a description of the `original' sunflower base construction in Lemma~\ref{lemdeltabase}. In this section, we discuss the proof of that lemma and give several other possible approaches. Let us abstract the setup of Lemma~\ref{lemdeltabase}.  The goal of the lemma was to give a construction of a base $\mathcal B$ such that
\begin{itemize} \item the upper closure of $\bb$ contains the family $\ff$;
\item the family $\bb$ has the same properties as $\ff$ (i.e., $L$-intersecting);
\item the number of sets of different sizes in $\mathcal B$ is controlled.
\end{itemize}
In general, these are the properties that we would like to see when constructing a base.

The bounds on $|\mathcal B^{(k)}|$ in Lemma~\ref{lemdeltabase} were rather bad: we argued that it did not have a $\Delta(k^r)$-system. The proof is what we could call a `bottom to top construction, top to bottom argument': when constructing the base family, we start from finding smaller kernels and then go to larger kernels. When seeking for contradiction, to the contrary, we start with a larger $\Delta$-system with bigger kernels and get a smaller $\Delta$-system with a smaller kernel.

A more efficient approach is to run a `top to bottom construction, bottom to top argument'. Let us describe a possible procedure. Start by including in $\m B$ the sets from the original family. Then, we simplify $\m B$ as follows: if some sets in the current family form a $\Delta(k+1)$-system, then replace them by the kernel of the $\Delta$-system. We perform this exchange as long as it is possible. Clearly, the resulting family has no $\Delta(k+1)$-systems (a much better restriction than no $\Delta(k^r)$-systems).
\begin{cla} If the original family $\ff\subset {[n]\choose \le k}$ had all pairwise intersections in $L$, then the resulting family $\mathcal B$ has all intersections  in $L$.
\end{cla}
\begin{proof}
  The proof is by induction on the step of the procedure. Assume that the current simplified family is $\bb'$, and $A,B\in \mathcal B'$ satisfy $|A\cap B|\notin L$. It cannot be that both sets $A,B$ belong to $\ff$, and thus one of them, say $A$, is a kernel of a $\Delta(k+1)$-system $A_1,\ldots, A_{k+1}$ that was replaced by $A$ at some earlier step of the procedure. Moreover, we may assume that $B$ appeared in the simplified family before $A$.  But then for some $i$ we have $B\cap A_i = B\cap A$, and thus $|B\cap A_i|\notin L$, a contradiction with the inductive assumption.
\end{proof}

In recent papers, Zakharov and the author \cite{KuZa} and then the author \cite{Kup54} developed the {\it peeling-simplification procedure},  probably the most efficient approach for the construction of  base-like structures for the family and for subsequent analysis of the structure of the family. First, we need to recall the notion of $r$-spread families. A family $\ff$ is {$r$-spread} for some $r>1$ if $|\ff(X)|< r^{-|X|}|\ff|$ for any set $X$. We will see  that $r$-spread families in many ways behave like sunflowers with $r$ petals, albeit they are much easier to find. Let us first give an analogue of the `top to bottom construction' above.

Start by including in $\m B$ the sets from the original family $\mathcal F\subset {[n]\choose \le k}$. Then, we simplify $\m B$ as follows: if there is a set $X$ and a family of sets $\mathcal G\subset \m B$, such that, first, all sets in $\mathcal G$ contain $X$ and, second, $\mathcal G(X)$ is $k$-spread, then replace $\m G$ by $X$ in the current family. We perform this exchange as long as it is possible.\footnote{Note that, technically, a family $\m B(X)$ with $X\in \m B$ being inclusion-maximal satisfies the condition since it has size $1$, but any restriction would have size $0$. But replacing $\m G(X)= \{X\}$ with $X$ does not change the family, so we stop the replacement if these are the only replacements left.} Clearly, the resulting family has no $k$-spread subfamilies.
\begin{cla} If the original family $\ff\subset {[n]\choose \le k}$ had all pairwise intersections in $L$, then the resulting family $\mathcal B$ has all intersections in $L$.
\end{cla}
\begin{proof}
  The argument is very similar to the previous one. The proof is by induction on the step of the procedure. Assume that the current simplified family is $\bb'$, and $A,B\in \mathcal B'$ satisfy $|A\cap B|\notin L$. Find the family $\mathcal G$ that was replaced by $A$ in the procedure. Then
  $$|\mathcal G|-\sum_{x\in B\setminus A}|\mathcal G(x)|> \Big(1-\frac {|B\setminus A|}k\Big)|\mathcal G|\ge \Big(1-\frac {k}k\Big)|\mathcal G|= 0.$$
    Thus, we can find a set $A'\in \mathcal G$ such that $A'\cap B =A\cap B$. Consequently, $|B\cap A'|\notin L$, a contradiction with the inductive assumption.
\end{proof}

Let us now bound the sizes of layers of $\bb$. It is not difficult to see that the following observation is valid.
\begin{obs}\label{obsspread}
  If $\m G\subset {[n]\choose \ell}$ is such that there is no $X$ such that $\m G(X)$ is $r$-spread, then $|\m G|\le r^\ell$.
\end{obs}
\begin{proof} Take inclusion-maximal $X$ that violates $r$-spreadness: a set $X$ such that $|\m G(X)|\ge r^{-|X|}|\m G|$. Then $\m G(X)$ is $r$-spread by the maximality of $X$. Since $\m G$ has no such (non-trivial) subfamily, it means that $|X| = \ell$, and so $1= |\m G(X)|\ge r^{-\ell}|\m G|$, or $|\m G|\le r^\ell$.
\end{proof}

This bound is already better than the bound coming from not containing a sunflower. Indeed, the observation shows that, as a result of the described procedure, we get a family $\bb$ such that  $|\bb^{(\ell)}|\le k^{\ell}$. This bound, however, may be inefficient in some cases. In what follows, let us limit ourselves to the case of  $t$-intersecting families for simplicity. Then $|\bb^{(t)}|$ should be either $0$ or $1$, but the bound that is guaranteed by the observation is only $k^t$. It is possible to somewhat improve it, but such improvement is insufficient for many purposes. In order to deal with this issue, we added another ingredient, called {\it peeling}, which we describe below. We present the construction for $t$-intersecting families, but it can be extended to many other properties. A very brief description is that we iteratively simplify the family and `peel off' the top layer of the family. At the next step, we simplify again, but with a better parameter $r$.

First, we record the following easy observation.
\begin{obs}\label{obs22}
For any  $t$-intersecting family $\s$ there exists a $t$-intersecting family $\T $ such that $\T$ is obtained from $\s$ by replacing some of the sets in $\s$ by their subsets, and such that $\T$ is maximal w.r.t. the property of being $t$-intersecting: for any proper subset $X \subsetneq T$ of a set $T\in \m T$ there exists $T' \in \T$ such that $|X \cap T'| < t$.
\end{obs}

Next, we describe the peeling procedure. Assume that we are given a $t$-intersecting family $\ff\subset {[n]\choose \le k}$.
\begin{enumerate}
    \item Let $\T_k$ be a family given by Observation~\ref{obs22} when applied to $\ff$.
    \item For $i = k, k-1\ldots, t$ we put $\W_i = \T_i \cap {[n] \choose i}$ and let $\T_{i-1}$ be the family given by Observation~\ref{obs22} when applied to the family  $\T_{i}\setminus \W_{i}$ playing the role of $\mathcal S$.
\end{enumerate}
Remark that $\T_i$ is $t$-intersecting for each $i=k,\ldots, t$ by definition. We summarize the properties of these series of families in the following lemma.

\begin{lem}[\cite{Kup54}]\label{lemkeyred} The following properties hold for each $i = k,k-1\ldots, t$ and any family $\aaa\subset 2^{[n]}$. 
\begin{itemize}
  \item[(i)] All sets  in $\T_i$ have size at most $i$.
  \item[(ii)] We have $\aaa[\T_{i}]\subset \aaa[\T_{i-1}]\cup \aaa[\W_{i}]$.
  \item[(iii)] The family $\T_i$ does not have a subfamily $\mathcal G$ and a set $X$ such that $\mathcal G(X)$ satisfies $|\mathcal G(X)|>1$ and is $(i-t+1)$-spread.
  \item[(iv)] We have $|\W_i|\le (6i)^{i-t}$.
  \item[(v)] If $\T_i$ consists of a single $t$-element set $X$ and this is not the case for $\T_{i+1}$ then $|\aaa[\T_{i+1}\setminus \W_{i+1}]|\le \frac{k}r |\aaa[T]|$.
\end{itemize}
\end{lem}

Let us comment on the lemma. Part (i) assures that the uniformity decreases as we progress. Part (ii), in particular, guarantees the first property from the beginning of this section: that the upper closure of the resulting family $\mathcal B$ contains $\ff$. Note that $\mathcal B$ here is the union of some $\m T_i$ and $\cup_{j>i}\m W_j$. It is up to us to decide, at which step to stop the procedure. Part (iii) guarantees that $\m T_i$, and consequently $\m W_i$, does not contain spread subfamilies. As a matter of fact, we could have replaced the use of Observation~\ref{obs22} with the following subprocedure: at step $i$, find any subfamily $\m G\subset \m T_i\setminus \m W_i$ and a set $X$ such that, first, $X$ is contained in all sets of $\m G$ and, second, $\m G(X)$ is $(i-t+1)$-spread. Remove $\m G(X)$ from $\m T_i\setminus \m W_i$ and add to $\m T_{i-1}$. Repeat until no such $\m G$ can be found.  From here, we could derive that $\m T_i$ stays $t$-intersecting.

Part (iii) ultimately allows to prove the bound in (iv). It is much better than the simpler bound  $|\W_i|\le k^{i}$ that we could obtain using our previous `top to bottom argument'. In particular, for $i = t$ (i.e., for sets from $\W_{t}$ that have size $t$) it gives the upper bound $1$ on the number of such sets, which is the correct number, and it gives constant bounds for sizes of $\W_{t+const}$. Parts (iv) and (v) allow to efficiently bound the size of $\aaa(\mathcal S)$ depending on the structure of $\mathcal S$. In order to get (iv) from (iii) and Observation~\ref{obsspread}, we first fix a subset $X$ of size $t$ in any of the sets $W\in \m W_i$ that is contained in at least ${i\choose t}^{-1}$ ratio of sets in $\m W_i$ (such $X$ exists by the $t$-intersecting property of $\m W_i$), and then apply part (iii) to $\m W_i(X)$.

One important difference between other base constructions and peeling-simplification is that in the former we do not actually create a remainder during the process, only potentially after the end (when `truncating' the base). In the latter, we gradually send sets into the remainder and take advantage of this during the process.\\


Next, let us discuss the bases ({\it nuclei}) used by F\"uredi for intersecting families in \cite{Fur78}. They are based on the following argument, essentially due to Erd\H os and Lov\'asz \cite{EL}. Take an intersecting family $\ff$ that is maximal w.r.t. that property (in the sense of Observation~\ref{obs22}). Let us take a set $F=\{x_1,\ldots, x_a\}\in \ff$. Any set in $\ff$ intersects it. Take $x_i$ and any set $F_i\in \ff$ such that $x_i\notin F_i$. All sets from $\ff$ that contain $x_i$ must additionally intersect $F_i$. Thus, any set in $\ff$ contains one of the pairs $\{(x_i, y): y\in F_i\}$. We may go on and construct $3$-element sets etc. In particular, this allows to efficiently bound the sets in $\ff$ of different sizes. Seen from another perspective, the produced sets (once we cannot find a set $F'$ that is disjoint with a generated sequence of elements) give exactly the family of minimal covers for the sets in $\ff$. These are the bases of the type used by Frankl in \cite{Fra2}. In a recent paper of Frankl \cite{Fra17}, the family of minimal covers is efficiently analyzed in order to bound the maximal diversity of an intersecting family. We may also say that this is essentially a simpler variant of the peeling-simplification procedure, with just one round of simplification and no peeling.

This notion of a base is close to the notion of a {\em generating set} from the celebrated paper of Ahlswede and Khachatrian \cite{AK}. Given a family $\ff\subset {[n]\choose k}$, a {\it generating set } $g(\mathcal F)$ is a family in ${[n]\choose \le k}$ such that the upper closure $\mathcal U$ of $g(\ff)$ satisfies $\mathcal U\cap {[n]\choose k} = \ff$. The difference with the notion of Frankl and F\"uredi is that if we take Frankl's base or F\"uredi's nucleus, and all $k$-element sets that contain at least one of the sets in the base, then we get a family $\ff'$ that contains, but is not necessarily equal to, $\ff$.  However, if the family is inclusion-maximal (which is typically a natural assumption), then the construction of Frankl or F\"uredi gives a particular generating set. In the proof of the Complete $t$-Intersecting Theorem \cite{AK}, the authors use generating sets somewhat differently, as compared to the Delta-system method. Since we can use shifting, generating sets are much more structured. The authors can then perform a very careful analysis of the contribution of different sets from the base to the size of the family, depending on the form of a particular set from the base, and then gradually modify the family, bringing it closer to one of the extremal examples.

\section{Other regimes, other decompositions}\label{sec7}
Before, we have mostly treated  the case when $n>n_0(k,t)$, where $t$ encodes, say, the size of the center of the forbidden configuration (i.e., the number of vertices of degree $\ge 2$). There are other regimes in which general results could be proved. Below, we assume that $k,t$ have this meaning.

In \cite{KL}, Keller and Lifshitz developed a junta approximation approach based on Boolean Analysis tools that allows to approximate the structure of the family by a smaller family with the same properties, but that lives on $f(t)$ coordinates. This allowed them to extend some of the results from Sections~\ref{sec3} and~\ref{sec4} to the regime $k>f(t), n>f(t)k$. Essentially, they solved the cases when the forbidden structure is a $k$-expansion, and the extremal configuration is either all sets containing a given $t$-element set (the Complete $t$-Intersection Theorem setting), or intersecting a given $t$-element set (the Erd\H os Matching Conjecture setting). In \cite{EKL}, Ellis, Keller and Lifshitz managed to solve the forbidden one intersection problem (no intersection $t-1$) in the regime  $k>f(t,\epsilon), \epsilon n<k<(1/2-\epsilon)n$. In this regime the extremal configuration is no longer the family of all sets containing a given $t$-element set, but a configuration of the Ahlswede-Khachatrian type: $\{F\in {[n]\choose k}: |F\cap [t+2i]|\ge t+i\}$ for a constant $i$.

In \cite{KLLM}, Keevash, Lifshitz, Long and Minzer developed an alternative, more efficient approach that uses hypercontractivity. They managed to prove the following result.

\begin{thm}[Keevash, Lifshitz, Long and Minzer \cite{KLLM}]
For any $r, \Delta\ge 2$ there is $C>0$ such that the following holds for any $r$-graph $\m G$ with $s$ edges, maximum degree $\Delta(\m G)\le \Delta$ and $\sigma_1(\m G)\ge 1$. For any $C\le k\le n/Cs$ we have $ex(n, \m H) = (1\pm \epsilon)\Big({n\choose k}-{n-\sigma_1(\m G)\choose k}\Big)$, where $\m H$ is the $k$-expansion of $\m G$.
\end{thm}
That is, the main difference with the previous result is that now the unspecified dependence is on the maximum degree and the uniformity of the underlying structure, and no longer on the number of edges (this dependence is linear). That is, it applies for the Erd\H os Matching Conjecture in the regime $n>Csk$ with some absolute constant $C$. However, it still does not allow to resolve, say, forbidden intersection $\ell$ problem for $\ell$ that is not constant, or forbidden sunflowers with a fixed and non-empty core and non-constant number of leaves.

We should say that both Boolean Analysis approaches mentioned above can be seen as exploiting the sharp threshold phenomenon. As we have seen, $\Delta$-system method does also exploits this, albeit to a much smaller extent. Zakharov and the author recently introduced a {\it spread approximation} method \cite{KuZa}, which is also using this philosophy.

Two other important ingredients is the simplification/peeling procedure, described in Section~\ref{secbase}, and the notion of {\it $\tau$-homogeneous families}.\footnote{It is an unfortunate clash of notation that $\tau$-homogeneous is used both here and in F\"uredi's theorem. We resolved it here by calling the latter a {\it homogeneous structure}, rather than a family.} The latter definition allowed to extend the method from $t$-intersection type problems to the forbidden $t$-intersection type problems, and also ultimately allowed to work with any sufficiently quasirandom families, instead of just, say, ${[n]\choose k}$. We say that a family $\ff\subset \m A$ is {\it $\tau$-homogeneous with respect to $\m A$}, if for any set $X$ we have $\frac {|\ff(X)|}{|\ff|}\le \tau^{|X|}\frac {|\aaa(X)|}{|\aaa|}$. If we use the counting measure notation, where $\mu(\ff) = |\ff|/|\aaa|$ and $\mu(\ff(X)) = |\ff(X)|/|\aaa(X)|$, then it transforms into $\mu(\ff(X))\le \tau^{|X|}\mu(\ff)$. We use this notion with $\tau$ very close to $1$, say, with $\tau = (1+\epsilon)^{1/t}$. Ultimately, this allows to control the behaviour of the restrictions on $\ff(X)$ at level $\le t$.

Using this method, one can solve general problems for $n> poly(s,t) k\log k$.

\section{Survey of future works}\label{sec8}
In this section, we summarize the developments in  the problems that were treated earlier in this survey. These include forbidden one intersection problem, $(n,k,L)$-systems, families with different forbidden configurations: sunflowers with fixed kernel, different simplices, expansions and contractions of families of small uniformity, hypertrees. We end with references to other surveys on closely related topics that we haven't covered.

\subsection{Forbidden one intersection problem}\label{sec61}
The following question was asked by Erd\H os \cite{Erd75}, generalizing the question of Erd\H os and S\' os for $\ell=1$: what is the largest family $\ff\subset {[n]\choose k}$ such that no two sets  $A,B\in\ff$ satisfy $|A\cap B| = \ell$? These families are called {\it $\ell$-avoiding}. We treated this question in Sections~\ref{sec241} and~\ref{sec33}. Some of the key results here are, of course, Theorem~\ref{thmff1} and~\ref{thmff2}. In \cite{KMW}, Keevash, Mubayi and Wilson  found the exact size of the largest $1$-avoiding family for $k=4$ and all $n$.

In the aforementioned results, the extremal families are families of sets containing a given $(\ell+1)$-element set. For $k\le 2\ell+1$, as we have seen in Section~\ref{sec331}, extremal examples are expected to be of Steiner system-type. In \cite{Fra83b}, using linear-algebraic approach, Frankl showed that for $k\le 2t\ell+1$ and $k-t$ prime, the largest $\ell$-avoiding family has size as suggested by Conjecture~\ref{conjff} (without the $(1+o(1))$), moreover, equality is only possible when $\ff$ is a Steiner system.

This problem is very interesting in other regimes, when  $k$ and $\ell$ are comparable to $n$. In  \cite{FW}, Frankl and Wilson gave a very strong upper bound ${n\choose k-\ell-1}$ on the size of $\ell$-avoiding families that works for essentially all interesting regimes of $n,\ell,k$ with $k> 2\ell$, but with a restriction that $k-\ell$ is a prime power. Their bound with $k,\ell$ being linear in $n$ was used to solve several discrete-geometric questions. Frankl and R\"odl \cite{FR}, motivated by further geometric applications, gave a worse exponential upper bound, but without number-theoretic restrictions. Recently, Kupavskii, Sagdeev and Zakharov \cite{KSZ} found an efficient way on how to deduce Frankl and R\"odl's result from the result of Frankl and Wilson.

Recently, exact results were proved in other regimes of parameters. Unless mentioned otherwise, the extremal family is the family of all sets containing a given $(\ell+1)$-element set. In \cite{KL}, Keller and Lifshitz solved the problem for $k>k(\ell)$ and $n>C(\ell)k$. In \cite{EKL}, Ellis, Keller and Lifshitz managed to solve the forbidden one intersection problem  in the regime  $k>k(\ell,\epsilon), \epsilon n<k<(1/2-\epsilon)n$. In this regime, the extremal configuration is no longer the family of all sets containing a given $(\ell+1)$-element set, but {\it Frankl configurations}. These are extremal configurations in the famous result of Ahlswede and Khachatrian \cite{AK} on $\ell$-intersecting families: $\{F\in {[n]\choose k}: |F\cap [\ell+2i]|\ge \ell+i\}$.
In the paper \cite{KuZa}, Zakharov and the author solved the problem in the regime when all parameters are polynomially dependent: in particular, for $n =k^\alpha$ and $t=k^\beta$ with $0<\beta<1/2$ and $\alpha>1+2\beta$, provided $k>k_0$.
In a recent paper \cite{Cher}, Cherkashin found an exact answer for the $\ell=1$ case when $n = k^2-k+1$ and all $k$.

We should also note that the size of the largest $\ell$-avoiding system $\ff\subset{[n]\choose k}$ is the independence number of the generalized Johnson graph $J(n,k,\ell)$ with vertex set ${[n]\choose k}$ and sets connected by an edge if they have intersection exactly $\ell$. Zakharov \cite{Zak} obtained results concerning the chromatic number of Johnson graphs. Many other properties, including supersaturation for edges and cliques, transference etc. were studied in recent papers by Raigorodskii and his students, see, e.g., \cite{DNRS}.

\subsection{$(n,k,L)$-systems}
 Let $m(n,k,L)$ stand for the size of the largest family $\ff\subset {[n]\choose k}$ such that $|F_1\cap F_2|\in L$ for any distinct $F_1,F_2\in \ff$.
The problem of determining $m(n, k, L)$ for general $L$ was proposed by Deza, Erd\H os, and Frankl \cite{DEF} and independently by Larman \cite{Lar} in 1978. It was addressed for $|L|=2$ by Deza, Erd\H os and Singhi \cite{DES} and then by Deza, Erd\H os, and Frankl \cite{DEF} for general $L$. We discussed their main result, Theorem~\ref{thmdef}, and its precursors in Section~\ref{sec2}. In 1975, Ray-Chaudhuri and Wilson \cite{RW} in the context of studying designs obtained a generalization of Ficher's inequality, which can be used to get the following general bound: $m(n,k,L)\le {n\choose s}$ for all $n\ge k$ and $|L|= s$. The proof uses linear algebra (and could be done either using  higher order incidence matrices or polynomials).

In the paper \cite{Fra86}, Frankl showed that the problem of determining the order of magnitude of the largest $(n,k,L)$-system is very rich. Namely, he showed that any rational number $\alpha$ can appear as the exponent: for any $\alpha$ there exists $k$ and a set $L$ such that $m(n,k,L) = \Theta(n^{\alpha})$. He used Delta-systems to show the upper bound. Continuing the work of Frankl \cite{Fra80b}, Frankl, Ota and Tokushige \cite{FOT96} determined  exponents of largest $(n,k,L)$-systems for all $k\le 12$ except two cases. For upper bounds, they used Delta-systems together with different reductions, such as Theorem~\ref{thmfur2} and reductions akin to the ones in Section~\ref{sec241}. Lower bounds use different algebraic and geometric constructions. This work was completed by Tokushige \cite{Toku}, who resolved the remaining two cases by providing the lower bound constructions.

We have discussed several reductions between $m(n,k,L)$ with different parameters throughout this survey. As we have seen in the proof of Theorem~\ref{thmdef}, Deza, Erd\H os and Frankl \cite{DEF} showed that, for fixed $k$, $$m(n,k,\{\ell_1,\ldots, \ell_r\})\sim m(n,k-\ell_1,\{\ell_2-\ell_1,\ldots, \ell_r-\ell_1\}).$$ Another reduction is given in Theorem~\ref{thmfur2}. Also, Theorem~\ref{thmfur1} characterizes the case when $m(n,k,L) =O(n)$. Frankl \cite{Fra86} showed that $m(n,k,L)\le m(f(n),f(k),f(L))$, where $f$ is a polynomial $f(x) = \sum_{i=1}^d a_i{x\choose i}$ with non-negative integer coefficients, and that in some cases the exponent for the two functions is the same.

Frankl and Wilson \cite{FW} proved the following modular analogue of the result of Ray-Chaudhuri and Wilson:
\begin{thm}
  Fix $L\subset [0,k-1]$. If  there is an integer-valued polynomial  $f$ of degree $r$ and a prime number $p$ such that $p|f(\ell)$ for every $\ell\in L$, and $p\not | f(k)$, then $m(n,k,L)\le {n\choose r}$.
\end{thm}

The proof uses linear algebra method. F\"uredi in \cite{Fur91} gives a very nice argument that allows us to compare the  $\Delta$-system method with the linear-algebraic method. We recite it here.

\begin{thm}
   Fix $L\subset [0,k-1]$. If $\ff$ is an $(n,k,L)$-system, and there is an integer-valued polynomial  $f$ of degree $r$ and a positive integer  $q$ such that $q|f(\ell)$ for every $\ell\in L$, and $q\not | f(k)$, then $|\ff|\le C_k{n\choose r}$, where $C_k$ depends on $k$ only.
\end{thm}
Note that the bound is weaker (working for large $n$ only), but we drop the assumption that $q$ is prime.
\begin{proof}
  Pass from $\ff$ to a homogeneous structure $\ff^*$. Let us show that the rank of $\ff^*$ is at most $r$. This will imply $|\ff^*|\le {n\choose r}$, and thus the bound in the theorem.

  Arguing indirectly, assume that the intersection structure $\m I(\ff)$ covers all $r$-element subsets. Apply inclusion-exclusion to each individual $i$-element set in $[k]$ and sum it up over all sets, getting
  $${k\choose i} = \sum_{I\in \m I}{|I|\choose i}-\sum_{I\in \m I}\sum_{I'\in \m I, I'\ne I}{|I\cap I'|\choose i}+\ldots.$$
  As $f$ is an integer-valued polynomial, we can write it in the form $f = \sum_{i=0}^r a_i {x\choose i}$ with integer coefficients $a_i$. Multiply the display by $a_i$ and sum it up over $i$. On the left hand side, we get $f(k)$, which is non-zero modulo $q$. On the right hand side, we get $\sum f(|I|)-\sum \sum f(|I\cap I'|)+\ldots$, which is $0$ modulo $q$, a contradiction.
\end{proof}

F\"uredi \cite{Fur91} conjectured that if there is a family $\m I\subset 2^{[k]}$ of sets of size $\ell\in L$ that is closed under intersection and such that its rank is at least $d+1$, then for some positive $\epsilon = \epsilon(k)$ we have $m(n,k,L) = \Omega(n^{d+\epsilon})$. It is true for $d=1$ via a construction of Frankl \cite{Fra83}. R\"odl and Tengan \cite{RT} verified this conjecture for $L =\{0,1,\ldots, d-1, q\}$ and a large class of families $\m I$, that, in particular, including complete designs (Steiner systems): each $q$-element set of $[k]$ is contained in exactly one $p$-element set. In this setup, $\epsilon>0$ from the conjecture is essential: if we ask for $\Omega(n^{d})$ only, then we may simply ignore intersection $q$ and consider a (partial) Steiner system $S(n,k,d)$ in which each $d$-element set is contained in (at most) one $k$-set. However, I do not know any non-trivial lower bounds for the conjecture of F\"uredi for general $L$.

\subsection{Sunflowers, unavoidable hypergraphs}
Let us return to the question, discussed in Section~\ref{sec41}. In 1977, Duke and Erd\H os \cite{DE} asked the following variant of the question of Erd\H os and Rado: what is the largest size $f(n,k,\ell,s)$ of the family $\ff\subset {[n]\choose k}$ that does not contain a $\Delta(s)$-system with kernel of size $\ell$?\footnote{The authors of different papers used different permutations for the $4$ parameters of the function. The size $n$ of the ground set is always the first, while the other indices got at least three different permutations in \cite{DE}, \cite{Fra78b}, and \cite{FF2}. Ours is the one used in \cite{FF2}.} For $s=2$ this is the forbidden one intersection problem, discussed in Section~\ref{sec61}, and for $\ell=0$ this is the famous Erd\H os Matching Conjecture \cite{E} (see \cite{FK21} for a survey). In Section~\ref{sec41} we mentioned that $f(n,k,\ell,s)=O_{k,s} (f(n,k,\ell,2))$ follows from Theorem~\ref{thmfur} of F\"uredi \cite{Fu1}, and that Frankl and F\"uredi determined the asymptotic of $ex(n, \m S^{(k)}_\ell(s))$ for fixed $s,k$ and $k\ge 2\ell+3$, see Theorem~\ref{thmff3}. The case of $k\le 2\ell+2$ resisted progress so far.

What happens if $k,s$ are not fixed? Duke and Erd\H os \cite{DE} in their paper showed that $f(n,k,1,s)\le O_k(s^2 n^{k-2}).$ It is the right dependence on $s$ and $n$ since we can use the same type of construction as in Section~\ref{sec411}, that is an `expansion' of a construction of a $2$-graph on $(s-1)^2$ vertices with no sunflower (a complete bipartite graph with parts of size $(s-1)$). Using weight functions, Chung and Frankl \cite{CF} determined the value of $f(n,3,1,s)$ exactly for $n>Ck^3$. It improved upon an earlier work of Frankl \cite{Fra78b}, and Chung \cite{Ch}, who determined the answer up to lower-order terms.

In cite \cite{DE}, Duke and Erd\H os noted that $f(n,3,2,s) \sim \frac 16 s n^2$ for constant $s$ (and which can be easily extended to $s=o(n)$). Note that we are asking for the largest hypergraph such that no pair of elements is contained in $s$ triples. The upper bound is a simple double-counting, and the lower bound follows by using a connection to Steiner triple systems (e.g., we may find a partial triple system using R\"odl's result, and then combine $s$ randomly permuted copies of those). In particular, we get $f(n,3,2,s) = \Theta (sn^2)$ for all $s$. Much more recently, Buci\'c, Dragani\'c, Sudakov and Tran \cite{BDST} showed that $f(n,4,2,s) = \Theta (s^2n^2)$ and $f(n,4, 3,s) =\Theta(s n^3)$. Their result was generalized by Brada\v c, Buci\'c, and Sudakov \cite{BBS}, who showed that $f(n,k,\ell,s) = \Theta_k(s^{\ell+1}n^{k-\ell-1})$ for $k\ge 2\ell+2$ and $f(n,k,\ell,s) = \Theta_k(s^{k-\ell}n^{\ell})$ for $k\le 2\ell+1$. They mostly followed the methods of Frankl and F\"uredi \cite{FF1}. One difficulty that they managed to overcome is that Theorem~\ref{thmfur} cannot be applied directly here, since $s$ is assumed to be large. They overcome it by constructing an intersection structure-like family, as it would look after the application of Theorem~\ref{thmfur}, without actually going through the same `filtering' process.

In a recent paper \cite{KN}, Noskov and the author managed to find the correct asymptotic of $f(n,k,\ell,s)$ for $k$ linear in $n$:  $n>poly(s,\ell) k$, $n\ge n_0(s,t)$ and $k\ge 2\ell+3$. (We note here that $n\ge n_0(s,\ell)$ with at least exponential $n$ is unavoidable here because of our poor understanding of the original Erd\H os--Rado sunflower problem, cf. examples from Section~\ref{sec411}.)

One motivation for studying $f(n,k,\ell,s)$ for $s$ that grows with $n$ is {\it unavoidable hypergraphs}. A $k$-uniform hypergraph $\m H$ is {\it $(n,e)$-unavoidable} if every $k$-uniform hypergraph with $e$ edges contains a copy of $\m H$. Let us denote $un_k(n,e)$ the maximum number of edges in an $(n,e)$-unavoidable hypergraph. Chung and Erd\H os \cite{CE} determined $un_3(n,e)$ up to a multiplicative constant for the full range of $e$. The case $k=4$ was resolved (again, up to a multiplicative constant) by Buci\'c, Dragani\'c, Sudakov, and Tran \cite{BDST}.

\subsection{Simplices, clusters}
As we have discussed in Section~\ref{sec43}, Erd\H os \cite{E4} conjectured that the largest family in ${[n]\choose k}$ without a triangle has size ${n-1\choose k-1}$ and is a star, and the question was generalized to $d$-dimensional simplices by Chv\'atal \cite{Chv} (see Definition~\ref{defsim} and Conjecture~\ref{conjchv}): if $\ff$ contains no $d$-simplex, $k>d$, and $n\ge (d+1)k/d$, then $|\ff|\le {n-1\choose k-1}$. Frankl proved it for $n\le dk/(d-1)$ \cite{Fra76}, as well as for $d=2,k\ge 5$ and $k\ge 3d+1$, provided $n\ge n_0(k)$ \cite{Fra81}. Then Frankl and F\"uredi proved it for any $k>d$, provided $n\ge n_0(k)$. The case $d=2$ (i.e., the original question of Erd\H os) was settled for all $n,k$ by Mubayi and Verstra\"ete \cite{MV}. Keevash and Mubayi \cite{KM} proved it for any $\zeta>0$ and $k_0(d)\le k\le (\frac 12-\zeta)n$, provided $n>n_0(d,\zeta).$ Keller and Lifshitz \cite{KL} proved it for all $k>k_0(d)$. This, combined with the result of Frankl and F\"uredi, covered all cases $n>n_0(d)$. Shortly after, Currier \cite{Curr2} made impressive progress on the problem and proved it for almost all values of parameters: for $d\ge 3$ and $n\ge 2k-d+2$. Together with the previous results, the following cases of the conjecture are left:  $k\ge 3, dk/(d-1)<n< 2k-d+2$, for $d$ that grow with $k$.

 Actually, there are two very different types of simplices that the aforementioned works dealt with. In Section~\ref{sec43}, we have discussed the very structured $d$-simplices that are called {\it special simplices} (see Definition~\ref{defssim}). Theorem~\ref{thmff6} by Frankl and F\"uredi \cite{FF2} shows that for $k\ge d+3$ and $n\ge n_0(k)$ we have $ssim(n,k,d)\le {n-1\choose k-1}$, i.e., it is enough to forbid only special simplices in order to get the same answer. They conjectured that the same should hold for $k=d+1,d+2$, and proved this for $d=2$. They also conjectured that for $n\ge 2k$ $ssim(n,k,k-1)\le {n-1\choose k-1}$. Some cases of it are resolved by Bermond and Frankl in \cite{BF}. Keller and Lifshitz \cite{KL} showed the validity of the conjecture for $k\ge k_0(d)$ and $n\ge C_0(d) k$. Using spread approximation method, Noskov and the author (in preparation) verified the conjecture for $n\ge poly(d,k)$.

The second type of simplices that was considered are much less structured and are called {\it simplex-clusters}. Let us first define clusters. A {\it $d$-cluster} is a collection of $d$ distinct sets in ${[n]\choose k}$ with empty common intersection and with union of size at most $2k$. A $2$-cluster is just two disjoint sets, and thus $2$-cluster-free families are the same as intersecting families. As another possible generalization of the Erd\H os--Ko--Rado theorem, Katona conjectured that the largest $3$-cluster-free family in ${[n]\choose k}$ has size ${n-1\choose k-1}$ (see \cite{FF0}). Frankl and F\"uredi \cite{FF0} proved Katona's conjecture for $n\ge k^2+3k$.
Mubayi \cite{Mu} proved the conjecture for  $3$-cluster-free families in ${[n]\choose k}$ has size at most ${n-1\choose k-1}$ for all $n,k\ge 3$, and conjectured that the same should hold for any $d\ge 4$. In a follow-up paper, he proved the asymptotic upper bound $(1+o(1)){n-1\choose k-1},$ as well as resolved the case $k=4$ and $n$ sufficiently large. The aforementioned result of Keevash and Mubayi \cite{KM} applies to $(d+1)$-cluster-free sets as well. Then Mubayi and Ramadurai~\cite{MR}, and independently F\"uredi and \"Ozkahya~\cite{FO} showed that the conjecture holds for $k>4$ and $n$ large enough. The case $n<2k$ (where the union condition in the definition of a cluster is automatically fulfilled) was resolved by Frankl \cite{Fra76}. Finally, Currier \cite{Curr1} made impressive progress and proved Mubayi's conjecture for the whole range of parameters: $d\le k$ and $n\ge dk/(d-1)$.\footnote{His work on Chv\'atal's conjecture builds upon this work.} The latter condition is needed so that there are $d$ sets in $[n]$ with empty common intersection.

For $d= k+1$, F\"uredi and \"Ozkahya \cite{FO} gave an elegant argument showing that a $(k+1)$-cluster-free family $\ff\subset{[n]\choose k}$ must have size at most ${n-1\choose k-1}$. First, note that for any $F\in \ff$ there must be a subset $X(F)$ of size $k-1$ that is not contained in any other set $F'\in \ff$. Otherwise, if $F = \{a_1,\ldots, a_k\}$ and $F_i\cap F = F\setminus \{a_i\}$, then $F\cap F_1\cap\ldots\cap F_k = \emptyset$ and $|F\cup F_1\cup \ldots\cup F_k|\le 2k$. Thus, these sets form a $(k+1)$-cluster. Thus the pairs  $(A_F,B_F):=(X(F),[n]\setminus F)$, $F\in \ff$, satisfy the property $A_F\cap B_{F'} = \emptyset$ iff $F = F'$. Applying the Bollob\'as set-pair inequality \cite{Bol} to these pairs, we conclude that $|\ff|\le {|A_F|+|B_F|\choose A_F} = {n-1\choose k-1}$.  We should also note that if in the definition of a $d$-cluster we change $2k$ to $2k-1$, then a complete $k$-partite hypergraph has no such substructure. Frankl and F\"uredi \cite{FF0} conjectured that for $d=3$ the extremum is given by a complete $k$-partite hypergraph and proved it for $k=3$, $n>n_0$. A related notion for a $3$-uniform hypergraph is being {\it cancellative}: not having three sets $A,B,C$ such that $A\Delta B\subset C$.

Merging the two definitions, that of a $d$-simplex and of a $(d+1)$-cluster, we get a $d$-simplex-cluster. Keevash and Mubayi made, quoting, `ambitious conjecture' that for $k\ge d+1>2, n>k(d+1)/d$ and a family in ${[n]\choose k}$ avoiding a $d$-simplex-cluster, its size is at most ${n-1\choose k-1}$, with equality only possible for stars. Lifshitz \cite{Lif} proved it for $n\ge n_0(d)$. Shortly after, Currier proved it in \cite{Curr2} for $n\ge 2k-d+2$ (this is actually the main result of \cite{Curr2}, which immediately implies the result for  Chv\'atal's conjecture).

Let us now discuss the variants of these questions that arose in the literature while studying simplices and clusters. A triangle, and, more generally, a $d$-simplex are the simplest examples of non-trivial intersecting families, that is, intersecting families with covering number $2$. In \cite{MV}, Mubayi and Verstra\"ete actually proved that if $\ff\subset {[n]\choose k}$ does not contain a non-trivial intersecting family on $d+1$ sets, $n\ge(d+1)k/d$ and $k\ge d+1$, then $|\ff|\le {n-1\choose k-1}$, with equality for stars only. They conjectured that the same should hold without the assumption $k\ge d+1$, provided $n$ is large enough. This was proved by Liu \cite{Liu}, who actually showed that is enough to forbid a certain configuration, which is called a $(\vec a,p)$-$\Delta$-system. If the dimension of $\vec a$ is $d$, then {\it $(\vec a,p)$-$\Delta$-system} is any family of the following type: start with a partition of a host set $F_0$ into $d$ parts $A_1,\ldots, A_d$. Let $B_i:= F_0\setminus A_i$. Now extend each $B_i$ into  $b_i>0$ pairwise disjoint $k$-sets, so that, first, all extensions of all $B_i$'s are pairwise disjoint and, second, $\sum b_i = p$. This construction actually generalized the construction of F\"uredi and \"Ozkahya \cite{FO}, who used it to prove the $d$-cluster conjecture of Mubayi for sufficiently large $n$. They used $b_1 =\ldots = b_d = 1$, in which case each $B_i$ has a unique extension $B'_i$. It is not difficult to see in this case that the union of all sets has size exactly $2k$, as well as they do not have a common intersection, so they indeed form a $d$-cluster. The proof of F\"uredi and \"Ozkahya, as well as that of Liu, involved Delta-systems.

Keevash and Mubayi \cite{KM} actually proved their result for families avoiding a $d$-cluster and for families avoiding a {\it strong $d$-simplex}. It is a collection of $d+2$ sets $A,A_1,\ldots, A_{d+1}$, such that $A_1,\ldots, A_{d+1}$ is a $d$-simplex, and $A$ intersects $\cap_{j\ne i}A_i$ for each $i\in[d+1]$. E.g., a strong $1$-simplex is a path of length $3$, i.e., three sets $A_1,A,A_2$, such that $A_1\cap A_2 =\emptyset$ and $A_1\cap A, A\cap A_2\ne \emptyset$.

\subsection{Hypertrees, expanded hypergraphs} In Section~\ref{sec42} we discussed a general result of Frankl and F\"uredi on expanded hypergraphs, as well as introduced an important notion of a cross-cut. Let us give a variant of the definition of an expanded hypergraph.


\begin{defn}[$t$-contractible hypergraphs]\label{defcont} For a positive integer $t \le k-1$,  a $k$-uniform family $\m H$ is {\em $t$-contractible} if each set of $\m H$ contains $t$ vertices of degree $1$. A $t$-contraction of $\m H$ is the $(k-t)$-uniform multi-family obtained by deleting $t$ degree $1$ vertices from each edge of $\m H$.
\end{defn}
The notion of $t$-contractible hypergraphs is slightly is more general than that of the $k$-expanded hypergraphs, see Definition~\ref{defexp}. The (only) difference is that the former allows edges of multiplicity more than $1$ in the contraction.

The following general result was proved by F\"uredi and Jiang \cite{FJ15} using $\Delta$-systems.
\begin{thm}[Jiang and F\"uredi, \cite{FJ15}]\label{thmjf}
  Let $k\ge 4$ be an integer. For any $k$-graph $\m H$ that is a subgraph of a $2$-contractible $k$-tree, we have
  $$ex(n,\m H) = (\sigma_1(\m H)+o(1)){n\choose k-1}.$$
\end{thm}
(We remind the reader that $\sigma_1(\m H)$ is the size of the $1$-crosscut minus $1$.) It was shown in \cite{FJ15}  that the conclusion of Theorem~\ref{thmjf} is no longer valid for all $1$-contractible trees.

This theorem extended several previous results on the subject, which we discuss below. First, we should note that for $k$-expansions of graph trees with at most  $k-1$ edges Theorem~\ref{thmff4} gives the same bound. Then,  F\"uredi \cite{Fu14} in 2014 proved this for $k$-expansions of graph trees, also under the condition of $k\ge 4$. (These are $(k-2)$-contractible $k$-trees, if we ignore the multi-edge situation.)
\begin{defn}[Linear paths and cycles]\label{deflinear}
 A $k$-uniform family is a {\em linear path $\m P_\ell^{(k)}$} of length $\ell$ ({\em linear cycle} $\m C_\ell^{(k)}$ of length $\ell$) if it is a $k$-expansion of a graph path of vertex length $\ell$ (cycle of length $\ell$).
\end{defn}
That is, F\" uredi, in particular, determined the asympototic for $ex(n,\m T)$ for linear paths. The exact value for linear paths of length $\ell$ for $k\ge 4$ and $n\ge n_0(k,\ell)$ was determined by F\"uredi, Jiang and Siever \cite{FJS}. Both papers used the Delta-system method. The case $k=3$ was settled by Kostochka, Mubayi and Verstra\"ete \cite{KMV1} using the `random sampling from the shadow' method, together with some Ramsey-type arguments. In the follow-up paper \cite{KMV2}, they extended the result of F\"uredi \cite{Fu14} concerning expansions of graph trees to the case of $k=3$. The case of forests consisting of linear paths was resolved for large $n$ in \cite{BK} using induction and the aforementioned results.

Another type of expansions that are covered by Theorem~\ref{thmjf} are linear cycles. Let us show that for $k\ge 5$ we can construct a $2$-contractible $k$-tree that contains a linear cycle of length $\ell$. Take the graph cycle $\{v_1,\ldots,v_\ell\}$ of length $\ell$. Now we add $3$-edges that will turn this graph cycle a tight $3$-tree: add $\{v_1,v_2,v_3\}, \{v_1,v_3,v_4\},\ldots, \{v_1,v_{\ell-1},v_\ell\}$ in this order. Then add extra vertices
 $u_1,\ldots, u_\ell$ so that $\{v_i,u_i,v_{i+1}\}_{i\in[\ell]}$ is the set of hyperedges of $\m C^{(3)}_\ell$. (We can add them in any order.) It is not difficult to check that the resulting $3$-uniform hypergraph is a tight $3$-tree. Now we may add $k-3$ distinct vertices to each hyperedge and  get a $(k-3)$-contractible $k$-tree that contains $\m C^{(k)}_\ell$.

 Thus, Theorem~\ref{thmjf} implies the asymptotic behavior of $ex(n,C^{(k)}_\ell)$ for $k\ge 5$. The exact value of  $ex(n,C^{(k)}_\ell)$ for large $n$ was actually determined by F\"uredi and Jiang for $k\ge 5$ in an earlier paper \cite{FJ13}. They used Delta-systems. Then the remaining cases $k=3,4$ were resolved by Kostochka, Mubayi and Verstra\"ete \cite{KMV1} using the `random sampling from the shadow' method.

We should note that F\"uredi and Jiang \cite{FJ15} actually proved sharper bounds. For $2$-reducible $k$-trees the error term in Theorem~\ref{thmjf} is $O(n^{k-2})$.  For $(k-2)$-reducible $k$-trees $\m T$ (i.e., essentially $k$-expansions of graph trees),  the upper bound on $ex(n,\m T)$ for large $n$ depends on the value of $f(n,k,1,s)$ (extremal function for forbidden sunflower with core of size $1$), and is sharp in some cases (cf. \cite[Theorem 4.5]{FJ15}).

$3$-expansions of different $2$-graphs were treated by Kostochka, Mubayi and Verstra\"ete in \cite{KMV2, KMV3}. In particular, in \cite{KMV2} the authors showed that for large $n$ $ex(n,\m H)\le {n\choose 2}$ for any $3$-uniform hypergraph that is an expansion of a $2$-graph with cross-cut $2$, and in \cite{KMV3} they showed that $ex(n,\m H) = O(n^2)$ for graphs with cross-cut $3$. They also showed that $ex(n,\m H)=\Theta(n^2)$ for the $3$-expansion of the graph of the $3$-dimensional cube. In \cite{LSY} the authors sharpened some of the results of \cite{KMV2}. We refer to \cite{MV} for a survey on the topic of extremal problems for graph expansions.

In the paper \cite{FJKMV} the authors determined the asymptotic of $ex(n, \m C)$, where $\m C$ is a $(a+b)$-uniform path, where consecutive edges alternately share $a$ and $b$ vertices. They developed an averaging technique which they called an asymmetric version of Katona's circle method.

\subsection{Other surveys}
In this short review, we covered only a small portion of the field. We refer to several other excellent surveys in the field. A survey by  F\"uredi \cite{Fur91} has a significant overlap with this survey; Mubayi and Verstra\"ete \cite{MV} treat Tur\'an problems for expansions; intersection theorems are treated in a survey by Ellis \cite{Ell}; methods for solving (mostly dense) hypergraph Turan problems are surveyed by Keevash in \cite{Kee2}. In this survey, we did not treat Berge hypergraphs. See, e.g., a recent paper \cite{FL} for a brief overview and an interesting result. Applications of sunflowers and the spread lemma are discussed in an expository paper by Rao \cite{Rao2}. There are also two books on extremal set theory: by Frankl and Tokushige \cite{FT} and by Gerbner and Patk\'os \cite{GP}.

\section{Acknowledgements} I am indebted to  Liza Iarovikova, Fedor Noskov, Nikolay Terekhov, Georgii Sokolov and Yakov Shubin for proofreading earlier versions of the manuscript and providing numerous suggestions that improved the presentation, as well as to Peter Frankl, Tao Jiang and the anonymous referee for giving multiple helpful comments on the final version of the text. The research is supported by the Ministry of Science and Higher Education of the Russian Federation, project No. FSMG-2024-0011

\end{document}